\documentclass[11pt, a4paper, leqno]{amsart}
\makeatletter
\def\l@subsection{\@tocline{2}{0pt}{1pc}{4.6em}{}}
\renewcommand{\tocsubsection}[3]{\indentlabel{\@ifnotempty{#2}{\hspace*{2em}\makebox[2em][l]{\ignorespaces#1 #2. \hfill}}}#3}
\makeatother

\usepackage{esint}
\usepackage{color}

\usepackage{amsthm}
\usepackage{amsmath}
\usepackage{amsfonts}
\usepackage{amssymb}
\usepackage{amscd}
\usepackage{mathrsfs}
\usepackage{enumerate}
\usepackage{enumitem}

\usepackage{hyperref}

\usepackage[utf8]{inputenc}
\usepackage[english]{babel}
\usepackage{anysize}
\usepackage{bm}
\usepackage{wasysym} %zodiaco
\usepackage{mathrsfs} %mathscr
\usepackage{bbm} %mathbbm
\usepackage{calligra} %calligra
\usepackage{enumerate} %(i)

\numberwithin{equation}{section}

\newcommand{\R}{\mathbb{R}}
\newcommand{\Op}{\operatorname{Op}}

\newcommand{\Id}{\operatorname{Id}}

\newcommand{\Z}{\mathbb{Z}}
\newcommand{\T}{\mathbb{T}}

\renewcommand{\Re}{\operatorname{Re}}
\renewcommand{\Im}{\operatorname{Im}}
\newcommand{\tr}{\mathrm{tr}}

\newtheorem{prop}{Proposition}
\numberwithin{prop}{section}

\newtheorem{lemma}{Lemma}
\numberwithin{lemma}{section}

\newtheorem{definition}{Definition}
\numberwithin{definition}{section}

\newtheorem{corol}{Corollary}
\numberwithin{corol}{section}

\newtheorem{teor}{Theorem}
\numberwithin{teor}{section}

\theoremstyle{remark}
\newtheorem{example}{Example}
\numberwithin{example}{section}

\theoremstyle{remark}
\newtheorem{remark}{Remark}
\numberwithin{remark}{section}

\usepackage{amsaddr}

\author[Víctor Arnaiz]{Víctor Arnaiz}
\title[Construction of quasimodes for non-selfadjoint operators]{Construction of quasimodes for non-selfadjoint operators via propagation of Hagedorn wave-packets}

\address{\small Université Paris-Saclay, CNRS, Laboratoire de mathématiques d’Orsay, 91405, Orsay, France.}

\email{victor.arnaiz@universite-paris-saclay.fr}

\begin{document}

\begin{abstract}
We construct quasimodes for some non-selfadjoint semiclassical operators { at the boundary of the pseudo-spectrum} using propagation of Hagedorn wave-packets. Assuming that the imaginary part of the principal symbol of the operator is non-negative and vanishes on certain points of the phase-space satisfying a { subelliptic} finite-type condition, we construct quasimodes that concentrate on these \textit{non-damped} points. More generally, we apply this technique to construct quasimodes for non-selfadjoint semiclassical perturbations of the harmonic oscillator that concentrate on non-damped periodic orbits or invariant tori satisfying a weak-geometric-control condition.
\end{abstract}

\maketitle

\tableofcontents

\section{Introduction and main results}

\subsection{Motivation} The study of the asymptotic behavior of wave or quantum propagation from the knowledge of the underlying classical dynamical system is the main objective of Semiclassical Analysis. The description of semiclassical asymptotics is intimately connected with the spectral distribution of large energy quantum states and play a central role in the description of long-time behavior of quantum waves. 

Selfadjoint operators on Hilbert spaces are fundamental in the description of quantum mechanics, since they modelize quantum observables. The spectral theory for selfadjoint operators provides precise bounds on the resolvent, together with very good control on functions of such operators solely in terms of the spectrum, which yields a good description of propagation phenomena for quantum models involving these operators. On the other hand, non-selfadjoint operators appear in several mathematical-physics problems such as convection-diffusion problems, Kramers-Fokker-Planck equations,
damped wave equations, scattering poles or linearized operators in fluid dynamics (see  \cite{Sjostrand_survey} for an introductory survey on spectral properties of non-selfadjoint operators). Contrary to the selfadjoint case, the resolvent of a non-selfadjoint operator may be very large even at points of the complex plane which are far from the spectrum. This introduces difficulties in the study of wave-propagation subject to non-selfadjoint operators, crystallizing in the notion of \textit{pseudo-spectrum}, given by the set of the points of the complex plane where the resolvent is ``asymptotically large" containing the spectrum.

In this work we focus on the study of the pseudo-spectrum for certain non-selfadjoint semiclassical operators and give new constructions of quasimodes for such systems. However, to motivate our results and techniques, we first introduce some ideas coming from previous studies on selfadjoint operators, particularly from the construction of eigenfunctions and quasimodes concentrating on low dimensional submanifolds of the phase-space. 

Let us recall some of these ideas for the well understood example of the quantum harmonic oscillator (see \cite{Ar_Mac18} for a complete description of quantum limits and semiclassical measures for sequences of eigenfunctions of the harmonic oscillator and \cite{DBievre92,DBievre93,Ojeda_Villegas10} for previous works in the study of phase-space concentration of quantum states of harmonic oscillators).  The large energy distribution of the quantum states of the harmonic oscillaor matches with the quasiperiodic structure of the classical Hamiltonian flow. Particularly, one can find sequences of eigenfunctions for the quantum harmonic oscillator that concentrate, in phase-space sense, on periodic orbits or minimal invariant tori by the classical Hamiltonian flow.  To fix ideas, let us consider the stationary Schrödinger equation associated with the quantum harmonic oscillator:
\begin{equation}
\label{e:Schrodinger_equation}
\widehat{H}_\hbar \Psi_\hbar = \lambda_\hbar \Psi_\hbar(x), \quad \Vert \Psi_\hbar \Vert_{ L^2(\R^d)} = 1,
\end{equation}
for the Hamiltonian $\widehat{H}_\hbar$ given by
\begin{equation}
\label{e:quantum_harmonic_oscillator}
\widehat{H}_\hbar := \frac{1}{2}\sum_{j=1}^d \omega_j \big( - \hbar^2 \partial_{x_j}^2 + x_j^2\big) , \quad \omega = (\omega_1, \ldots, \omega_d) \in \R^d_+,
\end{equation} 
where $\omega$ is called vector of frequencies and $\hbar > 0$ is a small semiclassical parameter. The semiclassical asymptotics for this system can be described in connection with the classical dynamics on $T^*\R^d \simeq \R^{2d}$ of the Hamiltonian system generated by the harmonic oscillator
$$
H(x,\xi) := \frac{1}{2} \sum_{j=1}^d \omega_j(\xi_j^2 + x_j^2), \quad (x,\xi) \in T^*\R^d \simeq \R^{2d}.
$$
Precisely, the spectrum of $\widehat{H}_\hbar$ consists of a discrete sequence of eigenvalues $(\lambda_{\hbar,k})_{k \in \mathbb{N}_0^d}$ tending to infinity for any fixed $\hbar > 0$. This sequence is totally explicit, and given by
\begin{equation}
\label{e:eigenvalue_oscillator}
\lambda_{\hbar,k} = \hbar \sum_{j=1}^d \omega_j \left( k_j + \frac{1}{2} \right), \quad k = (k_1, \ldots, k_d) \in \mathbb{N}^d_0.
\end{equation}
The semiclassical limit arises when choosing sequences $(\hbar, k_\hbar)$ such that $\lambda_\hbar = \lambda_{\hbar,k_\hbar} \to E_0$ as $\hbar \to 0^+$. Here $E_0 \in H(\R^{2d})$ is the classical energy, and modulo adding a constant to $H$, it can be chosen as $E_0 = 1$. For these sequences, the eigenstates $\Psi_\hbar$ concentrate, in phase-space sense via the Wigner distribution, on the level set $H^{-1}(1) \subset T^*\R^d$. Furthermore, due to the particularly simple quasiperiodic structure of the foliation of $H^{-1}(1)$ by invariant tori for the flow $\phi_t^H$ generated by the classical harmonic oscillator $H$, one can more precisely find sequences $(\Psi_\hbar)$ of eigenstates for $\widehat{H}_\hbar$ that concentrate near any prescribed minimal invariant torus $\mathcal{T} \subset H^{-1}(1)$. This is shown in \cite[Lemma 1]{Ar_Mac18}.

The construction of such sequences can be carried out by different ways. A very versatile technique is based on the propagation of coherent states. Precisely, let
$$
\varphi_0^\hbar[z_0](x) = \frac{1}{(\pi \hbar)^{d/4}} e^{-\frac{1}{\hbar} \vert x - x_0 \vert^2} e^{\frac{i}{\hbar} \xi_0 \cdot \left( x - \frac{x_0}{2} \right)}, \quad z_0 = (x_0,\xi_0) \in H^{-1}(1),
$$
be a coherent state with center $z_0$. The semiclassical Wigner distribution $W_\hbar[\varphi_0^\hbar[z_0]](z)$ (see Definition \ref{d:wigner_distribution} below) associated to $\varphi_{0}^\hbar[z_0]$ converges, in the weak-$\star$ topology of distributions, to the measure $\delta_{z_0}$. On the other hand, the time evolution for this coherent state by the quantum flow of $\widehat{H}_\hbar$ is given by
$$
e^{-\frac{it}{\hbar} \widehat{H}_\hbar} \varphi_0^\hbar[z_0](x) = e^{-i \frac{t \vert \omega \vert_1}{2}} \varphi^\hbar_0[\phi_t^H(z_0)](x), \quad t \in \R,
$$
where notice that the shape of the coherent-state is conserved, modulo a change in the phase, while the center point $z_0$ is translated into $\phi_t^H(z_0)$ by the classic flow. In particular, the Wigner distribution $W_\hbar[e^{-\frac{it}{\hbar} \widehat{H}_\hbar} \varphi_{0}^\hbar[z_0]]$ converges weakly-$\star$ to $\delta_{\phi_t^H(z_0)}$.

Assuming for a moment that $\omega = (1, \ldots, 1)$, hence the flow $\phi_t^H$ is periodic with period $2\pi$, one can consider the sequence
\begin{equation}
\label{e:eigenstate}
\Psi_\hbar(x) = \frac{1}{(\pi \hbar)^{1/4}} \left( \frac{\vert dH(z_0) \vert}{2\pi } \right)^{1/2} \int_0^{2\pi}  e^{\frac{it \lambda_\hbar}{\hbar}} e^{-\frac{it}{\hbar} \widehat{H}_\hbar} \varphi_0^\hbar[z_0](x) dt ,
\end{equation}
which turns out to be asymptotically normalized in $L^2(\R^d)$ and to satisfy equation \eqref{e:Schrodinger_equation} for a suitable choice of the sequence $\lambda_\hbar \to 1$ as $\hbar \to 0^+$ (see \cite[Lemma 1]{Ar_Mac18}). Moreover,
$$
W_\hbar[\Psi_\hbar] \rightharpoonup^\star \delta_{\mathcal{T}_\omega(z_0)},
$$
where $\mathcal{T}_\omega(z_0)$ is the minimal torus issued from $z_0$ by $\phi_t^H$ (in this case, a periodic orbit). The apparently exotic normalizing constant appearing in \eqref{e:eigenstate} fits with the non-trivial stationary-phase concentration of
$$
\Vert \Psi_\hbar \Vert^2_{L^2(\R^d)} =  \frac{\vert dH(z_0) \vert}{2\pi \sqrt{\pi \hbar}} \int_0^{2\pi} \int_0^{2\pi} e^{\frac{i(t-t') \lambda_\hbar}{\hbar}} \big \langle e^{-\frac{i t}{\hbar} \widehat{H}_\hbar} \varphi_0^\hbar[z_0],   e^{-\frac{i t'}{\hbar} \widehat{H}_\hbar} \varphi_0^\hbar[z_0] \big \rangle_{L^2(\R^d)} dt dt'
$$
near the diagonal $t \sim t'$. 

This technique, based on the propagation of wave-packets by the quantum flow of the system, has been  succesfully used to construct \textit{quasimodes}, or approximate solutions to the stationary Schrödinger equation, for more general selfadjoint operators near periodic or quasiperiodic orbits: see for instance \cite{DBievre92,DBievre93} for some propagation and localization results on closed trajectories; see \cite{Paul93} for  some quasimode constructions near closed trajectories via the use of { Bargmann} space, and \cite{EswaNon17,Eswarathasan18} for constructions of quasimodes for selfadjoint operators using propagation of coherent states near hyperbolic submanifolds. These works are part of a systematic study of propagation of wave-packets in the selfadjoint framework, see \cite{Robert12,Hagedorn85,Hagedorn98,RobertBook} among others.

The main purpose of this work is to extend this techinque to deal with non-selfadjoint operators whose symbols satisfy suitable principal-type conditions (see \cite[Thm. 1.4]{Dencker04}) at certain points (or more generally low-dimensional submanifolds) of the phase-space and reaching points of the complex plane at the boundary of the pseudo-spectrum, where the \textit{Hörmander-bracket condition} fails (see \cite[Thm 1.2]{Dencker04} for construction of quasimodes under this condition).

As a particular application of our techniques, we will construct quasimodes for some non-selfadjoint perturbations of the quantum harmonic oscillator, although our approach is considerably flexible and can be used elsewhere. The propagation of wave-packets in the non-selfadjoint setting entails certain difficulties with respect to the selfadjoint case, essentially because the propagator is not unitary and one does not have apriori estimates on the propagator norm. This work pretends to clarify and overcome some  of these difficulties (see \cite{Sjostrand09} for some insights in this direction).

Recently, several authors have put their attention in the study of the propagation of Hagedorn wave packets for non-Hermitian quadratic operators  (see \cite{Schubert15,Schubert12,Schubert11, Schubert18, Starov18} for a series of works in this project). Precisely, in \cite{Schubert18}, the interest is placed on the initial value problem
\begin{equation}
\label{e:evolution_system}
\big( i \hbar \partial_t  + \Op_\hbar( { \mathfrak{q}_t}) \big) \psi_\hbar(t,x) = 0 \, , \quad \psi_\hbar(0,\cdot) = \psi_\hbar^0 \in L^2(\R^d),
\end{equation}
where $\Op_\hbar(\cdot)$ is  the { semiclassical} Weyl quantization and $ \mathfrak{q}_t$ is a complex quadratic form on $\R^{2d}$. The authors obtain explicit formulas for the propagation of Hagedorn wave-packets (and more generally excited coherent states) subject to this equation. Contrary to the selfadjoint case, the $L^2$-norm of the evolved states is no longer conserved, while the center of the wave-packet follows a trajectory in phase-space which is described in terms of a dynamical system coupling both the real and the imaginary parts of $\mathcal{H}_t$. 

We will first employ these propagation results to address the problem of constructing quasimodes for non-selfadjoint operators of the form
\begin{equation}
\label{e:damped_wave_operator}
\widehat{P}_\hbar := \widehat{V}_\hbar + i \widehat{A}_\hbar,
\end{equation} 
where $\widehat{V}_\hbar = \Op_\hbar(V)$, $\widehat{A}_\hbar = \Op_\hbar(A)$, and the symbols $V,A \in { S^k(\R^{2d};\R)}$ { are real valued} and belong to the standard classes of symbols with growth $\langle z \rangle^k$ at infinity:
\begin{equation}
S^k(\R^{2d};\R) := \{ a \in \mathcal{C}^\infty(\R^{2d};\R) \, : \, \vert \partial^\gamma a(z) \vert \leq C_\gamma \langle z \rangle^k, \quad \gamma \in \mathbb{N}^{2d}_0 \},
\end{equation} 
where $\langle z \rangle := (1 + \vert z \vert^2)^{1/2}$. Moreover, we assume that $A \geq 0$ and it has vanishing set $A^{-1}(0)$ satisfying the following control condition { at} a given point $z_0 \in A^{-1}(0)$:
\medskip

\textbf{(GC)}: There exists $t \in \R$ such that $A \circ \phi_t^V(z_0) > 0$, 
\medskip

\noindent where $\phi_t^V$ denotes the Hamiltonian flow generated by $V$. In particular, $\nabla A(z_0) = 0$ and $\nabla V(z_0) \neq 0$, that is, $V+iA$ satisfies a local \textit{principal-type} condition at $z_0$ (see \cite[Eq. (1.6)]{Dencker04}). The geometric control condition \textbf{(GC)} appears frequently in the study of damped waves (see \cite{Lebeau96,RauchTaylor,Sjostrand00}) as a necessary and sufficient condition to obtain exponential decay rates for the energy. Recently \cite{Keeler20} this condition has also been studied in the anisotropic damped wave equation. Moreover, we will consider the particular case in which the point $z_0$ satisfies the following stronger finite-type condition: 
\begin{equation}
\label{e:non_degenerate0}
{ \gamma_0 =  \gamma_0(A,V,z_0)} := \big \langle X_V(z_0), \partial^2 A(z_0) X_V(z_0) \big \rangle > 0,
\end{equation}
where $X_V$ denotes the Hamiltonian vector field of $V$. In particular, \eqref{e:non_degenerate0} implies \textbf{(GC)} for the point $z_0$. This hypothesis has been treated into a more general framework in \cite{Dencker04,Sjostrand09} to obtain resolvent estimates for points of the complex plane near the value $V(z_0) + iA(z_0)$. Under this hypothesis, for any { $N \geq 1$ and any}
\begin{equation}
\label{e:range_beta_h}
{ 0 \leq \beta_\hbar \leq \left( C_N \hbar \log\frac{1}{\hbar} \right)^{2/3}}, 
\end{equation}
{ where $C_N = C_0 (N-2/3)$ for some positive constant $C_0 > 0$ depending only on $\gamma_0$}, we construct a quasimode at $\lambda_\hbar = V(z_0) + i \beta_\hbar$ of width { $O(\hbar^{2/3} \exp(-\beta_\hbar^{3/2}/C_0\hbar))$}, 
and with related semiclassical measure given by $\delta_{z_0}$. This leads to a { converse result} to \cite[Thm. 1.4]{Dencker04}, and its extension obtained in \cite[Thm. 1.2(B)]{Sjostrand09} via the semigroup approach. In particular, our quasimode construction provides  the following bound from below on the norm of the resolvent:
\begin{equation}
{ \Big \Vert \Big(  \widehat{P}_\hbar - (V(z_0) + i \beta_\hbar) \Big)^{-1} \Big \Vert_{\mathcal{L}(L^2)} \geq c_0 \hbar^{-2/3} \exp( \beta_\hbar^{3/2}/ C_0 \hbar) },
\end{equation}
for some $c_0 > 0$. Notice that points $z_0$ of the phase space satisfying \eqref{e:non_degenerate0} correspond to points $V(z_0) + iA(z_0)$ of the complex plane at the boundary $\partial \Lambda(V,A)$ of the pseudo-spectrum, see \cite[Eq. (1.3)]{Dencker04}, defined by
\begin{equation}
\label{e:pseudo-spectrum}
\Lambda(V,A) := \overline{\{ \zeta = V(z)+iA(z) \, : \, \{ V, A \}(z) \neq 0 \}};
 \end{equation}
that is, we are concerned with points where the Hörmander bracket condition $\{ V ,A \} \neq 0$ fails, see also \cite[Thm 1.2]{Dencker04} for constructions of quasimodes at points in the interior of the pseudo-spectrum, \cite[Sect. 26.2]{Hor94IV} for a discussion of this condition, and \cite{Pravda_Starov2004} for quasimode constructions with generalizations of this condition in one-dimensional systems. Therefore, our result extends \cite[Thm. 1.2]{Dencker04} (see also \cite{Davies02,Zworski01}) to points where $\{ A , V \}(z_0) = 0$. We recall (see \cite{Zworski01}, \cite[Thm. 12.8]{Zwobook}) that the more strict bracket condition $\{ V , A \}(z_0) < 0$ allows to conjugate the semiclassical operator, via the use of a Fourier integral operator (given in terms of a complex WKB construction associated to a positive Lagrangian submanifold of the complexification of $T^*\R^{d}$), microlocally near $z_0$, into a normal form given essentially by the annihilation operator of the harmonic oscillator at $w_0 = (0,0) \in \R^{2d}$ (see the proof of \cite[Thm. 12.8]{Zwobook}). When this condition fails at $z_0$, the construction of such a Fourier integral operator can not be carried out, since the associated Lagrangian submanifold is no longer positive. However, one can still study the evolution equation associated with $\widehat{P}_\hbar$ microlocally near $z_0$ to obtain  upper bounds on the resolvent (see \cite[Thms. 1.1 and 1.2]{Sjostrand09}). Our approach uses the same idea, although, comparing with this work, we describe very precisely the evolution of a Hagedorn wave-packet by the non-selfadjoint flow, using the quadratic approximation of $\widehat{P}_\hbar$ near $z_0$ (in the spirit of \cite{Robert12,RobertBook}), and then we estimate the remainding contributions to obtain a quasimode that saturate the resolvent estimate obtained in \cite[Thm. 1.2(B)]{Sjostrand09}. The main tool in \cite{Sjostrand09} is the use of an adapted FBI transform to the propagator of $\widehat{P}_\hbar$ microlocally near $z_0$. Alternatively, we use the Bargmann space to calculate the matrix elements of a microlocal approximation of $\widehat{P}_\hbar$ in a suitable orthonormal basis of excited coherent states (see Section \ref{s:matrix_elements}). 

To be more precise, our construction is based on the propagation of a wave-packet $\varphi_0^\hbar[z_0](x)$, centered at $z_0$, by the quantum flow generated by a non-selfadjoint pseudodifferential operator for small (microlocal) time. We will consider a symbol $\widetilde{P}$ approximating the complex symbol $V+iA$ near $z_0$, and we will write the solution $\varphi_\hbar(t,x)$ to the time-dependent Schrödinger equation
$$
\big( i\hbar \partial_t + \Op_\hbar(\widetilde{P}) \big) \varphi_\hbar(t,x) = 0, \quad \varphi_\hbar(0,x) = \varphi_0^\hbar[z_0](x),
$$ 
in a suitable $L^2$-basis $\{ \varphi_\alpha^\hbar[Z_t,z_t] \}_{\alpha \in \mathbb{N}^d_0}$ of excited coherent states. We will show that the evolution problem is well posed for small time $t \in [-\delta, \delta]$ and initial data belonging to a suitable subspace of analytic functions defined in terms of the decay of their coefficients in the basis $\{ \varphi_\alpha^\hbar[Z_t,z_t] \}$, and in particular for inital data given by $\varphi_0^\hbar[z_0](x)$. Then we will study the evolution of the elements of this basis by the propagator generated by the time-dependent quadratic approximation of $\widehat{P}_\hbar$ near the orbit issued from $z_0$ by the classical flow associated with the complex symbol of $\widehat{P}_\hbar$, using in particular \cite[Thm 4.9]{Schubert18}, and finally we will compare the evolution of these two systems employing Duhamel's principle. Our approach also extends some approximation results for solutions to the Schrödinger equation by wave-packets \cite{CombescureRobertWP} (see also \cite{Robert12,Hagedorn85,Hagedorn98})  to the non-Hermitian framework. 
Finally, the quasimode $\psi_\hbar$ will be obtained as:
\begin{equation}
\label{e:first_mention_quasimode}
\psi_\hbar =  \sqrt{\Theta_\hbar} \int_\R \chi_\hbar(t) e^{-\frac{i t}{\hbar}(\alpha_\hbar + i \beta_\hbar)} \varphi_\hbar(t,x) dt,
\end{equation}
for some $\chi_\hbar \in \mathcal{C}_c^\infty(\R)$ microlocalized near $t= 0$ and some normalizing constant $\Theta_\hbar$ so that $\Vert \psi_\hbar \Vert_{L^2(\R^d)} = 1$. Intuitively, the positive number $\beta_\hbar$ balances the $L^2$-density of $\varphi_\hbar(t,x)$ so that the semiclassical wave-front set of $\psi_\hbar$ concentrates at $z_0$ as $\hbar \to 0^+$. 

The results obtained here also apply to more general non-selfadjoint operators. In particular, let $(\varepsilon_\hbar)$ be a sequence of positive real numbers such that $\varepsilon_\hbar \to 0$ as $\hbar \to 0^+$, with $\hbar^2 \ll \varepsilon_\hbar \leq \hbar^\alpha$ for some $0 < \alpha < 2$, we consider, as in \cite{Ar_Riv18}, semiclassical perturbations of the quantum harmonic oscillator $\widehat{H}_\hbar$ of the form
\begin{equation}
\label{e:perturbed_operator}
\widehat{\mathcal{P}}_\hbar:= \widehat{H}_\hbar + \varepsilon_\hbar \widehat{V}_\hbar + i\hbar \widehat{A}_\hbar,
\end{equation}
where $\widehat{H}_\hbar$ is the semiclassical harmonic oscillator \eqref{e:quantum_harmonic_oscillator} and $\widehat{V}_\hbar$, $\widehat{A}_\hbar$ are  selfadjoint pseudodifferential operators with { real valued symbols} belonging in this case to $ S^0(\R^{2d};\R)$. 

In \cite{Ar_Riv18}, the problem is focused on the study of the asymptotic pseudo-spectrum of $\widehat{\mathcal{P}}_\hbar$ along sequences $\lambda_\hbar^\dagger = \alpha_\hbar + i \hbar \beta_\hbar$ such that
\begin{equation}
\label{e:good_sequences}
(\alpha_\hbar, \beta_\hbar) \rightarrow (1, \beta), \quad \text{as } \hbar \to 0^+,
\end{equation}
 under the following weak-geometric-control condition. Let $\delta > 0$ and denote $I_\delta = (1-\delta,1+\delta)$:

\medskip

\textbf{(WGC)}: For every $z_0 \in H^{-1}(I_\delta) \cap { \mathcal{I}_A^{-1}(0)}$, there exists $t \in \R$ such that ${ \mathcal{I}_A} \circ \phi_t^{\mathcal{I}_V}(z_0) > 0$,
\medskip

\noindent where, for any $a \in \mathcal{C}^\infty(\R^{2d})$, we denote by $ \mathcal{I}_a$ the average of $a$ by the flow $\phi_t^H$ of the harmonic oscillator:
$$
{ \mathcal{I}_a}(z) := \lim_{T \to + \infty} \frac{1}{T} \int_0^T a \circ \phi_t^H(z) dt,
$$
and $\phi_t^{\mathcal{I}_V}$ is the Hamiltonian flow generated by $\mathcal{I}_V$. In particular, see Appendix \ref{a:averaging_method}, $\mathcal{I}_a \in \mathcal{C}^\infty(\R^{2d})$ provided that $a \in \mathcal{C}^\infty(\R^{2d})$.

 In \cite[Thm. 1]{Ar_Riv18}, it is shown that the quasi-eignevalues of $\widehat{\mathcal{P}}_\hbar$ along  sequences $(\lambda_\hbar^\dagger)$ satisfying \eqref{e:good_sequences} remain at distance $\beta_\hbar \gg \varepsilon_\hbar$ from the real axis as $\hbar \to 0^+$, for quasimodes of width $o(\varepsilon_\hbar \hbar)$. In other words, there exists an asymptotic spectral gap of width larger than $\varepsilon_\hbar \hbar$  near the real axis for the pseudo-spectrum of $\widehat{\mathcal{P}}_\hbar$. As a consequence of this result, the following resolvent estimate for $\widehat{\mathcal{P}}_\hbar$ holds: For every $R > 0$, there exists a constant $\delta_R > 0$ such that, for $\hbar > 0$ small enough and for $\varepsilon_\hbar \geq \delta_R \hbar^2$,
\begin{equation}
\label{e:resolvent_estimate}
\left \vert 1 - \Re \lambda \right \vert \leq \delta, \quad \frac{\Im \lambda}{\hbar} \leq R \varepsilon_\hbar \implies \left \Vert \big( \widehat{\mathcal{P}}_\hbar - \lambda \big)^{-1} \right \Vert_{\mathcal{L}(L^2)} \leq \frac{1}{\delta_R \hbar \varepsilon_\hbar}.
\end{equation}

On the other hand, under suitable analytical hypothesis on $V$ and $A$, and assuming a Diophantine property on $\omega$, \cite[Thm. 2]{Ar_Riv18} (see also \cite{AschLebeau}), one can show that the true spectrum of $\widehat{\mathcal{P}}_\hbar$ along sequences $\lambda_\hbar^\dagger = \alpha_\hbar + i \hbar \beta_\hbar$ satisfying \eqref{e:good_sequences} is asymptotically at distance $\liminf \beta_\hbar = \beta \geq \epsilon$ from the real line, for some $\epsilon = \epsilon(V,A) > 0$. This means that the asymptotic spectral gap for the true eigenvalues of $\widehat{\mathcal{P}}_\hbar$, in the analytic case, is larger than the asymptotic spectral gap established in \cite[Thm. 1]{Ar_Riv18} for the quasi-eigenvalues. 

The remaining question is then if this different limiting behavior of the pseudo-spectrum and the true spectrum really occurs, or equivalently, if the spectrum lies deep inside the pseudo-spectrum (see \cite{Dencker04} for a discussion of this phenomenon in non-selfadjoint problems and for references to one-dimensional examples). 

In the present article, assuming $\varepsilon_\hbar = \hbar$ for simplicity, we answer this question, proving a converse for \cite[Thm. 1]{Ar_Riv18}. Precisely, under a suitable { finite-type} control condition (see \eqref{e:non-degenerate_2} below) we show that for any $ N \geq 1$ and any $\beta_\hbar$ satifying \eqref{e:range_beta_h}, there exists a quasimode $(\psi_\hbar^\dagger)$ at $\lambda_\hbar^\dagger = \alpha_\hbar + i \hbar \beta_\hbar$ for $\widehat{\mathcal{P}}_\hbar$ of width { $O(\hbar^{2/3} \exp(- \beta_\hbar^{3/2}/C_0\hbar))$}. This confirms that the pseudo-spectrum of $\widehat{\mathcal{P}}_\hbar$, assuming that $V$ and $A$ are real analytic, lies in a wider strip containing the asymptotic spectrum in the semiclassical limit.
 
To construct such a quasimode $(\psi_\hbar^\dagger,\lambda_\hbar^\dagger)$ for $\widehat{\mathcal{P}}_\hbar$, we first conjugate the operator $\widehat{\mathcal{P}}_\hbar$ into a quantum Birkhoff normal form ${ \widehat{\mathcal{P}}_\hbar^\dagger = \widehat{H}_\hbar + \hbar \Op_\hbar(\mathcal{I}_{P_\hbar}) + O(\hbar^N)}$, where $P_\hbar = P + O(\hbar)$ has principal symbol $P = V + iA$, so that the perturbation ${ \Op_\hbar(\mathcal{I}_{P_\hbar})}$ commutes with $\widehat{H}_\hbar$. To this aim we assume a Diophantine property on $\omega$ (see \eqref{e:diophantine_property} below). The construction of the normal form is now standard, and we give it in Appendix \ref{a:averaging_method}. We then propagate a wave packet $\varphi_0^\hbar[z_0]$ centered at $z_0$ by both the quantum flow of the harmonic oscillator and the non-selfadjoint flow generated by { $\Op_\hbar(\mathcal{I}_{P_\hbar})$}, obtaining our quasimode $\psi_\hbar^\dagger$ as a quantum average of $\varphi_0^\hbar[z_0]$ on the minimal invariant torus $\mathcal{T}_\omega(z_0)$ issued from $z_0$ by $\phi_t^H$, and (microlocally) on the orbit issued from $\mathcal{T}_\omega(z_0)$ by the ``non-selfadjoint"  flow generated by $\mathcal{I}_{P_\hbar}$, which is transversal to the flow $\phi_t^H$ due to condition \textbf{(WGC)}, in a similar way as \eqref{e:first_mention_quasimode}, see \eqref{e:second_mention_quasimode} below.
 
\subsection{Statement of our results}

In this section we state the main results of this article. We first recall the precise definition of quasimode.

\begin{definition} Let $\widehat{P}_\hbar = \Op_\hbar(P)$ be a semiclassical pseudo-differential operator. A quasimode $(\psi_\hbar, \lambda_\hbar)$ for $\widehat{P}_\hbar$ is a sequence of solutions to
\begin{equation}
\label{e:definition_quasimode}
\widehat{P}_\hbar \psi_\hbar = \lambda_\hbar \psi_\hbar + R_\hbar, \quad \Vert \psi_\hbar \Vert_{L^2(\R^d)} = 1, \quad \hbar \leq \hbar_0,
\end{equation}
where $r_\hbar := \Vert R_\hbar \Vert_{L^2(\R^d)} \to 0^+$. The sequence $(r_\hbar)$ is called the \textit{width} of the quasimode, and is typically of order $o(\hbar)$.
\end{definition}

Our first result concerning the operator $\widehat{P}_\hbar$ { given by \eqref{e:damped_wave_operator}} shows that it is possible to construct $O(\hbar^N)$-quasimodes for this operator concentrating on a given point $z_0 \in A^{-1}(0)$ satisfying \eqref{e:non_degenerate0}.

\begin{teor}
\label{t:T1}
Let $A,V \in { S^k(\R^{2d};\R)}$ with $A \geq 0$. Let $z_0 \in A^{-1}(0)$ such that
\begin{equation}
\label{e:non-degenerate}
\gamma_0 = \gamma_0(A,V,z_0) := \big \langle X_V(z_0), \partial^2 A(z_0) X_V(z_0) \big \rangle > 0.
\end{equation}
Then, there exists a constant $C_0> 0$ depending only on $\gamma_0$ such that, for any { $N \geq 1$ and any $\beta_\hbar$ satisfying \eqref{e:range_beta_h}}, where $C_N =  C_0 (N-2/3)$, there exists a quasimode $(\psi_\hbar, \lambda_\hbar)$ for $\widehat{P}_\hbar$ of width ${ O(\hbar^{2/3} \exp(-\beta_\hbar^{3/2}/C_0\hbar))}$ at $\lambda_\hbar = V(z_0) + i \beta_\hbar$, and
\begin{equation}
\label{e:weak_limit}
W_\hbar[\psi_\hbar] \rightharpoonup^\star \delta_{z_0},
\end{equation}
where $W_\hbar[\psi_\hbar]$ is the semiclassical Wigner distribution of $\psi_\hbar$ (see Definition \ref{d:wigner_distribution} below).
\end{teor}

\begin{remark}
{ Notice that, when $\beta_\hbar$ reaches the upper bound in \eqref{e:range_beta_h}, the width of the quasimode $(\psi_\hbar,\lambda_\hbar)$ given by Theorem \ref{t:T1} becomes of order $O(\hbar^{N})$. Using a diagonal argument in $\hbar$ and $N$, with the hypothesis of Theorem \ref{t:T1} one can prove the existence of quasimodes of width $O(\hbar^\infty)$ with $\hbar^{2/3- \epsilon} \gg \beta_\hbar \gg (\hbar \log \hbar^{-1})^{2/3}$ for every $\epsilon > 0$. An interesting question that we do not cover in this article is to give necessary and sufficient regularity hypothesis on the symbols $V$, $A$ to construct quasimodes of width $O(\hbar^{2/3} \exp(-\beta_\hbar^{3/2}/C_0\hbar))$ for  larger $\beta_\hbar \to 0^+$. Notice that, when $\lim_{\hbar \to 0} \beta_\hbar = \beta > 0$ and the symbols $V$ and $A$ are analytic, one retrieves the exponential quasimodes of \cite[Thm. 1.2]{Dencker04} in the interior of the pseudo-spectrum.}
\end{remark}

\begin{remark}
Condition \eqref{e:non-degenerate} can be weakened. Indeed, for our results to hold it is enough that $\partial^{2m} A[X_V](z_0) > 0$ for some $m \geq 1$, where $\partial^{2m} A[X_V]$ denotes the $2m$-tensor of derivatives with all entries evaluated at $X_V$. In this case, one can construct quasimodes of width $ O(\hbar^{\frac{2m}{2m+1}} \exp( -\beta_\hbar^{(2m+1)/2m}/C_0\hbar))$ for
$$
0 \leq \beta_\hbar \leq (C_N \hbar \log\hbar^{-1})^{2m/(2m+1)}.
$$ 
\end{remark}

\begin{remark}
More generally, one can use our technique to construct quasimodes at $z_0 \in A^{-1}(0)$ with $\partial^{2m} A[X_V](z_0) = 0$ for every $m \geq 1$, assuming that $A \circ \phi_t^V(z_0) = 0$ for $t \in [0,\delta)$ and $A \circ \phi_t^V(z_0) > 0$ for $t \in (-\delta,0)$, which is the typical situation for a function $A \in \mathcal{C}_c^\infty(\R^{2d})$ satisfying condition \textbf{(GC)}. In this case, one can construct quasimodes of width { $O(\hbar \exp(- \beta_\hbar/C_0\hbar))$} for $0 \leq \beta_\hbar \lesssim \hbar \log \hbar^{-1}$, in connection with \cite[Thm 1.2(A)]{Sjostrand09}.
\end{remark}

To state our next result concerning the perturbed harmonic oscillator $\widehat{\mathcal{P}}_\hbar$, we first recall some facts about the classical harmonic oscillator.  We consider the decoupled one-dimensional harmonic oscillators
$$
H_j(x,\xi) = \frac{1}{2} \big( \xi_j^2 + x_j^2 \big), \quad j \in \{1, \ldots , d \},
$$
which constitute a set of $d$-independent integrals of the motion in involution. Indeed, the harmonic oscillator $H$ can be written as a function of $H_1, \ldots, H_d$,
\begin{equation}
\label{e:linear_for_H}
H = \mathcal{L}_\omega(H_1, \ldots, H_d),
\end{equation}
where $\mathcal{L}_\omega : \R^d_+ \to \R$ is the linear form defined by $\mathcal{L}_\omega(E) = \omega \cdot E$, and moreover, $\{ H_j , H_k \} = 0$ for every $j,k\in \{1, \ldots , d \}$. It then follows that $\phi_t^H$, the Hamiltonian flow of $H$, can be written as 
\[
\phi_t^H(z)=\Phi_z(t\omega), \quad t\in\R, \;z = (x,\xi)\in\R^{2d},
\] 
where the generalized flow $\Phi_z(\tau)$ is given by
\begin{equation}
\label{e:multiflow}
\Phi_z(\tau) := \phi_{t_d}^{H_d} \circ \cdots \circ \phi_{t_1}^{H_1}(z), \quad \tau=(t_1, \ldots, t_d) \in \R^d,
\end{equation}
and $\phi_{t}^{H_j}$ denotes the flow of $H_j$. These flows are totally explicit, they act as a rotation of angle $t$ on the plane $(x_j,\xi_j)$. Therefore, $\Phi_z$ is $2\pi\Z^d$-periodic for every $z\in\R^{2d}$
and we will identify it to a function defined on the torus $\mathbb{T}^d:=\mathbb{R}^d/2\pi\mathbb{Z}^d$. 

Let us define
\begin{equation}
\label{e:various_notation}
M_H:=(H_1,\hdots,H_d),\quad X:=(0,\infty)^d,\quad \Sigma:=\R^d_+\setminus X.
\end{equation}
For every vector of energies $E\in\R^d_+$, let $\mathcal{T}_E:=M_H^{-1}(E)$ be the invariant torus with vector of energies given by $E$; these tori are invariant by the flow $\phi^H_t$. If $E\in X$ then $\mathcal{T}_E$ is Lagrangian and, for every $z_0\in M_H^{-1}(E)$, $\Phi_{z_0}:\T^d  \longrightarrow \mathcal{T}_E$ is a diffeomorphism; moreover, 
\[
\Phi_{z_0}^{-1}\circ \phi^H_t\circ \Phi_{z_0}(\tau) = \tau +t\omega,\quad \forall t\in\R.
\]
Kronecker's theorem then shows that the orbit of $\phi^H_t$ from any point $z_0\in M_H^{-1}(X)$ is dense in a subtorus $\mathcal{T}_\omega(z_0)$ of $\mathcal{T}_E$. The dimension of $\mathcal{T}_\omega(z_0)$ depends on the arithmetic relations between the components of $\omega$. Let 
\[
\langle \omega_1,\hdots,\omega_d \rangle_{\mathbb{Q}}
\]
be the linear subspace of $\R$, viewed as a vector space over the rationals, spanned by the frequencies; then
\[
d_\omega := \operatorname{dim} \langle \omega_1,\hdots,\omega_d \rangle_{\mathbb{Q}} = \dim \mathcal{T}_\omega(z_0).
\]
 Notice, in particular, that if $d_\omega = d$, then $\mathcal{T}_\omega(z_0)=\mathcal{T}_E$ provided that $M_H(z_0)=E \in X$. In the opposite case, when $d_\omega=1$, the flow $\phi^H_t$ is periodic of period 
\[
T_\omega:=2\pi k_\omega/\omega_{1},
\]
where $k_\omega$ is the least positive integer such that $k_\omega\omega_j/\omega_{1}\in\Z$ for every $j=1,\hdots,d$.	 When $d_\omega<d$, the vector of frequencies $\omega$ is said to be \textit{resonant}.

To deal with the case $M_H(z_0) = E \in { \Sigma \cap \mathcal{L}_\omega^{-1}(1)}$, define for $v\in\R^d$ and $z\in\R^{2d}$, 
\begin{equation}
\label{e:projection_torus}
\pi_z(v):=(\mathbf{1}_{(0,\infty)}(H_1(z))v_1,\hdots,\mathbf{1}_{(0,\infty)}(H_d(z))v_d).
\end{equation}
In this case, the map $\Phi_{z_0}$ is no longer a diffeomorphism but one still has:
\[
\phi^H_t\circ \Phi_{z_0}(\tau) = \Phi_{z_0}(\tau +t\pi_{z_0}(\omega)),\quad \forall t\in\R.
\]
Therefore, the orbit issued from such $z_0$ is again dense in a torus of  dimension  $1 \leq  d_0 <d_\omega$, which we will still denote by $\mathcal{T}_\omega(z_0)$.

In order to state our results, we need to assume a Diophantine property on the vector of frequencies $\omega$. This is important to construct a normal form for $\widehat{\mathcal{P}}_\hbar$ (see Section \ref{s:normal_form} of the Appendix), so that $\widehat{V}_\hbar$ and $\widehat{A}_\hbar$ are averaged by the quantum flow $e^{-\frac{i}{\hbar} \widehat{H}_\hbar}$ up to order $N$, ensuring that these averages commute with  $\widehat{H}_\hbar$ (see \cite{Charles08} and the references therein).  

\begin{definition}
A vector $\omega \in \R^d_+$ is called partially Diophantine if there exist constants $\varsigma > 0$ and $\gamma = \gamma(\omega) \geq 0$ such that
\begin{equation}
\label{e:diophantine_property}
\vert \omega \cdot k \vert \geq \frac{ \varsigma}{\vert k \vert^\gamma}, \quad k \in \mathbb{Z}^d \setminus \Lambda_\omega,
\end{equation}
where the resonant set $\Lambda_\omega$ is given by \eqref{e:submodule}.
\end{definition}
\begin{remark}
Notice that $\omega = (1, \ldots, 1)$ is obviously partially Diophantine.
\end{remark}

We next state our main result concerning the construction of quasimodes for $\widehat{\mathcal{P}}_\hbar$:

\begin{teor}
\label{t:T2}
Let $\varepsilon_\hbar = \hbar$, $A,V \in S^0(\R^{2d};\R)$ with $\mathcal{I}_A \geq 0$. Assume that $\omega$ is partially Diophantine and $d_\omega < d$. Suppose that, for a given $z_0 \in H^{-1}(1) \cap \mathcal{I}_A^{-1}( 0)$,
\begin{equation}
\label{e:non-degenerate_2}
 \gamma_0 = \gamma_0(\mathcal{I}_V,\mathcal{I}_A,z_0) := \big \langle X_{\mathcal{I}_V}(z_0), \partial^2 \mathcal{I}_A(z_0) X_{\mathcal{I}_V}(z_0) \big \rangle > 0.
\end{equation}
Then there exists a constant $C_0> 0$ such that, { for every $N \geq 1$} and every $\beta_\hbar$ satisfying \eqref{e:range_beta_h}, where $C_N =  C_0 (N-2/3)$, there exists a quasimode $(\psi^\dagger_\hbar, \lambda_\hbar^\dagger)$ for $\widehat{\mathcal{P}}_\hbar = \widehat{H}_\hbar + \hbar (\widehat{V}_\hbar + i \widehat{A}_\hbar)$ of width $O(\hbar^{2/3}\exp(-\beta_\hbar^{3/2}/C_0\hbar))$ so that 
$$
 \lambda_\hbar^\dagger = \omega \cdot E_\hbar + \hbar \, \mathcal{I}_V(z_0) + i  \hbar \beta_\hbar,
$$
where {  $E_\hbar = M_H(z_0) + O(\hbar) \in \big( \operatorname{Sp}_{L^2(\R^d)}(\Op_\hbar(H_1)), \ldots, \operatorname{Sp}_{L^2(\R^d)}(\Op_\hbar(H_d)) \big)$}, where $\operatorname{Sp}_{L^2(\R^d)}(\cdot)$ denotes the spectrum, and
$$
W_\hbar[\psi_\hbar^\dagger] \rightharpoonup^\star \delta_{\mathcal{T}_\omega(z_0)},
$$
where $\mathcal{T}_\omega(z_0)$ is the torus issued from $z_0$ by the flow $\phi_t^H$.
\end{teor}

\begin{remark}
The assumption $\varepsilon_\hbar = \hbar$ can be relaxed to deal with $\hbar^2 \ll \varepsilon_\hbar \lesssim \hbar^{\alpha}$, for some $\alpha < 2$, and $0 \leq \beta_\hbar \leq (C_N \varepsilon_\hbar \log \varepsilon_\hbar^{-1} )^{2/3}$. We prefer not to deal with this case for the sake of simplicity, and since $\varepsilon_\hbar = \hbar$ is the regime in which $V$ and $A$ interact at the same scale.
\end{remark}

\begin{remark}
As for Theorem \ref{t:T1}, the assumption \eqref{e:non-degenerate_2} can be weakened to the case in which $\partial^{2m} \mathcal{I}_A[X_{\mathcal{I}_V}](z_0) > 0$ for some $m \geq 1$. In this case, for $0 \leq \beta_\hbar \leq (C_N \hbar \log \hbar^{-1})^{2m/(2m+1)}$ one obtains quasimodes of width $O(\hbar^{\frac{2m}{2m+1}} \exp( -\beta_\hbar^{(2m+1)/2m}/C_0\hbar))$.
\end{remark}

\begin{remark}
Under condition \eqref{e:non-degenerate_2}, $\mathcal{T}_\omega(z_0) \subset \mathcal{I}_A^{-1}(0) \cap H^{-1}(1)$ is a smooth subtorus of dimension $1 \leq d_0 \leq d_\omega$. 
\end{remark}

\begin{remark}
Notice that, if $d_\omega = d$, then condition \textbf{(WGC)} can only be satisfied if $\mathcal{I}_A > 0$ on $H^{-1}(1)$. Indeed, in this case, $X_{\mathcal{I}_V}(z_0)$ is tangent to $\mathcal{T}_\omega(z_0)$, and then $\mathcal{I}_V(z)$ is constant along this direction.
\end{remark}

To prove Theorem \ref{t:T2}, we construct a quasimode $\psi_\hbar$ for the normal form { $\widehat{\mathcal{P}}_\hbar^\dagger = \widehat{H}_\hbar + \hbar \Op_\hbar(\mathcal{I}_{P_\hbar}) + O(\hbar^N)$}, given by Proposition \ref{l:first_normal_form} of Appendix \ref{a:averaging_method}. The orbit issued from $z_0$ by the flow $\phi_t^H$ is dense on the minimal invariant torus $\mathcal{T}_\omega(z_0)$, which has dimension $1 \leq d_0 \leq d_\omega$. One can parametrize this torus from a flat subtorus $\mathbb{T}_{d_0} \subset \mathbb{T}^d$, so that we can denote $\mathbb{T}_{d_0} \ni \tau \mapsto z(\tau) \in \mathcal{T}_\omega(z_0)$ (see also \cite{Ar_Mac18}).  Moreover, the complex symbol $\mathcal{I}_{P_\hbar}$ generates a classic flow (see Lemma \ref{l:center_evolution}) which commutes with $\phi_s^H$ and is transversal to $\mathcal{T}_\omega(z_0)$ at $z(\tau)$ for every $\tau \in \mathbb{T}_{d_0}$, provided that \eqref{e:non-degenerate_2} holds. Denoting by $z(\tau,t)$ the orbit issued from $z_0$ by these two commuting flows, we first consider the propagation of the wave packet $\varphi_0^\hbar[z_0]$ by the quantum flow $e^{- \frac{i}{\hbar} \tau \cdot \Op_\hbar(H_1, \ldots, H_d)}$, that is, $\varphi_\hbar(\tau,0,x)= e^{- \frac{i}{\hbar} \tau \cdot \Op_\hbar(H_1, \ldots, H_d)}\varphi_0^\hbar[z_0]$, for $\tau \in \mathbb{T}_{d_0}$, and then consider the evolution equation
$$
\big( i\hbar \partial_t + \Op_\hbar(\widetilde{P}) \big) \varphi_\hbar(\tau,t,x) = 0, \quad \varphi_\hbar(0,0,x) = \varphi_0^\hbar[z_0](x), 
$$
microlocally near $t = 0$, where $\widetilde{P}$ is a suitable approximation of the symbol $\mathcal{I}_{P_\hbar}$ near the orbit $z(\tau,t)$. Finally, we will obtain $\psi_\hbar$ as
\begin{equation}
\label{e:second_mention_quasimode}
\psi_\hbar =  \sqrt{\Theta_\hbar} \int_{\mathbb{T}_{d_0}} \int_\R \chi_\hbar(t) e^{\frac{i \tau \cdot E_\hbar}{\hbar}} e^{-\frac{i t}{\hbar}(\alpha_\hbar + i \beta_\hbar)} \varphi_\hbar(\tau,t,x) dt \mu_\omega^{z_0}(d\tau),
\end{equation}
where $\mu_\omega^{z_0}$ denotes the Haar measure of $\mathbb{T}_{d_0}$, for some normalizing constant $\Theta_\hbar$, and  
$$
E_\hbar = \hbar\left( N_1(\hbar) + \frac{1}{2}, \ldots, N_d(\hbar) + \frac{1}{2} \right), \quad N_j(\hbar) \in \mathbb{N}_0,
$$ 
is chosen so that $\omega \cdot E_\hbar$ is a sequence of eigenvalues for $\widehat{H}_\hbar$ tending to one as $\hbar \to 0^+$. Conjugating back $\psi_\hbar$ by the Fourier integral operator giving the normal form of $\widehat{\mathcal{P}}_\hbar$, we obtain our quasimode $\psi_\hbar^\dagger$ for the original operator $\widehat{\mathcal{P}}_\hbar$. We will show that conjugation by this Fourier integral operator leaves semiclassical measures invariant, so the phase-space semiclassical concentration on $\mathcal{T}_\omega(z_0)$ remains { unchanged}.
 
%\begin{corol}
%\label{c:converse_resolvent_estimate}
%Let $R > 0$. If $\lambda \in \mathbb{C}$ satisfies
%$$
%\vert 1 + \hbar \langle V \rangle(z_0) - \Re \lambda \vert \leq R \hbar^{5/3}, \quad \vert \Im \lambda \vert \leq R\hbar^{5/3},
%$$
%then
%$$
%\left \Vert \big( \widehat{\mathcal{P}}_\hbar - \lambda \big)^{-1} \right \Vert_{\mathcal{L}(L^2)} \geq \frac{1}{R\hbar^{5/3}}.
%$$
%\end{corol}

%\begin{remark}
%Corollary \ref{c:converse_resolvent_estimate} gives a converse estimate for the upper bound \eqref{e:resolvent_estimate} obtained in \cite{Ar_Riv18}. Moreover, it shows that the asymptotic spectral gap obtained in \cite[Thm 2]{Ar_Riv18} for the true eigenvalues of $\widehat{\mathcal{P}}_\hbar$ does not hold for general quasimodes of width $O(\hbar^N)$.
%\end{remark}

\subsection{Related works}

Some foundational works in the study of the asymptotic properties of the damped wave equation are  \cite{Lebeau96,RauchTaylor,Sjostrand00} (see also \cite{Markus88,Markus_Matsaev80} for related works). In the context of the damped-wave equation on Riemannian manifolds, in \cite[Thm. 0.1]{Sjostrand00} it is shown that  eigenvalues of the damped-wave operator verify a Weyl law in the high frequency limit and, moreover, these eigenvalues lie in a strip of the complex plane which can be completely determined in terms of the average of the damping function along the geodesic flow \cite[Thms. 0.0 and 0.2]{Sjostrand00} (see also \cite{Lebeau96,Markus88,Markus_Matsaev80}). These results are particular cases of a later systematic study
of  non-selfadjoint semiclassical problems which have been the object 
of several works. More precisely, it has been investigated how the spectrum and the pseudo-spectrum are asymptotically distributed inside 
the strip determined in \cite{Sjostrand00} and how the dynamics of the underlying classical Hamiltonian influences this asymptotic distribution. One can ask about precise estimates on the distribution of eigenvalues inside the strip; this question has been addressed both in the chaotic case \cite{Anantharaman10},  and in the completely integrable 
framework~\cite{Hitrik02,HitrikSjostrand04,HitrikSjostrand05,HitrikSjostrand08,HitrikSjostrand08b,HitrikSjostrand12,HitrikSjostrand18,HitrikSjostrandVuNgoc07}. 
In this series of works, the authors describe the distribution of eigenvalues in certain regions of the complex plain for non-selfadjoint perturbations of selfadjoint $\hbar$-pseudodifferential operators for which the classical Hamiltonian flow generated by its principal symbol has suitable periodic or quasiperiodic structure, and study how periodic orbits, resonant or Diophantine tori, in different situations, influence the distribution of the spectrum in terms of the size of the perturbation and its average by the principal Hamiltonian flow. In particular, spectral contributions coming from rational or Diophantine tori, tunneling effects and Weyl's laws are obtained for these systems. An important assumption along these works is that the subprincipal symbol of the selfadjoint operator does not interfere with the imaginary part of the perturbation, in the sense that the size of real part of the perturbation is larger, the subprincipal symbol vanishes or Poisson commutes with the imaginary part of the principal symbol. On the other hand, the present work precisely focus on this interaction between the real and imaginary parts of the perturbation, and how this interaction generates a rich structure in the pseudo-spectrum.

It is also natural to focus on how eigenfrequencies accumulate at the boundary of the strip and also  to get resolvent estimates near this boundary. Again, this question has been explored both in 
the integrable case~\cite{AnantharamanLeautaud,AschLebeau,BurqGerard18,BurqHitrik, HitrikSjostrand04} 
and in the chaotic one~\cite{Christianson07, ChristiansonSchenckVasyWunsch,Jin17,Nonnenmacher11,Riviere14,Schenck10,Schenck11}.  In this spirit, one can ask more generally about the structure of the pseudospectrum of a general non-selfadjoint (pseudo-)differential operator near the boundary of the range of the principal symbol. In this framework, Theorem \ref{t:T1} gives a converse result to \cite[Thm. 1.4] {Dencker04} and \cite[Thm. 1.2]{Sjostrand09} under finite-type dynamical conditions, while Theorem \ref{t:T2} gives a converse result to \cite[Thm. 1]{Ar_Riv18}, where the dynamical-control-condition appears in the subprincipal part of the operator.

Among other things, our study initiated in \cite{Ar_Riv18} and continued in the present work, is motivated by earlier results \cite{AschLebeau} for the damped-wave equation on the sphere. In that reference, it is shown how a selfadjoint pertubation of the principal symbol of the damped wave operator on the $2$-sphere can create a spectral gap near the real axis in the high frequency limit. Moreover, we have also been motivated by previous works \cite{Mac_Rive16,Mac_Riv18,Mac_Riv19} on semiclassical asymptotics for the Schrödinger equation associated to some other completely integrable systems, as the Schrödinger equation on the torus or a Zoll manifold.

As for~\cite{Ar_Mac18}, we restrict ourselves to the case of non-selfadjoint perturbations of semiclassical harmonic oscillators on $\mathbb{R}^d$. Yet it is most likely that the methods presented here can be adapted to deal with semiclassical operators associated with more general completely integrable systems, including damped wave equations on Zoll manifolds, see \cite{Villegas18}, where new constructions of quasimodes are given for non-selfadjoint perturbations of the Laplace Beltrami operator on Zoll manifolds at points satisfying the Hörmander bracket condition.

Some other related works concerned with the construction of quasimodes for non-selfadjoint operators are those of \cite{Viola13,Starov07,Starov08}, in which the authors focus on the study of the pseudo-spectrum, resolvent estimates and Mehler's formulas for certain non-selfadjoint quadratic operators; \cite{Siegl18,Viola15}, where a systematic study of the speudo-spectrum of non-selfadjoint operators of 1D systems is done. It is also relevant the work \cite{Starov08a}, where some extensions of the results of \cite{Dencker04} are given to the injective pseudo-spectrum, showing absence of quasimodes at $\lambda_\hbar = 0$ of width $O(\hbar^{k/(k+1)})$ for some pseudodifferential operators satisfying principal-type conditions.

\subsection*{Acknowledgements}

This project has received funding from the European Research Council (ERC) under the European Union's Horizon 2020 research and innovation programme (grant agreement No. 725967). This work has also been partially supported by grant MTM2017-85934-C3-3-P (MINECO, Spain). The author thanks Fabricio Macià for many inspiring and useful conversations about semiclassical analysis, and Stéphane Nonnenmacher for fruitful comments and discussions on this work.

\section{Hagedorn wave packets}

In this section we briefly review some constructions and results of \cite{Schubert18} (see also \cite{Hitrik19,Hagedorn85,Hagedorn98}) to introduce the notions of Hagedorn wave-packets, using the formalism of Lagrangian frames.

\subsection{Lagrangian frames}
First,  we  discuss some complex-symplectic linear algebra. We consider the real vector space $\R^{2d}$ endowed with the symplectic form $\R^{2d} \times \R^{2d} \ni (z,w) \mapsto z \cdot \Omega w \in \R$, where $\Omega \in \R^{2d \times 2d}$ is the canonical symplectic matrix:
\begin{equation}
\Omega = \left( \begin{array}{cc}
0 &  - \operatorname{Id}_d \\
 \operatorname{Id}_d & 0 
 \end{array} \right).
\end{equation} 
Those matrices $F \in \R^{2d \times 2d}$ that respect the standard symplectic structure satisfy $F^T \Omega F = \Omega$. They constitute the symplectic group $\operatorname{Sp}(d,\R)$. Writing $F = (U,V)$ with blocks $U,V \in \R^{2d \times d}$, one can associate to $F$ the complex rectangular matrix $Z = U - i V \in \mathbb{C}^{2d \times d}$, which satisfies:
\begin{equation}
\label{e:normalized_lagrangian_frame}
Z^T \Omega Z = 0, \quad Z^* \Omega Z = 2i \Id_d.
\end{equation}
The matrices $Z \in \mathbb{C}^{2d\times d}$ satisfying \eqref{e:normalized_lagrangian_frame} are called normalized Lagrangian frames. They are in one-to-one correspondence with the real symplectic matrices: If $Z \in \mathbb{C}^{2d \times d}$ is a normalized Lagrangian frame, then $F = (\operatorname{Re}(Z), -\operatorname{Im}(Z))$ is symplectic.

By the first property of \eqref{e:normalized_lagrangian_frame}, all column vectors $l,l'$ of $Z$ satisfy
$$
l \cdot \Omega l' = 0,
$$
that is, $l$ and $l'$ are \textit{skew-orthogonal}. A subspace $L \subset \mathbb{C}^d \oplus \mathbb{C}^d$ is called \textit{isotropic} if all vectors in $L$ are skew-orthogonal to each other. Moreover, $L$ is called Lagrangian if it is isotropic and has maximal dimension $d$. From the second property of \eqref{e:normalized_lagrangian_frame} (normalization), one can see that all vectors $l \in \operatorname{range} Z \setminus \{ 0 \}$ satisfy
$$
\frac{i}{2} (\Omega \overline{l}) \cdot l > 0.
$$
In other words, the quadratic form
$$
h(z,z') := \frac{i}{2}(\Omega \overline{z}) \cdot z' = \frac{i}{2} \overline{z} \cdot \Omega^T z', \quad z,z' \in \mathbb{C}^d \oplus \mathbb{C}^d,
$$
is positive on the range of $Z$. Such a Lagrangian subspace is called \textit{positive}.

If $L$ is a positive Lagrangian subspace, then $\overline{L}$ is Lagrangian too. Moreover, all vectors $l \in \overline{L} \setminus \{ 0 \}$ satisfy
$$
h(l,l) = \frac{i}{2} \overline{l} \cdot \Omega^T l < 0,
$$
so that $\overline{L}$ is called \textit{negative} Lagrangian. Moreover, $L \cap \overline{L} = \{ 0 \}$, hence
$$
\mathbb{C}^d \oplus \mathbb{C}^d = L \oplus \overline{L},
$$
where this decomposition is orthogonal in the sense that
$$
h(l,l') =0, \quad \text{for all } \; l \in L, \, l' \in \overline{L}.
$$

With any Lagrangian subspace $L \subset \mathbb{C}^d \oplus \mathbb{C}^d$, one can associate many Lagrangian frames spanning $L$, that is, $L = \operatorname{range}(Z)$. Indeed, every two normalized Lagrangian frames $Z_1, Z_2$ spanning the same Lagrangian subspace $L$ are related by a unitary matrix $U$, so that $Z_1 = Z_2U$. This implies that the Hermitian squares $Z_1Z_1^* = Z_2 Z_2^*$ are the same. Moreover, one can define the \textit{metric} and \textit{complex} structure of $L$:

\begin{definition}{\cite[Def. 2.6]{Schubert18}}
Let $L \subset \mathbb{C}^d \oplus \mathbb{C}^d$ be a positive Lagrangian subspace and $Z$ be a normalized Lagrangian frame spanning $L$.
\begin{enumerate}
\item The real, symmetric, positive definite, symplectic matrix
\begin{equation}
\label{e:metric}
G = \Omega^T \Re (Z Z^*) \Omega
\end{equation}
is called the symplectic metric of $L$.

\item The symplectic matrix 
$$
J = - \Omega G,
$$
with $J^2 = - \Id_{2d}$ is called the  complex structure of $L$.
\end{enumerate}
\end{definition}

The complex structure $J$ can be used to precisely write the orthogonal projections from $\mathbb{C}^d \oplus \mathbb{C}^d$ onto $L$ and $\overline{L}$:

\begin{prop}{\cite[Prop. 2.3 and Corol. 2.7]{Schubert18}}\label{p:projections} Let $L \subset \mathbb{C}^d \oplus \mathbb{C}^d$ be a positive Lagrangian and $Z$ a normalized Lagrangian frame with $\operatorname{range} Z = L$. Then,
$$
\pi_L = \frac{i}{2} Z Z^* \Omega^T, \quad \text{and}  \quad \pi_{\overline{L}} = - \frac{i}{2} \overline{Z}Z^T \Omega^T,
$$
are the orthogonal projections (with respect to the two form $h$) onto $L$ and $\overline{L}$, respectively. Moreover,
$$
\pi_L = \frac{1}{2} \big( \Id_{2d} + i J \big), \quad \text{and} \quad \pi_{\overline{L}} = \frac{1}{2} \big( \Id_{2d} - i J \big).
$$
\end{prop}

\subsection{Coherent and excited states}
\label{s:coherent_and_excited_states}

In this section we recall the construction of an orthonormal basis of $L^2(\R^d)$ consisting of Hermite-type states obtained from a given normalized Lagrangian frame $Z$ and centered at a phase-space point $z \in \R^{2d}$. These states are called coherent and excited Hagedorn wave-packets. We will sometimes use the identification $ \mathbb{R}^{2d} \equiv \mathbb{C}^d$ given by $z = (x,\xi) \equiv x + i \xi$ without mention. 

With any normalized Lagrangian frame $Z$, one can associate a lowering operator, or annihilator, and a raising operator, or creator. These are (pseudo-)differential operators with linear symbols which are  called \textit{ladder operators}. Let us denote by
$$
\hat{z} = {\hat{p} \choose \hat{q} }
$$
the semiclassical quantization of the momentum-position vector $z = (p,q)$. Precisely:
$$
\hat{p} \psi(x) = -i \hbar \nabla_x \psi(x), \quad \hat{q} \psi(x) = x \psi(x).
$$
\begin{definition}[Ladder operators]
Let $l \in \mathbb{C}^d \oplus \mathbb{C}^d$. We set
\begin{equation}
A[l] := \frac{i}{\sqrt{2 \hbar}} l \cdot \Omega \hat{z}, \quad A^\dagger[l] := - \frac{i}{\sqrt{2\hbar}} \overline{l} \cdot \Omega \hat{z}.
\end{equation}
$A[l]$ is called lowering operator, while $A^\dagger[l]$ is called raising operator.
\end{definition}
Let $Z$ be a normalized Lagrangan frame with columns $l_1, \ldots, l_d$, we will denote $A[Z]$ and $A^\dagger[Z]$ the vectors of annihilation and creation operators, respectively:
\begin{align*}
A[Z] & := (A[l_1], \ldots, A[l_d])^T = \frac{i}{\sqrt{2\hbar}} Z^T \Omega \hat{z}, \\[0.2cm]
A^\dagger[Z] & := (A^\dagger[l_1], \ldots, A^\dagger[l_d])^T = -\frac{i}{\sqrt{2\hbar}} Z^* \Omega \hat{z}.
\end{align*}
For any multi-index $\alpha \in \mathbb{N}^d_0$, we also set
\begin{align*}
A_\alpha[Z] & := A[l_1]^{\alpha_1} \cdots A[l_d]^{\alpha_d}, \\[0.2cm]
A^\dagger_\alpha[Z] & := A^\dagger[l_1]^{\alpha_1} \cdots A^\dagger[l_d]^{\alpha_d}.
\end{align*}
One can moreover center the ladder operators above on different points of the phase-space by considering its conjuation by the Heisenberg-Weyl translation operator.

\begin{definition}
\label{d:heisenberg_weyl}
The Heisenberg-Weyl translation operator $T[z]$ is defined by:
$$
T[z] := \exp \left( -\frac{i}{\hbar} z \cdot \Omega \hat{z} \right), \quad z = q + ip \in \mathbb{C}^{d}.
$$
More precisely, the operator $T[z]$ acts on $\psi \in L^2(\R^d)$ as
$$
T[z]\psi(x) = e^{\frac{i}{\hbar} p \cdot \left( x - \frac{q}{2} \right)} \psi(x-q).
$$
\end{definition}
We also define the centered ladder operators by
\begin{equation}
A[Z,z] := \frac{i}{\sqrt{2\hbar}} Z^T \Omega(\hat{z} - z), \quad A^\dagger[Z,z] = -\frac{i}{\sqrt{2 \hbar}} Z^* \Omega (\hat{z}-\overline{z}).
\end{equation}
It follows easily that conjugation of the ladder operators by the Weyl Heisenberg-Weyl translation operator, just translates its center:
$$
T[w] A[Z,z] T[w]^* = A[Z,z+w], \quad T[w] A^\dagger[Z,z] T[w]^* = A^\dagger[Z,z+w].
$$
Using the Heisenberg-Weyl translation operator, we also define the centered ladder operators:
\begin{align}
\label{e:conjugation_to_center}
A_\alpha[Z,z] := T[z] A_\alpha[Z] T[z]^*, \quad A^\dagger_\alpha[Z,z] := T[z] A^\dagger_\alpha[Z] T[z]^*.
\end{align}

We next give the construction of the ground-state, or coherent Hagedorn wave packet with Lagrangian frame $Z$ and center $z \in \R^{2d}$:

\begin{lemma}{\cite[Lemma 3.6]{Schubert18}} Let\footnote{We use the notation $Z = (\mathbf{P},\mathbf{Q})^\mathfrak{t} := { \mathbf{P} \choose \mathbf{Q}}$ to write the Lagrangina frame $Z$ as a block matrix.} $Z = (\mathbf{P},\mathbf{Q})^\mathfrak{t} \in \mathbb{C}^{2d \times d}$ be a normalized Lagrangian frame and let $z =q+ip \in \mathbb{C}^d$. Then the matrices $\mathbf{Q}, \mathbf{P} \in \mathbb{C}^{d\times d}$ are invertible and
$$
\operatorname{Im} (\mathbf{P}\mathbf{Q}^{-1}) = (\mathbf{Q}\mathbf{Q}^*)^{-1} > 0.
$$
In particular, for $\hbar > 0$,
\begin{equation}
\label{e:coherent_state}
\varphi_0^\hbar[Z,z](x) := \frac{1}{(\pi \hbar)^{\frac{d}{4}}} \det (\mathbf{Q})^{-\frac{1}{2}} \exp \left( \frac{i}{2\hbar} \mathbf{P} \mathbf{Q}^{-1} (x-q) \cdot (x-q) + \frac{i}{\hbar} p\cdot \left(x-\frac{q}{2} \right) \right)
\end{equation}
is a square integrable function with $\Vert \varphi_0^\hbar[Z,z] \Vert_{L^2(\R^d)} = 1$. Moreover, the matrix $B := \mathbf{P} \mathbf{Q}^{-1}$ belongs to the Siegel upper half-space, namely, the space of complex symetric matrices with positive definite imaginary part.
\end{lemma}

The function $\varphi_0^\hbar[Z,z](x)$ given by \eqref{e:coherent_state} is called Hagedorn coherent state.
\begin{definition}
Let $\alpha \in \mathbb{N}^d$, $Z$ be a normalized Lagrangian frame and and $z \in \mathbb{C}^d$. The $\alpha$-Hagedorn excited state is defined by
\begin{equation}
\varphi_\alpha^\hbar[Z,z](x) = \frac{1}{\sqrt{\alpha!}} A_\alpha^\dagger[Z,z] \varphi_0^\hbar[Z,z](x).
\end{equation}
We will denote by $\varphi_\alpha^\hbar[Z] := \varphi_\alpha^\hbar[Z,0]$ the Hagedorn state centered at $z = 0$, for $\alpha \in \mathbb{N}^d$.
\end{definition}

As we have already anticipate, the family of Hagedorn excited states form an orthonormal basis of $L^2(\R^d)$:

\begin{lemma}{\cite[Thm. 3.7]{Schubert18}}
The set $\{ \varphi_\alpha[Z,z] \}_{\alpha \in \mathbb{N}^d}$ is an orthonormal basis of $L^2(\R^d)$.

\end{lemma}

Moreover, the Hagedorn excited states are Hermite-type functions with polynomial prefactor given by a recurrence relation involving the Lagrangian frame $Z$:

\begin{lemma}{\cite[Prop. 4]{Dietert17}}
\label{l:polynomial_prefactor}
Let $Z = (\mathbf{P},\mathbf{Q})^\mathfrak{t}$ be a normalized Lagrangian frame and let $z = q + ip \in \mathbb{C}^d$. Then, for every $\alpha \in \mathbb{N}^d$,
$$
\varphi_\alpha^\hbar[Z,z](x) = \frac{1}{\sqrt{2^{\vert \alpha \vert} \alpha !}} p_\alpha^M \left( \frac{ \mathbf{Q}^{-1} (x - q) }{\sqrt{\hbar}} \right) \varphi_0^\hbar[Z,z](x),
$$
where the polynomials $\{p^M_\alpha \}_{\alpha \in \mathbb{N}^d}$ are recursively defined by $p_0^M = 1$, and
\begin{equation}
\label{e:recurrence_relation}
p^M_{\alpha + e_j}(x) = 2x_j p_\alpha^M(x) - 2 e_j \cdot M \nabla p^M_\alpha(x),
\end{equation}
where $M = \mathbf{Q}^{-1} \overline{\mathbf{Q}}$.
\end{lemma}

We next prove a technical lemma, which will be used later on, providing an estimate for the $L^1(\R^d)$ norm of the Fourier transform of $\varphi_\alpha^\hbar[Z]$.

\begin{lemma}
\label{e:estimate_L1_norm}
Let $Z$ be a normalized Lagrangian frame and let $ v \in \R^{d} \setminus \{ 0 \}$. Then there exists a positive constant $C = C(Z,v) > 0$ such that, for every $\alpha \in \mathbb{N}^d$,
$$
\int_\R \big \vert \widehat{\varphi}_{\alpha}^1[Z] (r v )  \big \vert dr \leq C^{\vert \alpha \vert }.
$$
\end{lemma}

\begin{remark}
\label{r:with derivatives}
The same estimate holds for $\partial^\gamma \widehat{\varphi}_{\alpha}^1[Z] (r v )$, with $\vert \gamma \vert \leq N$, for $C = C(N,Z,v) > 0$.
\end{remark}

\begin{proof}
By Lemma \ref{l:polynomial_prefactor}, we have that
$$
\varphi_{\alpha}^1[Z](x) = \frac{ (\det \mathbf{Q})^{-1/2}}{\pi^{1/2}} \frac{1}{\sqrt{ 2^{\vert \alpha \vert} \alpha!}} \sum_{ \substack{ \beta \leq \alpha \\ \vert \beta \vert \equiv \vert \alpha \vert \; ( \text{mod } 2)}} b_{\alpha \beta}[Z] (\mathbf{Q}^{-1} x)^\beta \exp \left( i \mathbf{P}\mathbf{Q}^{-1} x \cdot x \right),
$$
where the coefficients $b_{\alpha \beta} = b_{\alpha \beta}[Z]$ are given by recursive relations: $b_{00} = 1$ and
\begin{align*}
b_{\alpha+e_j,\beta} & = 2 b_{\alpha, \beta-e_j} - 2 \sum_{k=1}^{2d} M_{kj} (\beta_k  +  1) b_{\alpha,\beta + e_k}, \quad j = 1, \ldots, 2d,
\end{align*}
which is the coefficient version of \eqref{e:recurrence_relation}.
We first show that, under the more general recurrence relation: $b_{00} = 1$ and
\begin{align}
\label{e:generalized_recurrence}
b_{\alpha+e_j,\beta} & = 2 \sum_{k=1}^{2d} N_{kj}  b_{\alpha, \beta-e_k} - 2 \sum_{k=1}^{2d} M_{kj} (\beta_k  +  1) b_{\alpha,\beta + e_k}, \quad j = 1, \ldots, 2d,
\end{align}
for some $M,N \in \mathbb{C}^{d\times d}$, and denoting $m_d = 2d \cdot \sup_{j,k}\{ \vert  M_{kj} \vert , \vert N_{kj} \vert \}$, one has:
\begin{equation}
\label{e:coefficient_estimate}
\big \vert b_{\alpha \beta} \vert \leq  \frac{  m_d^{\vert \alpha \vert}  \vert \alpha \vert!}{ \left( \frac{ \vert  \alpha \vert  -  \vert\beta \vert }{2} \right)! \beta!} .
\end{equation}
To this aim, we proceed by induction. The estimate for $b_{00} = 1$ is trivial. Moreover, using the induction hypothesis we get
\begin{align*}
\vert b_{\alpha+e_j,\beta} \vert & \leq 2 \sum_{k=1}^{2d} \vert N_{kj} \vert \vert b_{\alpha,\beta-e_k} \vert + 2 \sum_{k=1}^{2d} \vert  M_{kj} \vert ( \beta_k  + 1 ) \vert b_{\alpha,\beta + e_k} \vert \\[0.2cm] 
 & \leq 2 m_d^{\vert \alpha \vert } \sup_{j,k}\{ \vert  M_{kj} \vert , \vert N_{kj} \vert \} \left( \sum_{k=1}^{2d} \frac{ \beta_k \vert \alpha \vert !}{  \left( \frac{ \vert \alpha \vert - \vert \beta \vert + 1 }{2} \right)!\,  \beta!} +   \frac{  (\beta_k + 1) \vert \alpha \vert!}{  \left( \frac{ \vert \alpha \vert - \vert \beta \vert -1 }{2} \right)! (\beta  +e_k)!} \right) \\[0.2cm]
 & \leq   m_d^{\vert \alpha \vert +1 } \sup_{k \in \{1, \ldots, 2d \}} \frac{1}{\left( \frac{ \vert \alpha \vert - \vert \beta  \vert + 1}{2} \right)! \beta!}  \big( \beta_k \vert \alpha \vert! + ( \vert \alpha \vert - \vert\beta \vert + 1 ) \vert \alpha \vert! \big) \\[0.2cm]
 & \leq \frac{ m_d^{\vert \alpha \vert + 1}  \vert \alpha + e_j \vert!}{\left( \frac{\vert \alpha \vert - \vert \beta \vert + 1}{2} \right)!  \beta !},
\end{align*}
then the claim follows. 

On the other hand, taking the Fourier transform of $\varphi_\alpha^1[Z]$ and using \eqref{e:recurrence_relation}, we see that 
$$
\widehat{\varphi}^1_\alpha[Z](\xi) = \frac{1}{\sqrt{2^{\vert \alpha \vert} \alpha!}} q_\alpha^M( \xi) \widehat{\varphi}_0^1[Z](\xi),
$$
where the polynomials $q_\alpha^M$ are defined by the following recurrence relation: $q_0^M = 1$ and
$$
q^M_{\alpha + e_j}(\xi) = e_j \cdot 2i\mathbf{Q}^{-1} \nabla q_\alpha^M(\xi) - e_j \cdot 2 i M \mathbf{Q}^{-1} \xi q_\alpha^M(\xi).
$$
Therefore, denoting
$$
q^M_\alpha(\xi) = \sum_{\beta \leq \alpha} \hat{b}_{\alpha \beta} \xi^\beta,
$$
we see that the coefficients $\hat{b}_{\alpha \beta}$ are defined by: $\hat{b}_{00} =1$ and
$$
\hat{b}_{\alpha,\beta+ e_j} = 2i \sum_{k=1}^{2d} \mathbf{Q}_{kj}^{-1} (\beta_k +1) \hat{b}_{\alpha, \beta+e_k} - 2i \sum_{k=1}^{2d} (M\mathbf{Q}^{-1})_{kj} \hat{b}_{\alpha,\beta-e_k}.
$$
In particular, the coefficients $\hat{b}_{\alpha \beta}$ satisfy a recurrence relation as \eqref{e:generalized_recurrence}. Therefore, using that $B = \mathbf{P} \mathbf{Q}^{-1}$ belongs to the Siegel upper half-space, we obtain the existence of $C = C(Z) > 0$ such that
$$
 \big \vert \widehat{\varphi}_{\alpha}^1[Z] (r v )  \big \vert \leq \frac{1}{\sqrt{ 2^{\vert \alpha \vert} \alpha!}}  \sum_{ \substack{\beta \leq \alpha \\ \vert \beta \vert \equiv \vert \alpha \vert \; (\text{mod } 2)}} \vert \hat{b}_{\alpha \beta} \vert \vert rv \vert^{\vert \beta \vert} e^{- C \vert rv \vert^2},
$$
and thus, using \eqref{e:coefficient_estimate} for the coefficients $\hat{b}_{\alpha \beta}$, we get
$$
\int_\R \big \vert \widehat{\varphi}_{\alpha}^1[Z] (r v )  \big \vert dr \leq  \frac{C^{\vert \alpha \vert}}{\sqrt{ 2^{\vert \alpha \vert} \alpha!}}  \sum_{\beta \leq \alpha} \frac{ \vert \alpha \vert!}{ \left( \frac{ \vert \alpha \vert - \vert \beta \vert}{2} \right)! \, \beta !} \Gamma \left( \frac{ \vert \beta \vert +1 }{2} \right).
$$ 
Finally, using repeatedly the following standard properties of the Gamma function:
\begin{align*}
\frac{2^{2x-1}}{\sqrt{\pi}} \Gamma(x)^2&  \leq  \Gamma(2x)  < x^{\frac{1}{2}} \Gamma(x)^2  2^{2x-1}, \quad x > 0, \\[0.2cm]
\Gamma \left( \frac{x_1 + x_2}{2} \right) & \leq \Gamma(x_1)^{1/2} \Gamma(x_2)^{1/2}, \quad x_1,  \, x_2 >0,
\end{align*}
where the first one is consequence of Legendre's duplication formula and Gautschi's inequality, while the second one is consequence of Jensen's inequality, one can show that $\alpha ! \geq \vert \alpha \vert! C_d^{\vert \alpha \vert}$ { (notice that this inequality also holds by the multinomial expansion $d^{n} = \sum_{\vert \alpha \vert = n} {\vert \alpha \vert \choose \alpha_1 \cdots \alpha_d}$ so that $C_d = d^{-1}$)}, and moreover,
$$
\int_\R \big \vert \widehat{\varphi}_{\alpha}^1[Z] (r v )  \big \vert dr \leq C^{\vert \alpha \vert} \sum_{\beta \leq \alpha}  {\vert \alpha \vert \choose \vert \beta \vert}^{1/2} \leq C^{\vert \alpha \vert}.
$$
Then the lemma follows.
\end{proof}

\subsection{Hagedorn wave packets in phase space}

A very important property of Hagedorn wave-packets is that its structure is invariant by the Wigner transform \cite{Dietert17,Troppmann2017}; that is, the Wigner transform of a Hagedorn coherent or excited state is again a coherent or excited state in phase space. 

\begin{definition}
\label{d:wigner_distribution}
Let $\psi, \varphi \in L^2(\R^d)$. The semiclassical (cross) Wigner function $W_\hbar[\varphi, \psi](z)$ is defined by
\begin{equation}
W_\hbar[\varphi, \psi](z) := \frac{1}{(2\pi)^d} \int_{\R^d} e^{i \xi \cdot v} \varphi\left( x - \frac{\hbar v}{2} \right) \overline{\psi\left( x + \frac{\hbar v}{2} \right)} dv, \quad z = (x,\xi) \in \R^{2d}.
\end{equation}
If $\varphi = \psi$, we denote $W_\hbar[\varphi] := W_h[\varphi,\varphi]$.
\end{definition}

Let $Z_1$, $Z_2$ be two normalized Lagrangian frames. It turns out  that the lifted frame $\mathcal{Z} \in \mathbb{C}^{4d\times 2d}$ defined by
$$
\mathcal{Z} := \left( \begin{array}{c} \mathcal{P} \\[0.1cm] \mathcal{Q} \end{array} \right) := \left( \begin{array}{cc}  
-\Omega \overline{Z}_1 &  \Omega Z_2  \\[0.1cm]
\frac{1}{2} \overline{Z}_1 &  \frac{1}{2} Z_2 \end{array} \right),
$$
is again a normalized Lagrangian frame. 

Let us consider
$$
\mathcal{W}_{\alpha,\beta}^\hbar[Z_1,Z_2](z) : = W_\hbar \big[ \varphi_\alpha^\hbar[Z_1], \varphi_\beta^\hbar[Z_2] \big](z).
$$
 By \cite[Prop. 62]{Troppmann2017}, one has:
\begin{equation}
\label{e:ground_state_lift}
\Phi_{(0,0)}^\hbar[\mathcal{Z}](z) := \mathcal{W}_{0,0}^\hbar[Z_1,Z_2](z)=  \frac{1}{(\pi \hbar)^d} \det \big( \Re G \big)^{1/4} e^{-\frac{1}{\hbar} G z \cdot z},
\end{equation}
where $2i G = \mathcal{P} \mathcal{Q}^{-1}$ defines the mixed metric for $Z_1,Z_2$. In particular, if $Z_1 = Z_2$, then $G$ is real and $\det(G) = 1$.

For the excited states, the following holds:
\begin{prop}{\cite[Thm. 6.1]{Troppmann2017}}
\label{p:wave_packets_phase_space}
Let $\alpha,\beta \in \mathbb{N}^d$. Then
$$
\mathcal{W}_{\alpha,\beta}^\hbar[Z_1,Z_2](z) = \frac{1}{\sqrt{\alpha! \beta!}} A_{(\alpha,\beta)}^\dagger[\mathcal{Z}] \Phi_{(0,0)}^\hbar[\mathcal{Z}](z) =: \Phi_{(\alpha,\beta)}^\hbar[\mathcal{Z}](z).
$$
\end{prop}

\section{Propagation of Hagedorn wave-packets}

This section is devoted to describe the properties of Hagedorn wave packets when they propagate through the action of a non-Hermitian operator during a small interval of time. In \cite{Schubert18}, the authors obtain a complete description of the propagation of Hagedorn states $\{ \varphi_\alpha^\hbar[Z,z] \}$ for quadratic operators. This constitutes the heart of our proof, but we also need to obtain estimates for the propagation of a Hagedorn wave packet centered at $z_0$ by the action of a more general operator. We focus on the study of the evolution equation:
\begin{equation}
\label{e:simplified_non_selfadjoint_problem}
 \big( i \hbar \partial_t  + \widehat{P}_\hbar \big) \varphi_\hbar(t,x) = 0, \quad \varphi_\hbar(0,x) = \varphi_0^\hbar[Z_0,z_0](x),
\end{equation}
for small $t \in (-\delta, \delta)$, where $Z_0$ is a given normalized Lagrangian. To this aim, we make the ansatz 
\begin{equation}
\label{e:ansatz}
\varphi_\hbar(t,x) = \sum_{\alpha \in \mathbb{N}^d} c_\alpha(t) \varphi_\alpha[Z_t,z_t](x), \quad 
\end{equation}
with unknowns given by the pair $(Z_t,z_t)$, which controls the evolution of the orthonormal basis $\{ \varphi_\alpha^\hbar[Z_0,z_0] \}$, and the vector of coefficients $\vec{c}(t) = (c_\alpha(t))_{\alpha \in \mathbb{N}^d}$, which corrects the error terms reflecting the particular interaction between excited states caused by the non-quadratic propagation. As we will see below, it is convenient to replace the operator $\widehat{P}_\hbar$ in \eqref{e:simplified_non_selfadjoint_problem} by some approximation $\Op_\hbar(\widetilde{P}(t,\cdot))$ near $z_t$ (see \eqref{e:simplified_non_selfadjoint_problem2} below), to ensure that the solution $\varphi_\hbar(t,x)$ is well defined and the coefficients $c_\alpha(t)$ decay sufficiently fast. As we will see, this will be sufficient to construct our quasimode since $\widehat{P}_\hbar$ and $\Op_\hbar(\widetilde{P})$ coincide microlocally near $z_t$ up to order $N$.

First of all, the center $z_t$ is  given by the following system of differential equations, which couples the evolution of the center $z_t$ with the evolution of the Lagrangian subspace $L_t$ from  $L_0$ via its metric $G_t$ (see \cite{Schubert11} and \cite[Thm. 4.3 and Corol. 4.7]{Schubert18}):
\begin{lemma}
\label{l:center_evolution}
Let $z_0 \in \R^{2d}$ and $P = V + iA$. Then, there exists $\delta > 0$ such that the system of equations
\begin{align}
\label{e:center_evolution}
\dot{z}_t & = - \Omega  \operatorname{Re} \nabla P(z_t) - G_t^{-1} \operatorname{Im} \nabla  P(z_t), \\[0.2cm]
\label{e:metric_evolution}
\dot{G}_t & =  - \operatorname{Re} \partial^2 P(z_t) \Omega G_t + G_t \Omega \operatorname{Re} \partial^2 P(z_t) -  \operatorname{Im} \partial^2 P(z_t) - G_t \Omega  \operatorname{Im} \partial^2 P(z_t) \Omega G_t,
\end{align}
where $\partial^2P$ denotes the Hessian of $P$, has a unique solution for $G_t \vert_{t = 0} = \Id$ and $z_t \vert_{t = 0} = z_0$ for $-\delta \leq t \leq \delta$, such that $G_t$ is real, symplectic, symmetric and positive definite.
\end{lemma} 

The Ricatti equation \eqref{e:metric_evolution} gives the evolution of the metric $G_t$ for the complex structure associated with the Lagrangian subspace $L_t = S_tL_0$, which evolves according with the complex symplectic matrix $S_t$ obeying:
\begin{equation}
\label{e:matrix_evolution}
\dot{S}_t = - \Omega \, \partial^2 P(z_t) S_t, \quad S_0 = \Id_{2d},
\end{equation}
for $- \delta \leq t \leq \delta$. The vector field giving the expression for $\dot{z}_t$ in \eqref{e:center_evolution} is the sum of the Hamiltonian vector field  $-\Omega \nabla \Re P(z_t)$  and the friction term $-G_t^{-1} \nabla \Im P(z_t)$, which pushes the particle outside the Hamiltonian classical orbit. In the case in which $z_0$ is a non-damped point for $P$, that is $\Im P(z_0) = 0$, then the friction term is activated at the damped region $\{ \Im P > 0 \}$, and its main effect consists in pushing the particle towards the non-damped point $z_0$. 

Once we have given the orbit for the center $z_t$, the evolution of the Lagrangian frame $Z_t$ is obtained easily from the evolution of the symplectic matrix $S_t \in \mathbb{C}^{2d \times 2d}$ obeying \eqref{e:matrix_evolution}. Precisely, defining the Hermitian and positive definite matrix (see \cite[Sect. 4.3]{Schubert18}):
\begin{equation}
\label{e:normalizing_matrix}
N_t := \left( \frac{1}{2i} (S_t Z_0)^* \Omega ( S_t Z_0) \right)^{-1/2},
\end{equation}
then $Z_t$ is given by the normalized Lagrangian frame
$$
Z_t := S_t Z_0 N_t.
$$

%\begin{definition}
%Let $s > 0$. We say that $a \in L^1(\R^{2d})$ belongs to the space $\mathcal{A}_s(\R^{2d})$ if
%$$
%\Vert a \Vert_s := \int_{\R^{2d}} \vert \widehat{a}(w) \vert e^{s \vert w \vert} dw < + \infty,
%$$
%where $\widehat{a}$ denotes the Fourier transform and $\vert w \vert$ the Euclidean norm on $\R^{2d}$.

%Let $\rho,s > 0$, we define th space $\mathcal{A}_{\rho,s}(\R^{2d})$ of functions $a \in L^1(\R^{2d})$ such that
%$$
%\Vert a \Vert_{s,\rho} := \frac{1}{(2\pi)^d} \sum_{k \in \mathbb{Z}^d} \Vert a_k \Vert_s e^{\rho \vert k \vert^{3/2}} < + \infty,
%$$
%where
%$$
%a_k(z) = \int_{\mathbb{T}^d} a \circ \Phi_\tau^H(z) e^{-i k \cdot \tau} d\tau, \quad k \in \mathbb{Z}^d.
%$$
%\end{definition}

In order to compute the vector of coefficients $\vec{c}(t) = (c_\alpha(t))$, and hence the solution  $\varphi_\hbar(t,x)$ to \eqref{e:simplified_non_selfadjoint_problem}, we first compute the solution to the Schrödinger  equation given by the quadratic approximation of $P(z)$ near $z_t$. To this aim, we expand $P$ by Taylor near $z_t$ as $P(z) = P_{2}(t,z) + P_N(t,z) + R_{N}(t,z)$, where
\begin{align}
\label{e:taylor_2}
P_{2}(t,z) & = P(z_t) + (z- z_t) \cdot \nabla P(z_t) + \frac{1}{2} (z-z_t) \cdot \partial^2P(z_t) (z-z_t), \\[0.2cm]
\label{e:taylor_N}
P_N(t,z) & = \sum_{3 \leq \vert \beta \vert \leq N} \frac{ \partial^\beta P(z_t)}{\beta!} (z-z_t)^\beta, \\[0.2cm]
\label{e:taylor_R}
R_{N}(t,z) & = \sum_{\vert \beta \vert ={ N+1}} \frac{\vert \beta \vert}{\beta!} (z-z_t)^\beta \int_0^1 (1- s)^{\vert \beta \vert-1} D^\beta P\big( z_t + s(z-z_t)\big) ds.
\end{align}
Notice that the evolution of $(z_t,G_t)$ only depends on $P_{2}(t,z)$ via the equations \eqref{e:center_evolution} and \eqref{e:metric_evolution}. 

We next truncate the polynomial symbol $P_N(t,z)$ near $z_t$ to ensure that the solution $\varphi_\hbar(t,x)$ is be well defined. Let $\chi \in \mathcal{C}_c^\infty(\R)$ be a cut-off function equal to one near zero, we consider the truncated symbol $\chi\big( \vert F_t^{-1} (z - z_t) \vert^2 \big) P_{N} (t,z)$ near $z_t$, where $F_t$ is the symplectic matrix associated with the normalized Lagrangian frame $Z_t$. Moreover, it is convenient to approximate this symbol by a further  one that fits with anti-Wick quantization, hence we can use the Bargmann space to compute the matrix elements in an easier way. Precisely, we define:
\begin{equation}
\label{e:truncated_P_N}
\widetilde{P}_N(t,z) := \sigma^{\operatorname{AW}}_{Z_t,N} \big( \chi\big( \vert F_t^{-1} (z - z_t) \vert^2 \big) P_{N} (t,z)\big) ,
\end{equation}
where $\sigma^{\operatorname{AW}}_{Z_t,N}(a) = a + O(\hbar)$ denotes the anti-Wick approximation of $a$ of order $N$, defined by \eqref{e:anti-Wick_approximation} below.  We then replace  the evolution problem \eqref{e:simplified_non_selfadjoint_problem} by: 
\begin{equation}
\label{e:simplified_non_selfadjoint_problem2}
 \big( i \hbar \partial_t  + \Op_\hbar(\widetilde{P}(t,z) \big) \varphi_\hbar(t,x) = 0, \quad \varphi_\hbar(0,x) = \varphi_0^\hbar[Z_0,z_0](x),
\end{equation}
for small $t \in [-\delta, \delta]$, where:
$$
\widetilde{P}(t,z) := P_2(t;z) + \widetilde{P}_N(t,\cdot) * W_\hbar[\varphi_0^\hbar[Z_t]](z).
$$ 
The convolution with $W_\hbar[\varphi_0^\hbar[Z_t]]$ connects Weyl quantization with anti-Wick quantization (see Section \ref{s:matrix_elements} below), and allows us to compute the matrix elements of 
$$
\Op_{\hbar,Z_t}^{\operatorname{AW}}(\widetilde{P}_{N}) := \Op_\hbar( \widetilde{P}_N(t,\cdot) * W_\hbar[\varphi_0^\hbar[Z_t]])
$$ 
in the Bargmann space. More precisely, one has 
$$
\widetilde{P}_N * W_\hbar[\varphi_0^\hbar[Z_t]](z) = \chi\big( \vert F_t^{-1} (z - z_t) \vert^2 \big) P_N(t,z) + O(\hbar^{N+1})
$$ 
that is, we take an approximation of $\chi P_N$ up to order $N$ by an anti-Wick symbol. 

To  study the evolution equation \eqref{e:simplified_non_selfadjoint_problem2}, we make the ansatz \eqref{e:ansatz} and first describe the evolution by the quadratic part $P_2(t,z)$. Later, in Section \ref{s:propagation}, we compare it with the whole evolution by $\widetilde{P}$, obtaining our propagation result stated in Proposition \ref{p:final_estimate_coefficients}. The main idea comes from the works \cite{Robert12,RobertBook}.

\subsection{Quadratic evolution}

In this section we focus on the quadratic equation:
\begin{equation}
\label{e:quadratic_equation}
\big( i \hbar \partial_t  + \Op_\hbar(P_{2}(t,z)) \big) \varphi_{\hbar}(t,x) = 0, \quad \varphi_{\hbar}(0,x) = \sum_{\alpha \in \mathbb{N}^d} c_\alpha \varphi_\alpha^\hbar[Z_0,z_0].
\end{equation}
Our aim is to give a suitable differential equation for the vector of coefficients $\vec{c}(t) = (c_{\alpha}(t))$ such that
$$
\varphi_{\hbar}(t,x) = \sum_{\alpha \in \mathbb{N}^d} c_{\alpha}(t) \varphi_\alpha^\hbar[Z_t,z_t](x), \quad t \in [-\delta,\delta].
$$
To ensure the existence of the solution $\varphi_\hbar(t,x)$, we assume that $\vec{c} = (c_\alpha)_{\alpha \in \mathbb{N}^d} \in \ell_\rho(\mathbb{N}^d)$ (see Appendix \ref{a:evolution_equations}). Then we will see that $\vec{c}(t) \in \ell_{\rho - \sigma}(\mathbb{N}^d)$ for some $0 < \sigma < \rho$ and $t \in [-\delta,\delta]$.
\begin{prop}
\label{p:coefficients_quadratic_part}
The vector of coefficients $\vec{c}(t)$ satisfies the differential equation:
$$
\left \lbrace \begin{array}{ll}
\dot{c}_{\alpha}(t) & \displaystyle  =  \left( \dot{\varrho}_t + \frac{i\dot{\Lambda}_t}{\hbar}  \right) c_\alpha(t) + \sum_{\beta \in \mathbb{N}^d} \kappa_{\alpha \beta}(t) c_\beta(t), \quad \alpha \in \mathbb{N}^d, \\[0.2cm]
c_{\alpha}(0) & = c_\alpha,
\end{array} \right. 
$$
where
\begin{align}
\label{e:beta_t}
\varrho_t & = -\frac{1}{4} \int_0^t \operatorname{tr}\big( G_s^{-1} \operatorname{Im} \partial^2 P(z_s) \big) ds \\[0.2cm]
\label{e:alpha_t}
\Lambda_t & = - \int_0^t \left( \frac{ \dot{p}_s \cdot q_s -  \dot{q}_s \cdot p_s}{2}  - P(z_s) \right) ds,
\end{align}
and the matrix elements $(\kappa_{\alpha \beta}(t))$ satify $\kappa_{00}(t) = 0$, $\kappa_{\alpha \beta}(t) = 0$ unless $\vert \alpha \vert \geq \vert \beta \vert$ and $\vert \alpha - \beta \vert \leq 2$, and there exists $C > 0$ such that:
\begin{equation}
\label{e:estimate_matrix_elements}
\sup_{-\delta \leq t \leq \delta} \vert \kappa_{\alpha \beta}(t) \vert \leq C \vert \alpha \vert.
\end{equation}
\end{prop}

\begin{corol}
\label{c:kappa_well_defined}
Let $\vec{c} \in \ell_\rho(\mathbb{N}^d)$ for some $\rho > 0$. Then there exists $0 < \sigma < \rho$ and $\delta = \delta(\rho,\sigma) > 0$ such that $\vec{c}(t) \in \mathcal{C}([-\delta,\delta];\ell_{\rho - \sigma}(\mathbb{N}^d))$.
\end{corol}

\begin{proof}
Let $\mathcal{K}(t)$ be the operator defined by 
$$
\left( \mathcal{K}(t) \vec{c} \right)_\alpha =  \sum_{\beta \in \mathbb{N}^d} \kappa_{\alpha \beta}(t) c_\beta, \quad \alpha \in \mathbb{N}^d.
$$
Then, by \eqref{e:estimate_matrix_elements},
\begin{align*}
\Vert \mathcal{K}(t) \vec{c} \Vert_{\rho - \sigma} & = \sum_{\alpha \in \mathbb{N}^d} \vert ( \mathcal{K}(t) \vec{c})_\alpha \vert e^{(\rho - \sigma) \vert \alpha \vert} \\[0.2cm]
 &  \leq  C \sum_{ \alpha \in \mathbb{N}^d} \sum_{\vert \alpha - \beta \vert \leq 2}  \vert \alpha \vert \vert c_\beta \vert e^{(\rho - \sigma)\vert \alpha \vert} \\[0.2cm]
 & \leq C_\rho \sum_{\beta \in \mathbb{N}^d} \vert \beta \vert e^{- \sigma \vert \beta \vert} \vert c_\beta \vert e^{\rho \vert \beta \vert} \\[0.2cm]
 & \leq C_\rho \sup_{r \geq 0} r e^{- \sigma r} \sum_{\beta \in \mathbb{N}^d} \vert c_\beta \vert e^{\rho \vert \beta \vert} \\[0.2cm]
 & = \frac{C_\rho}{e \sigma} \Vert \vec{c} \Vert_{\rho}
\end{align*}
This implies that $\mathcal{K} \in \mathcal{C}([-\delta,\delta]; \mathscr{D}_\rho)$. Applying  Lemma \ref{l:homogeneous_evolution}, the claim holds.
\end{proof}

From the proof of Proposition \ref{p:coefficients_quadratic_part} together with \cite[Thm. 4.5]{Schubert18} one obtains the explicit expression for the coefficients 
$$
c_\alpha(t) =  e^{\frac{i}{\hbar} \Lambda_t(z_0) + \varrho_t} \sum_{\vert \alpha \vert \leq \vert \beta \vert} a_{\alpha \beta}(t) c_\beta = e^{\frac{i}{\hbar} \Lambda_t(z_0) + \varrho_t} \mathcal{A}(t) \vec{c}, 
$$
where $\mathcal{K}(t) = \dot{\mathcal{A}}(t) \mathcal{A}(t)^{-1}$. Precisely, one finds (see \cite[Thm. 4.9]{Schubert18}):
\begin{corol}
The solution $\varphi_\hbar(t,x)$ is given by:
\begin{align*}
 \varphi_{\hbar}(t,x) & =  \sum_{\beta \in \mathbb{N}^d } \frac{e^{\frac{i}{\hbar} \Lambda_t(z_0) + \varrho_t}}{\sqrt{\beta!}} c_\beta p_\beta\left( \sqrt{\frac{2}{\hbar}} N_t Q_t^{-1}(x-q_t) \right) \varphi_0^\hbar[Z_t,z_t](x) \\[0.2cm]
 & =: \sum_{\beta \in \mathbb{N}^d } e^{\frac{i}{\hbar} \Lambda_t(z_0) + \varrho_t} c_\beta \sum_{\vert \alpha \vert \leq \vert \beta \vert} a_{\alpha \beta}(t) \varphi_\alpha^\hbar[Z_t,z_t](x),
\end{align*}
where $z_t = (q_t,p_t)$ is given by Lemma \ref{l:center_evolution}, and the Hermite-type polynomials $p_\alpha = p_\alpha(t)$ are explicit and given by the recurrence relation:
$$
p_0(x;M_t) = 1, \quad p_{\alpha + e_j}(x,M_t) = x_j p_\alpha(x,M_t) - e_j \cdot M_t \nabla p_\alpha(x;M_t), \quad j = 1, \ldots, d,
$$
with
$$
M_t = \frac{1}{4}(S_t \overline{Z}_0)^T G_t (S_t \overline{Z}_0) + N_t \mathbf{Q}_t^{-1} \overline{\mathbf{Q}}_t \overline{N}_t.
$$
\end{corol}
\begin{remark}
Notice that the main change of the $L^2$-norm of $\varphi_\hbar(t,x)$ is given by the diagonal term $e^{\frac{i}{\hbar} \Lambda_t(z_0) + \varrho_t}$.
\end{remark}

\begin{proof}[Proof of Proposition \ref{p:coefficients_quadratic_part}]

We first focus on the propagation of the coherent state $\varphi_0^\hbar[Z_0,z_0]$. We claim that:
\begin{equation}
\label{e:quadratic_equation_0}
\big( i \hbar \partial_t + \Op_\hbar(P_{2}) \big) \varphi_0^\hbar[Z_t,z_t]  = \left(  i \hbar \dot{\varrho}_t - P_2(t,z_t) - \frac{ \dot{q}_t \cdot p_t - \dot{p}_t \cdot q_t}{2} \right) \varphi_0^\hbar[Z_t,z_t].
\end{equation} 
To show this, first notice that, by \cite[Prop. 4.8]{Schubert18}, 
$$
\varphi_0^\hbar[Z_t,z_t] = \det(N_t)^{1/2} \varphi_0^\hbar[S_t Z_0,z_t],
$$
where $\det(N_t)^{1/2} = e^{\varrho_t}$ is given by \eqref{e:beta_t}. We next compute:
\begin{align*}
i \hbar \partial_t \varphi_0^\hbar[Z_t,z_t] & =  i \hbar \big(  \dot{ \varrho}_t  +  \partial_t (\det \mathbf{Q}_t)^{-1/2}  \det \mathbf{Q}_t^{1/2} \big) \varphi_0^\hbar[S_t Z_0,z_t] \\[0.2cm]
 & \quad +  i \hbar \partial_t \left( \frac{i}{2\hbar}(x - q_t) \cdot B_t (x-q_t) + \frac{i}{\hbar} p_t \cdot (x- q_t) \right) \varphi_0^\hbar[S_t Z_0,z_t],
\end{align*}
where $B_t = \mathbf{P}_t \mathbf{Q}_t^{-1}$. By Jacobi's determinant formula $(\partial_t \det \mathbf{Q}_t)/ \det \mathbf{Q}_t = \tr (\partial_t \mathbf{Q}_t \mathbf{Q}_t^{-1})$, we also have
$$
i \hbar  \partial_t (\det \mathbf{Q}_t)^{-1/2}  \det \mathbf{Q}_t^{1/2} = -\frac{i \hbar}{2} \tr ( \partial_t \mathbf{Q}_t \mathbf{Q}_t^{-1}).
$$
Moreover, denoting
$$
 \partial^2 P(z_t) = \left( \begin{array}{cc} P_{pp} & P_{pq} \\[0.2cm]
 P_{qp} & P_{qq}
 \end{array} \right),
$$
and since $S_t$ solves equation \eqref{e:matrix_evolution}, we obtain:
\begin{align*}
 i \hbar \partial_t \left( \frac{i}{2\hbar}(x - q_t) \cdot B_t (x-q_t) + \frac{i}{\hbar} p_t \cdot (x- q_t) \right) & \\[0.2cm]
  & \hspace*{-6.5cm}  =  \frac{ \dot{p}_t \cdot q_t + \dot{q}_t \cdot p_t}{2} - \dot{p}_t \cdot x + \dot{q}_t \cdot B_t (x-q_t) - \frac{1}{2} (x- q_t) \cdot \dot{B_t} (x - q_t) \\[0.2cm]
  & \hspace*{-6.5cm} =  \frac{ \dot{p}_t \cdot q_t + \dot{q}_t \cdot p_t}{2} - \dot{p}_t \cdot x + \dot{q}_t \cdot B_t (x-q_t) \\[0.2cm]
  & \hspace*{-6cm} - \frac{1}{2} (x-q_t) \cdot( P_{qp} B_t + P_{qq} + B_t P_{pq} + B_t P_{pp} B_t)(x- q_t).
\end{align*}
On the other hand, in order to compute $\Op_\hbar(P_{2}) \varphi_0^\hbar[S_tZ_0,z_t]$, we use the definition of $P_2$ and notice that
$$
(\hat{z} - z_t) \cdot \nabla P(z_t) \varphi_0^\hbar[S_tZ_0, z_t] = \big( \nabla_q P(z_t) \cdot (x-q_t)  + \nabla_p P(z_t) \cdot  B_t (x- q_t)   \big) \varphi_0^\hbar[S_t Z_0], z_t],
$$
and
\begin{align*}
\frac{1}{2} \partial^2 P(z_t) (\hat{z} - z_t) \cdot (\hat{z} - z_t) \varphi_0^\hbar[S_tZ_0,z_t]  \\[0.2cm]
 & \hspace*{-3cm}  = -  \frac{i\hbar}{2} \tr (P_{pp} B_t + P_{pq}) \varphi_0^\hbar[S_tZ_0,z_t] \\[0.2cm]
 & \hspace*{-3cm} \quad + \frac{1}{2} (x-q_t) \cdot( P_{qp} B_t + P_{qq} + B_t P_{pq} + B_t P_{pp} B_t)(x- q_t),
\end{align*}
from which we deduce that the quadratic terms in $(x-q_t)$ cancel. Using \eqref{e:center_evolution}, we also see that
\begin{align*}
\dot{p}_t & =  \operatorname{Re} \nabla_q P(z_t) - \mathbf{P}_t \mathbf{P}_t^* \operatorname{Im} \nabla_p P(z_t) - ( \mathbf{P}_t \mathbf{Q}_t^* - i \Id ) \operatorname{Im} \nabla_q P(z_t), \\[0.2cm]
\dot{q}_t & = - \operatorname{Re} \nabla_p P(z_t) - ( \mathbf{Q}_t \mathbf{P}_t^*  + i \operatorname{Id}) \operatorname{Im} \nabla_p P(z_t) - \mathbf{Q}_t \mathbf{Q}_t^* \operatorname{Im} \nabla_q P(z_t).
\end{align*}
Thus
$$
\nabla_q P(z_t) + B_t \nabla_p P(z_t) =  \dot{p}_t - B_t \dot{q}_t,
$$ 
and then the linear terms in $x$ also cancel. Finally, 
$$
- \nabla_q P(z_t) \cdot q_t - B_t \nabla_p P(z_t) \cdot q_t = - q_t \cdot \big( \nabla_q P(z_t) + B_t \nabla_p P(z_t) \big) =  q_t \cdot B_t \dot{q}_t - q_t \cdot \dot{p}_t,
$$
and moreover $\partial_t \mathbf{Q}_t = -P_{pq} \mathbf{Q}_t - P_{pp} \mathbf{P}_t$, hence $\partial_t \mathbf{Q}_t \mathbf{Q}_t^{-1} = - P_{pq} - P_{pp} \mathbf{P}_t \mathbf{Q}_t^{-1}$. Summing up, we get \eqref{e:quadratic_equation_0}. This implies, in particular, that $\kappa_{\alpha 0}(t) =0$ for every $\vert \alpha \vert \geq 0$.

We next look at the evolution of the excited states, which relies on the ladder operators. To this aim, we want to compare the ladder operators $A^\dagger_\alpha[Z_t,z_t]$ and $A_\alpha[Z_t,z_t]$ with $A^\dagger[S_tZ_0,z_t]$ and $A[S_tZ_0,z_t]$. Considering the projection operators onto the Lagrangian spaces $L_t$ and $\overline{L}_t$, we have, by Proposition \ref{p:projections},
$$
\pi_{L_t} = \frac{i}{2} Z_t Z_t^* \Omega^T, \quad \pi_{\overline{L}_t} = - \frac{i}{2} \overline{Z}_t Z_t^T \Omega^T,
$$
and then we can decompose
\begin{align*}
\overline{S}_t Z_0 & =  \pi_{L_t} \overline{S}_tZ_0 + \pi_{\overline{L}_t} \overline{S}_t Z_0, \\[0.2cm]
S_t Z_0 & =  \pi_{L_t} S_t Z_0 + \pi_{\overline{L}_t} S_t Z_0.
\end{align*}
We then have, by \cite[Lemma 4.4]{Schubert18}, the following linear equation for the ladder operator $A^\dagger[\overline{S}_t Z_0]$:
\begin{align*}
A^\dagger [\overline{S}_t Z_0] & = A^\dagger [ \pi_{L_t} \overline{S}_tZ_0] - A[\pi_{L_t} S_t \overline{Z}_0] \\[0.2cm]
 & = A^\dagger[ Z_t C_t] - A[Z_tD_t] \\[0.2cm]
 & = C_t^* A^\dagger[Z_t] - D_t^T A[Z_t],
\end{align*}
where $C_t = \frac{i}{2} Z_t^* \Omega^T \overline{S}_t Z_0$, $D_t = \frac{i}{2} Z_t^* \Omega^T S_t \overline{Z}_0$; and similarly, for $A[S_tZ_0]$:
\begin{align*}
A[S_t Z_0] & = A[\pi_{L_t} S_tZ_0] - A^\dagger[\pi_{L_t} \overline{S}_t \overline{Z}_0] \\[0.2cm]
 & = A[Z_t E_t] - A^\dagger[Z_t F_t] \\[0.2cm]
 & = E_t^T A[Z_t] - F_t^* A^\dagger[Z_t],
\end{align*}
where $E_t = \frac{i}{2} Z_t^* \Omega^T S_t Z_0$ and $F_t = \frac{i}{2} Z_t^* \Omega^T \overline{S}_t \overline{Z}_0$.
%\begin{align*}
%C_t & = \frac{i}{2} Z_t^* \Omega^T \overline{S}_t Z_0,  \\[0.2cm]
%D_t & = \frac{i}{2} Z_t^* \Omega^T S_t \overline{Z}_0, \\[0.2cm]
%E_t & = \frac{i}{2} Z_t^* \Omega^T S_t Z_0, \\[0.2cm]
%F_t & = \frac{i}{2} Z_t^* \Omega^T \overline{S}_t \overline{Z}_0.
%\end{align*}
 In other words,
\begin{equation}
\label{e:matrix_equation_for_ladder_operators}
\left( \begin{array}{c} 
A^\dagger [\overline{S}_t Z_0] \\[0.2cm]
A[S_t Z_0]
\end{array} \right) =
\left( \begin{array}{cc}
C_t^* & - D_t^T \\[0.2cm]
-F_t^* & E_t^T
\end{array} \right) \left( \begin{array}{c} 
A^\dagger [Z_t] \\[0.2cm]
A[Z_t]
\end{array} \right).
\end{equation}
The same equation \eqref{e:matrix_equation_for_ladder_operators} remains valid for the translations $A^\dagger[Z_t,z_t]$ and $A[Z_t,z_t]$ due to the conjugation property \eqref{e:conjugation_to_center}. Notice, moreover, that $C_t = \Id + O(t)$, $E_t = \Id + O(t)$, $D_t = O(t)$, and $F_t = O(t)$, then the matrix of \eqref{e:matrix_equation_for_ladder_operators} is invertible for small $t$. Let us denote this inverse, for $t \in [-\delta,\delta]$, by:
\begin{equation}
\label{e:inverse_matrix}
\left( \begin{array}{cc}
C_t^* & - D_t^T \\[0.2cm]
-F_t^* & E_t^T
\end{array} \right)^{-1} =: \left( \begin{array}{cc}
X_t & Y_t \\[0.2cm]
V_t & W_t
\end{array} \right).
\end{equation}
Next, consider the derivative corresponding to the ladder operators:
$$
i \hbar \partial_t \big( A^\dagger_\alpha[Z_t,z_t] \varphi_0^\hbar[Z_t,z_t](x) \big).
$$
To compute this derivative, let $t \mapsto w_t \in \mathbb{C}^{2d}$ be the complex curve satisfying
$$
\dot{w}_t = - \Omega \nabla P(z_t), \quad w_0 = z_0.
$$
Using that $J_t \Omega = \Omega^T G_t \Omega = G_t^{-1}$, it follows that $z_t$ is given by the real projection of $w_t$ via the complex structure $J_t$: 
$$z
_t =  \Re w_t + J_t \Im w_t.
$$ 
Then, by \cite[Thm. 3.12]{Schubert18}, we have
$$
A^\dagger[\overline{S}_tZ_0,z_t] = A^\dagger[\overline{S}_tZ_0, w_t], \quad A[S_tZ_0,z_t] = A[S_tZ_0, w_t].
$$
On the other hand, using the symbolic calculus for pseudodifferential operators, we observe that
$$
i\hbar \partial_t A^\dagger[\overline{S}_t l_0, w_t] = - [ \Op_\hbar(P_{2}), A^\dagger[S_tl_0, w_t] ], \quad i \hbar \partial_t A[S_t l_0, w_t] = -[ \Op_\hbar(P_{2}), A[S_tl_0, w_t]].
$$ 
Therefore, by $\eqref{e:matrix_equation_for_ladder_operators}$ and \eqref{e:inverse_matrix}, 
\begin{align*}
  i \hbar \partial_t A^\dagger[l^j_t,z_t]  & = i \hbar \partial_t \big( e_j \cdot A^\dagger[Z_t,z_t] \big) \\[0.2cm]
 & = i \hbar \partial_t \big( e_j \cdot ( X_t A^\dagger[\overline{S}_tZ_0,z_t] + Y_t A[S_t Z_0,z_t]) \big) \\[0.2cm]
 & =   i \hbar e_j \cdot ( \dot{X}_t A^\dagger[\overline{S}_tZ_0,z_t] + \dot{Y}_t A[S_t Z_0,z_t]) +   [ \Op_\hbar(P_{2}), A^\dagger[l_t^j,z_t] ] \\[0.2cm]
 & = i \hbar \mathcal{L}_t^j \big( A^\dagger[Z_t,z_t], A[Z_t,z_t] \big) + [ \Op_\hbar(P_{2}), A^\dagger[l_t^j,z_t] ],
\end{align*}
where the tensor term $\mathcal{L}_t^j \big( A^\dagger[Z_t,z_t], A[Z_t,z_t] \big)$ is given by
\begin{align*}
\mathcal{L}_t^j \big( A^\dagger[Z_t,z_t], A[Z_t,z_t] \big) & =  e_j \cdot \big( (\dot{X}_t C_t^* - \dot{Y}_tF_t^*) A^\dagger[Z_t,z_t] + ( \dot{Y}_t E_t^T - \dot{X}_t D_t^T) A[Z_t,z_t] \big)   \\[0.2cm]
 & =  \sum_{k=1}^d \nu_{jk}^1(t) A^\dagger [l_t^k,z_t] + \nu_{jk}^2(t) A[l_t^k,z_t],
\end{align*}
for certain $\nu_{jk}^l(t) \in \mathbb{C}$,  $l \in \{ 1,2 \}$. Using the definition, $A^\dagger_\alpha[Z_t,z_t]  = A^\dagger[l_t^1,z_t]^{\alpha_1} \cdots A^\dagger[l_t^d,z_t]^{\alpha_d}$, we obtain by the product rule that
\begin{align*}
i \hbar \partial_t A_\alpha^\dagger[Z_t,z_t] + \big[  \Op_\hbar(P_{2}) , A_\alpha^\dagger[Z_t,z_t] \big] & \\[0.2cm]
 & \hspace*{-4cm} =  \sum_{j=1}^d \alpha_j A^\dagger[l_t^j,z_t]^{\alpha_j -1} \left( \sum_{k=1}^d \nu_{jk}^1(t) A^\dagger[l_t^k,z_t] + \nu_{jk}^2(t) A[l_t^k,z_t] \right) \prod_{j \neq j'} A^\dagger[l_t^{j'},z_t] \\[0.2cm]
 & \hspace*{-4cm} =  \sum_{ \substack{\vert \alpha \vert \geq \vert \beta \vert \\ \vert \alpha - \beta \vert \leq 2}} \kappa_{\alpha \beta}(t) A_\alpha^\dagger[Z_t,z_t],
\end{align*}
where $\vert \kappa_{\alpha \beta}(t) \vert \leq C \vert \alpha \vert$, for some constant $C$ uniformly bounded for $-\delta \leq t \leq \delta$.  Finally, the differential equation for the coefficients $c_{\alpha}(t)$ is given by:
$$
\left \lbrace \begin{array}{ll}
\dot{c}_{\alpha}(t) & \displaystyle  = \left(  \dot{\varrho}_t + \frac{i\dot{\Lambda}_t}{\hbar} \right) c_{\alpha}(t) + \sum_{\substack{ \vert \alpha \vert \geq \vert \beta \vert \\ \vert \beta - \alpha \vert \leq 2}} \kappa_{\alpha \beta}(t) c_{\beta}(t), \quad \alpha \in \mathbb{N}^d, \\[0.2cm]
c_{\alpha}(0) & = c_\alpha.
\end{array} \right.
$$
\end{proof}

\subsection{Matrix elements}
\label{s:matrix_elements}

In this section, we compute the matrix elements of the remainder term $\Op_\hbar(\widetilde{P}_N * W_\hbar[\varphi_0^\hbar[Z_t]])$ on the basis $\{ \varphi_\alpha^\hbar[Z_t,z_t] \}$, leading to the whole evolution system for the coefficients $\vec{c}(t) = (c_\alpha(t))$ of \eqref{e:ansatz} solving \eqref{e:simplified_non_selfadjoint_problem2}. The main result of this section is:

\begin{prop}[Matrix elements]
\label{p:matrix_elements} Let $\chi \in \mathcal{C}_c^\infty(\R)$ be a bump function near zero. Set 
$$
\widetilde{P}_{N}(t,z) := \sigma_{Z_t,N}^{\operatorname{AW}} \big( \chi(\vert F_t^{-1} (z-z_t) \vert^2) P_{N}\big) (t,z),
$$ 
where the anti-Wick approximation $\sigma_{Z_t,N}^{\operatorname{AW}}$ is defined by \eqref{e:anti-Wick_approximation}. Then the operator 
$$
\Op_{\hbar,Z_t}^{\operatorname{AW}}(\widetilde{P}_{N}) := \Op_\hbar(\widetilde{P}_N * W_\hbar[\varphi_0^\hbar[Z_t]])
$$ 
satisfies, for $ t \in [-\delta,\delta]$, for every $\alpha \in \mathbb{N}_0^d$ and every $\gamma \in \mathbb{N}_0^d - \{ \alpha \}$:
\begin{align}
\label{e:matrix_estimate_general}
\left \vert \big \langle \Op_{\hbar,Z_t}^{\operatorname{AW}}(\widetilde{P}_{N}) \varphi_{\alpha+\gamma}^\hbar[Z_t,z_t], \varphi_{\alpha}^\hbar[Z_t,z_t] \big \rangle_{L^2(\R^d)} \right \vert \leq C_N \hbar \big( 1+ \vert \alpha \vert  \big), \quad \text{if } \vert \gamma \vert \leq 2N,
\end{align}
and the left-hand-side vanishes for $\vert \gamma \vert > 2N$.  Moreover, for every $\gamma \in \mathbb{N}_0^d$ such that $\vert \gamma \vert \leq 2N$, 
\begin{equation}
\label{e:matrix_estimate_first_line}
\left \vert \big \langle \Op_{\hbar,Z_t}^{\operatorname{AW}}(\widetilde{P}_{N}) \varphi_\gamma^\hbar[Z_t,z_t], \varphi_{0}^\hbar[Z_t,z_t] \big \rangle_{L^2(\R^d)} \right \vert \leq C_N \hbar^{3/2}.
\end{equation}
\end{prop}

\begin{remark}
Notice that $\widetilde{P}_N \in \mathcal{C}_c^\infty(\R^{2d})$ due to the cut-off function $\chi$. In particular, $\Op_{\hbar,Z_t}^{\operatorname{AW}}(\widetilde{P}_{N})$ is a compact operator and the matrix elements \eqref{e:matrix_estimate_first_line} vanish for
$$
\vert \alpha \vert , \vert \alpha + \gamma \vert \geq C_N \hbar.
$$ 
However, to gain the $\hbar$ factor in \eqref{e:matrix_estimate_general} and obtaining an estimate uniform in $\hbar$, we pay with the growth $(1+\vert \alpha \vert)$. This will be sufficient to show that $\vec{c}(t)$ of \eqref{e:ansatz} belongs to $\ell_{\rho - 3\sigma}(\mathbb{N}^d)$ for some $\rho > 0$, $0 < \sigma < \rho/3$, and $t \in [-\delta,\delta]$ (see Section \ref{s:propagation} below). 

%In other words, our solution \eqref{e:ansatz} will be a linear convination of a finite number of excited states $\varphi_\alpha^\hbar[Z_,z_t]$, where this number grows like $C/\hbar$ as $\hbar \to 0^+$. The norm $\Vert \vec{c}(t) \Vert_{\rho - 3\sigma}$, however, is uniformly bounded.

%One can wonder if it is possible to approximate the solution $\varphi_\hbar(t,x)$ by a finite linear convination of excited states, uniformly in $\hbar$, such as in the propagation results for sefadjoint-operators \cite{Robert12}. One can obtain this finite linear convination by taking the cut-off
%$$
%\chi \left( \frac{ \vert F_t^{-1} (z-z_t) \vert^2}{ C \hbar} \right),
%$$
%instead of $\chi(\vert F_t^{-1} (z-z_t) \vert^2)$. To prove that the solutions in both cases coincide modulo $O(\hbar^N)$ seems to be a bit more technical, so we prefer to use this perhaps weaker approximation result.

\end{remark}

To facilitate the calculations, we exploit the Bargmann space representation. This is very convenient to compute the matrix elements in the basis of excited coherent states, due to the particular form of the Hermite functions in Bargmann space, namely given by holomorphic monomials. Moreover, we will see that the matrix element corresponding to the index $(\alpha, \alpha + \gamma)$ is asymptotically given by the $\gamma$-Fourier coefficient of a trigonometric polynomial (see Lemma \ref{l:hermite_computation} below) of degree $2N$, which will vanish provided that $\vert \gamma \vert > 2N$. Therefore, the infinite matrix associated with the operator $\Op_{\hbar,Z_t}^{\operatorname{AW}}(\widetilde{P}_{N})$ is close to be diagonal, and then the propagation by this matrix can be easily estimated (see Section \ref{s:propagation}). Similar ideas have been used in \cite[Thm. 4.1]{Paul93}. 

The Bargmann space $\mathcal{H}_\hbar$ is given by the Hilbert space of holomorphic functions (see for instance \cite{Bargmann61,Bargmann67,Berezin_Shubin91,Borthwick_Graffi05}):
$$
\mathcal{H}_\hbar := L^2_{\text{hol}}\left(\mathbb{C}^d , e^{-\frac{\vert z \vert^2}{2\hbar}} \frac{dz \, d\overline{z}}{(2\pi \hbar)^{d/2}} \right),
$$
{where $L^2_{\text{hol}}(\mathbb{C}^d,d\mu)$ defines the space of holomorphic functions with finite $L^2$ norm with respect to the measure $\mu$ on $\mathbb{C}^d$}. The Bargmann transform $\mathcal{B}_\hbar : L^2(\R^{d}) \to \mathcal{H}_\hbar$ is the unitary operator defined by the following integral operator:
$$
\mathcal{B}_\hbar \, \psi(z) :=  \frac{1}{(\pi \hbar)^{d/4}}\int_{\R^d} \exp \left[ - \frac{1}{2\hbar}( { z^2} + \vert x \vert^2 - 2\sqrt{2} z \cdot x) \right] \psi(x) dx.
$$
Under the Bargmann transform, the eigenfunctions of the harmonic oscillator $\widehat{H}_\hbar$ have a particular convenient form:
$$
\mathcal{B}_\hbar \, \varphi_\alpha^\hbar (z) = \frac{z^\alpha}{\big( (2\hbar)^{\vert \alpha \vert}  \alpha! \big)^{1/2}}, \quad \alpha \in \mathbb{N}^d,
$$
where we denote $\varphi_\alpha^\hbar := \varphi_\alpha^\hbar[Z_0]$ for $Z_0 = (i\Id,\Id)^\mathfrak{t}$.
%while the harmonic oscillator $\widehat{H}_\hbar$ itself is conjugated into
%$$
%\mathcal{B}_\hbar \, \widehat{H}_\hbar \, \mathcal{B}_\hbar^{-1} = \hbar \sum_{j=1}^d \omega_j \left( z_j \frac{\partial}{\partial z_j} + \frac{1}{2} \right).
%$$
Moreover, the Bargmann transform $\mathcal{B}_\hbar$  intertwins anti-Wick ope\-rators with Toeplitz ope\-rators. Identifying  $\mathbb{C}^d$ with $\R^{2d}$ via $z = x + i \xi$, the following holds (see \cite[Appendix]{Borthwick_Graffi05} or \cite[Sect. 5.2]{Berezin_Shubin91}):
$$
\mathcal{B}_\hbar \, \Op_\hbar^{\text{AW}}(a) \, \mathcal{B}_\hbar^{-1} = T_\hbar(a),
$$
where the anti-Wick quantization of $a$ is defined by
$$
\Op_\hbar^{\text{AW}}(a) := \Op_\hbar\big( a \circ W_\hbar[ \varphi_0^\hbar, \varphi_0^\hbar]\big),
$$
and the Toeplitz operator $T_\hbar(a) : \mathcal{H}_\hbar \to \mathcal{H}_\hbar$ is given by
$$
T_\hbar(a) = \Pi_\hbar M(a), 
$$
where $M(a)$ defines the multiplication operator on $L^2\big(\mathbb{C}^d \, , \, e^{-\frac{\vert z \vert^2}{2\hbar}} dz \, d\overline{z} \big)$, and 
$$
\Pi_\hbar :L^2\left(\mathbb{C}^d , e^{-\frac{\vert z \vert^2}{2\hbar}} \frac{dz \, d\overline{z}}{(2\pi \hbar)^{d/2}} \right) \to \mathcal{H}_\hbar
$$ 
is the orthogonal projection onto the holomorphic subspace. 

Let us also define the modified anti-Wick quantization associated to a normalized Lagrangian frame $Z$:
$$
\Op_{\hbar,Z}^{\text{AW}}(a) := \Op_\hbar \big( a * W_\hbar[\varphi_0^\hbar[Z],\varphi_0^\hbar[Z] ] \big).
$$
The Anti-Wick quantization and the Weyl quantization are equivalent in the semi\-classical limit. Indeed, one can show (see Lemma \ref{l:weyl_to_antiwick} below) that
\begin{equation}
\label{weyl_vs_antiwick}
\Op_{\hbar,Z}^{\text{AW}}(a) = \Op_\hbar(a) + O(\hbar).
\end{equation}

We next establish the correspondence between polyonomial symbols for Weyl and anti-Wick quantization. Let $Z$ be a normalized Lagrangian frame, we define the coefficients $\lambda_\alpha[Z]$, for $\vert \alpha \vert \equiv 0 \, (\textnormal{mod } \, 2)$, by
\begin{equation}
\label{e:coefficients}
 \lambda_\alpha[Z] := \int_{\R^{2d}} y^\alpha
  \Phi_{(0,0)}^1[\mathcal{Z}](y) dy.
\end{equation}
Notice, in particular, that if $Z = (i \Id, \Id)^\mathfrak{t}$ then $\lambda_\alpha[Z] = \displaystyle \frac{ \alpha!}{4^{\frac{\vert \alpha \vert}{2}} \left( \frac{ \vert \alpha \vert}{2} \right) !}$.
\begin{lemma}
\label{l:weyl_to_antiwick}
Let $Z$ be a normalized Lagrangian frame and  $q \in \mathcal{C}_c^\infty(\R^{2d})$. Let $N \in \mathbb{N}$, then
\begin{equation}
\label{e:weyl_to_antiwick}
 \Op_{\hbar}(q) = \Op_{\hbar,Z}^{\textnormal{AW}} \left( \sum_{m = 0}^{N} \sum_{\vert \alpha \vert = 2m} \frac{(-1)^m \hbar^{m} \mu_\alpha[Z]}{\alpha!} D^\alpha q \right) + O(\hbar^{N +1}),
\end{equation}
where  the coefficients $\mu_\alpha[Z]$ are uniquely determined in terms of \eqref{e:coefficients} by following identities: For every $\gamma \in \mathbb{N}^{2d}_0$ with $\vert \gamma \vert \equiv 0 \, (\textnormal{mod } 2)$,
\begin{equation}
\label{e:coeeficients_inverse}
\sum_{\substack{ \alpha \leq \gamma \\ \vert \alpha \vert \equiv 0 \, (\textnormal{mod } 2)}} \frac{(-1)^{\frac{\vert \alpha \vert}{2}}  \mu_\alpha[Z] \lambda_{\gamma - \alpha}[Z]}{\alpha! (\gamma - \alpha)!}  = \left \lbrace \begin{array}{ll} 
1 & \text{if} \quad \gamma = 0, \\[0.2cm]
0 & \text{if} \quad \gamma \neq 0.
\end{array} \right.
\end{equation}
\end{lemma}
\begin{remark}
Notice that, if $Z = (i \Id, \Id)^\mathfrak{t}$, then $\mu_\alpha[Z] = \lambda_\alpha[Z] = \displaystyle \frac{ \alpha!}{4^{\frac{\vert \alpha \vert}{2}} \left( \frac{ \vert \alpha \vert}{2} \right) !}$.
\end{remark}

\begin{definition}
We define
\begin{equation}
\label{e:anti-Wick_approximation}
\sigma_{Z,N}^{\operatorname{AW}}(q) := \sum_{m = 0}^{N} \sum_{\vert \alpha \vert = 2m} \frac{(-1)^m \hbar^{m} \mu_\alpha[Z]}{\alpha!} D^\alpha q.
\end{equation}
\end{definition}

\begin{proof}[Proof of Lemma \ref{l:weyl_to_antiwick}]
By definition, we have
$$
\Op^{\text{AW}}_{\hbar,Z}(q) = \Op_\hbar\big( q * W_\hbar[\varphi_0^\hbar[Z], \varphi_0^\hbar[Z]] \big),
$$
where, by \eqref{e:ground_state_lift},
$$
W_\hbar[\varphi_0^\hbar[Z], \varphi_0^\hbar[Z]] = \Phi_{(0,0)}^\hbar[\mathcal{Z}](z) = \frac{1}{\pi^d} e^{-\frac{1}{\hbar} G z \cdot z},
$$
with $G$ given by \eqref{e:metric}.  Then, expanding $q(w)$ by Taylor's theorem near $z$, 
\begin{align*}
q * \Phi_{(0,0)}^\hbar[\mathcal{Z}](z) & \\[0.2cm]
 & \hspace*{-2cm} = \int_{\R^{2d}} q(w) \Phi_{(0,0)}^\hbar[\mathcal{Z}](w-z) dw \\[0.2cm]
 & \hspace*{-2cm}  = \sum_{\vert \alpha \vert \leq  N} \frac{1}{\alpha!} D^\alpha q(z) \int_{\R^{2d}} (w-z)^\alpha
  \Phi_{(0,0)}^\hbar[\mathcal{Z}](w-z) dw  + \int_{\R^{2d}} R_N(z,w) \Phi_{(0,0)}^\hbar[\mathcal{Z}](w-z) dw \\[0.2cm]
 & \hspace*{-2cm}  = \sum_{\vert \alpha \vert \leq N} \frac{\hbar^{\frac{ \vert \alpha \vert}{2}}}{\alpha!} D^\alpha q(z) \int_{\R^{2d}} y^\alpha
  \Phi_{(0,0)}^1[\mathcal{Z}](y) dy + \int_{\R^{2d}} R_N(z,z+y) \Phi_{(0,0)}^\hbar[\mathcal{Z}](y) dy.
\end{align*}
For the derivatives of $q$ of odd degree, we have
\begin{align*}
\int_{\R^{2d}} y^\alpha
  \Phi_{(0,0)}^1[\mathcal{Z}](y) dy & = \frac{1}{\pi^d} \int_{\R^{2d}} y^\alpha e^{-Gy \cdot y} dy \\[0.2cm]
   & = \frac{1}{\pi^d} \int_{\R^{2d}} (F^{-T} y)^\alpha e^{-\vert y\vert^2} dy \\[0.2cm]
   & = 0,
\end{align*}
where $G = F F^T$ with $F$ real symplectic. Thus
$$
q * \Phi_{(0,0)}^\hbar[\mathcal{Z}](z) = \sum_{m = 0}^{N } \sum_{\vert \alpha \vert = 2m} \frac{\hbar^{m} \lambda_\alpha[Z]}{\alpha!} D^\alpha q(z) +  \int_{\R^{2d}} R_N(z,z+y) \Phi_{(0,0)}^\hbar[\mathcal{Z}](y) dy.
$$

On the other hand,
\begin{align*}
\Phi_{(0,0)}^\hbar[\mathcal{Z}] * \left( \sum_{m = 1}^{N} \sum_{\vert \alpha \vert = 2m} \frac{(-1)^m \hbar^{m} \lambda_\alpha[Z]}{\alpha!} D^\alpha q \right) & \\[0.2cm] 
 & \hspace*{-5cm} = \sum_{m=0}^{N} \sum_{m' = 0}^{N - m} \sum_{\vert \alpha \vert = 2m} \sum_{\vert \beta \vert = 2m'} \frac{(-1)^m \hbar^{m+m'} \mu_\alpha[Z] \lambda_\beta[Z]}{\alpha! \beta!} D^{\alpha + \beta}q(z) \\[0.2cm]
 & \hspace*{-5cm} = \sum_{k=0}^{N} \sum_{\vert \gamma \vert = 2k} \hbar^{k}D^{\gamma}q(z)  \sum_{\substack{ \alpha \leq \gamma \\ \vert \alpha \vert \equiv 0 \, (\text{mod } 2)}}  \frac{(-1)^{\frac{\vert \alpha \vert}{2}}  \mu_\alpha[Z] \lambda_{\gamma - \alpha}[Z]}{\alpha! (\gamma - \alpha)!}  \\[0.2cm]
 & \hspace*{-5cm} = q(z),
\end{align*}
provided that condition \eqref{e:coeeficients_inverse} holds.
\end{proof}

The following identity allows us to compute the matrix elements of an anti-Wick operator on the basis of Hermite functions:
\begin{align}
\label{wigner_with_bargmann}
\big \langle \varphi_\beta^\hbar, \Op_\hbar^{\text{AW}}(a) \varphi_\alpha^\hbar \big \rangle_{L^2(\R^d)}  = \frac{1}{C_{\hbar , \alpha,\beta}} \int_{\mathbb{C}^d} z^{\beta} a(z) \overline{z}^{\alpha} e^{-\frac{\vert z \vert^2}{2\hbar}} dz \, d\overline{z},
\end{align}
where
\begin{equation}
\label{e:normalizing_constant}
C_{\hbar, \alpha,\beta} =  \pi^d (2\hbar)^{d + \frac{ \vert \alpha \vert + \vert \beta \vert}{2}}(\alpha! \beta!)^{\frac{1}{2}}.
\end{equation}

\begin{lemma}
\label{l:hermite_computation}
Let $p(z)$ be a polynomial of degree $N$, homogeneous of degree $n$ at zero. For any $\alpha \in \mathbb{N}_0^d$, $\gamma \in \mathbb{N}_0^d - \{ \alpha \}$, set 
$$
z_{\alpha,\gamma,\hbar} := \left( \sqrt{\hbar(2\alpha_1 + \gamma_1 + 1)}, \ldots, \sqrt{\hbar(2\alpha_d + \gamma_d + 1)} \right) \in \R^{d}.
$$
Then, for any  $\chi \in \mathcal{C}_c^\infty(\R)$ and $q(z) = \chi(\vert z \vert^2) p(z)$, one has
\begin{align*}
 \frac{1}{C_{\hbar , \alpha,\alpha+\gamma}} \int_{\mathbb{C}^d} z^{\alpha+\gamma}  q(z) \overline{z}^{\alpha} e^{-\frac{\vert z \vert^2}{2\hbar}} dz \, d\overline{z} & \\[0.2cm]
 & \hspace*{-3cm}  = \frac{\Lambda(\alpha,\gamma)}{(2\pi)^d} \int_{\mathbb{T}^d}  q \circ \Phi_\tau^H(z_{\alpha,\gamma,\hbar},0) e^{-i \gamma \cdot \tau} d\tau + O_N\left( \hbar^{\frac{n}{2}}(1+ \vert \alpha \vert)  \right),
\end{align*}
as $\hbar \to 0$, {where we denote $\Phi_\tau^H(z) \equiv \Phi_z(\tau)$} (see \eqref{e:multiflow}), and
$$
\Lambda(\alpha,\gamma) := \frac{1}{[\alpha!(\alpha+\gamma)!]^{\frac{1}{2}}} \prod_{j=1}^d \Gamma\left( \frac{2\alpha_j + \gamma_j +2}{2} \right),
$$ 
where $\Gamma$ denotes the Gamma function. Moreover, if $\vert \gamma \vert > N$, then
\begin{equation}
\label{e:vanishes}
 \frac{1}{C_{\hbar , \alpha,\alpha+\gamma}} \int_{\mathbb{C}^d} z^{\alpha+\gamma}  q(z) \overline{z}^{\alpha} e^{-\frac{\vert z \vert^2}{2\hbar}} dz = 0.
\end{equation}
\end{lemma}

\begin{remark}
\label{limit_of_the_constant}
 Notice that, for any $\gamma \in \mathbb{N}_0^d - \{ \alpha \}$ with $\vert \gamma \vert \leq 2N$, $\vert \Lambda(\alpha,\gamma) \vert \leq 1$.
\end{remark}

\begin{proof}
Identifying $\R^{2d} \simeq \mathbb{C}^d$ and taking polar coordinates
$$
z = \Phi_\tau^H(r,0) \equiv \big( r_1 e^{i \tau_1}, \ldots, r_d e^{i \tau_d} \big), \quad \tau \in \mathbb{T}^d, \quad r = (r_1, \ldots, r_d) \in \R^d_+,
$$
we have:
$$
\int_{\mathbb{C}^d} z^{\alpha+\gamma} { q(z)} \overline{z}^{\alpha} e^{-\frac{\vert z \vert^2}{2\hbar}} dz \, d\overline{z} = \int_{\R^d_+} \int_{\mathbb{T}^d} { q}\circ \Phi_\tau^H(r,0) e^{-i \gamma\cdot \tau} \prod_{j=1}^d r_j^{2\alpha_j + \gamma_j + 1} e^{-\frac{r_j^2}{2\hbar}} dr d\tau.
$$
Since $q$ is a polynomial of degree $N$, this shows \eqref{e:vanishes}. Now, we perform the following change of variables, shifting the center to $z_{\alpha,\gamma,\hbar}$ and zooming by $1/(2\hbar)^{1/2}$:
$$ 
r_j = \sqrt{2\hbar}s_j + \sqrt{\hbar(2\alpha_j + \gamma_j + 1)} , \quad s_j \in \left[ - \sqrt{\frac{2\alpha_j + \gamma_j + 1}{2}}, \infty \right),  \quad j =1, \ldots , d.
$$
We aim at showing that
\begin{align}
\label{e:intermediate_bound}
\frac{(2\hbar)^{\frac{d}{2}}}{C_{\hbar,\alpha,\alpha + \gamma} }\prod_{j=1}^d r_j^{2\alpha_j + \gamma_j + 1} e^{-\frac{r_j^2}{2\hbar}} \leq C_d e^{-\frac{ \vert s \vert^2}{2}},  
\end{align}
for some constant $C_d > 0$ depending only on the dimension $d$. Indeed, by the following inequality
$$
(\sqrt{2} s + \sqrt{B})^{B} e^{-\frac{( \sqrt{2} s + \sqrt{B})}{2}^2} \leq e^{-\frac{s^2}{2}} \, \left( \frac{B}{e} \right)^{\frac{B}{2}}, \quad s \geq - \left(\frac{B}{2}\right)^{\frac{1}{2}}, \quad B \geq 0,
$$
we have
\begin{equation}
\label{e:intermediate_estimate}
r_j^{2\alpha_j + \gamma_j + 1} e^{-\frac{r_j^2}{2\hbar}} \leq e^{-\frac{s_j^2}{2}} e^{-\frac{2\alpha_j+\gamma_j+1}{2}} \big(\hbar(2\alpha_j + \gamma_j + 1) \big)^{\frac{2\alpha_j + \gamma_j + 1}{2}}.
\end{equation}
Using next Stirling's formula
$$
n! \geq \sqrt{2\pi} n^{n+\frac{1}{2}} e^{-n}, \quad n \geq 1,
$$
the right-hand-side of \eqref{e:intermediate_estimate} can be bounded by
$$
e^{-\frac{s_j^2}{2}}e^{-\frac{2\alpha_j+\gamma_j+1}{2}} \big(\hbar(2\alpha_j + \gamma_j + 1) \big)^{\frac{2\alpha_j + \gamma_j + 1}{2}} \leq e^{-\frac{s_j^2}{2}} \left( \frac{\hbar^{2\alpha_j+\gamma_j+1} (2\alpha_j + \gamma_j + 1)!}{\sqrt{2\pi}(2\alpha_j + \gamma_j + 1)^{\frac{1}{2}}} \right)^{\frac{1}{2}}.
$$
Recall, from \eqref{e:normalizing_constant}, that
$$
C_{\hbar, \alpha,\alpha+\gamma} =  \pi^d (2\hbar)^{d + \frac{ \vert \alpha \vert + \vert \alpha+\gamma \vert}{2}}\big[ \alpha!(\alpha+\gamma)! \big]^{\frac{1}{2}}.
$$
Then, using the following standard property of the Gamma function:
\begin{align}
\label{e:property_gamma_function}
\Gamma(2x) \lesssim x^{\frac{1}{2}} \Gamma(x)^2  2^{2x-1} \lesssim x^{\frac{1}{2}} \Gamma(x-y)\Gamma(x+y) 2^{2x-1}, \quad 0 \leq y \leq x,
\end{align}
where the notation $\lesssim$ means inequality modulo multiplication by a universal constant, with
$$
x = \frac{2\alpha_j + \gamma_j + 2}{2}, \quad y = \frac{\gamma_j}{2},
$$
we conclude \eqref{e:intermediate_bound}. 
\medskip

On the other hand, denoting
$$
\mathbf{q}(r,\tau) := q \circ \Phi_\tau^H(r,0), \quad (r,\tau) \in \R^d_+ \times \mathbb{T}^d,
$$
and using Taylor's theorem,
\begin{align*}
q\circ \Phi_\tau^H(z_{\alpha,\gamma,\hbar} + \sqrt{2\hbar} s,0)  = q\circ \Phi_\tau^H(z_{\alpha,\gamma,\hbar},0) + \sqrt{2\hbar} s \cdot \int_0^1 \partial_r \, \mathbf{q} \big(z_{\alpha,\gamma,\hbar} + t \sqrt{2\hbar}s,\tau \big) dt.
\end{align*}
Since $p$ is a polynomial of degree $N$, for every $r \in \R^d_+$, the Fourier coefficients $\widehat{\mathbf{q}}(r,\gamma)$ given by
$$
\widehat{\mathbf{q}}(r,\gamma) = \frac{1}{(2\pi)^d} \int_{\mathbb{T}^d} \mathbf{q}(r,\tau) e^{-i \gamma \cdot \tau} d\tau
$$ 
vanish for $\vert \gamma \vert \geq N+1$. Moreover, {using that $\chi$ has compact support}, we obtain that
$$
 \sup_{ \vert \gamma \vert \leq N} \Big \vert \partial_r \, \widehat{\mathbf{q}} \big(z_{\alpha,\gamma,\hbar} + t \sqrt{2\hbar}s, \gamma \big) \Big \vert  \leq C_N \hbar^{\frac{n-1}{2}} (1+ \vert \alpha \vert)(1 + \vert s \vert^2).
$$
Finally, since
$$
\int_0^\infty r^{2\alpha_j+\gamma_j+1} e^{-\frac{r_j^2}{2\hbar}} dr_j = \frac{1}{2} \Gamma\left( \frac{2\alpha_j + \gamma_j + 2}{2} \right) (2\hbar)^{\frac{2\alpha_j+\gamma_j+2}{2}},
$$
we obtain 
\begin{align*}
\left \vert \frac{1}{C_{\hbar , \alpha,\alpha+\gamma}} \int_{\mathbb{C}^d} z^{\alpha+\gamma} q(z) \overline{z}^{\alpha} e^{-\frac{\vert z \vert^2}{2\hbar}} dz \, d\overline{z} - \frac{\Lambda(\alpha,\gamma)}{(2\pi)^d} \int_{\mathbb{T}^d} q \circ \Phi_\tau^H(z_{\alpha,\gamma,\hbar},0) e^{-i\gamma \cdot \tau} d\tau \right \vert  \\[0.4cm]
 & \hspace*{-7cm}  \leq C_N  \hbar^{\frac{n}{2}}  (1+ \vert \alpha \vert)\int_{\R^d} \vert s \vert(1 + \vert s \vert^2) e^{-\frac{ \vert s \vert^2}{2}} ds \\[0.2cm]
  & \hspace*{-7cm}  = O_N \left(\hbar^{\frac{n}{2}}(1+  \vert \alpha \vert) \right).
\end{align*}
\end{proof}

\begin{proof}[Proof of Proposition \ref{p:matrix_elements}]
By \eqref{e:anti-Wick_approximation}, we have: 
\begin{align*}
\label{e:remainder_in_antiwick}
\sigma_{Z_t,N}^{\operatorname{AW}} \big( \chi(\vert F_t^{-1} (z-z_t) \vert^2) P_{N}\big) (t,z) & = \sum_{j=0}^N \sum_{m=0}^{N-j} \hbar^m \chi^{(j)} (\vert F_t^{-1} (z-z_t) \vert^2) P_{j,m}(t,z - z_t) \\[0.2cm]
 & =: \sum_{j=0}^N \sum_{m=0}^{N-j} \hbar^m q_{j,m}(t,z-z_t),
\end{align*}
where $P_{j,m}(t,\cdot)$ is a polynomial of degree $N-m+j$ and homogeneous of degree $3-m+j$. 
%, and 
%\begin{equation}
%\label{e:remainder_at_origin}
%\widetilde{R}_{N,0}(t,z) := \sum_{\vert \beta \vert = 3} \frac{\vert \beta \vert}{\beta!} z^\beta \int_0^1 (1- s)^{\vert \beta \vert -1} D^\beta P_{N}\big( z_t + sz\big) ds.
%\end{equation}}
We have, by Egorov's theorem for metaplectic operators, that
\begin{align*}
q_{j,m} * W_\hbar \big[ \varphi_0^\hbar[Z_t], \varphi_0^\hbar[Z_t] \big] (F_tz) & = \int_{\R^{2d}} q_{j,m}(t,F_t z - w) W_\hbar \big[ \varphi_0^\hbar[Z_t], \varphi_0^\hbar[Z_t] \big](w) dw \\[0.2cm]
 & = \int_{\R^{2d}} q_{j,m} (t,F_t z - F_t w) W_\hbar \big[ \varphi_0^\hbar, \varphi_0^\hbar \big](w) dw \\[0.2cm]
 & = \big(q_{j,m} \circ F_t \big) * W_\hbar \big[ \varphi_0^\hbar, \varphi_0^\hbar \big](z).
\end{align*}
Thus, using  Egorov's theorem for metaplectic operators one more time yields:
\begin{align*}
 \big \langle \Op_{\hbar,Z_t}^{\textnormal{AW}}(q_{j,m}(t,z-z_t)) \varphi_{\alpha+\gamma}^\hbar[Z_t,z_t], \varphi_{\alpha}^\hbar[Z_t,z_t] \big \rangle_{L^2(\R^d)} \\[0.2cm] 
 & \hspace*{-7cm} = \big \langle \Op_\hbar\big( q_{j,m}(t,z-z_t) * W_\hbar[\varphi_0^\hbar[Z_t], \varphi_0^\hbar[Z_t]] \big) \widehat{T}[z_t] \varphi_{\alpha+\gamma}^\hbar[Z_t],  \widehat{T}[z_t] \varphi_{\alpha}^\hbar[Z_t] \big \rangle_{L^2(\R^d)} \\[0.2cm]
  & \hspace*{-7cm} =  \big \langle \Op_\hbar\big( q_{j,m} * W_\hbar[\varphi_0^\hbar[Z_t], \varphi_0^\hbar[Z_t]] \big) \varphi_{\alpha+\gamma}^\hbar[Z_t],   \varphi_{\alpha}^\hbar[Z_t] \big \rangle_{L^2(\R^d)} \\[0.2cm]
  & \hspace*{-7cm} =  \big \langle \Op_\hbar\big( q_{j,m} * W_\hbar[\varphi_0^\hbar[Z_t], \varphi_0^\hbar[Z_t]] \big) \varphi_{\alpha+\gamma}^\hbar[Z_t],   \varphi_{\alpha}^\hbar[Z_t] \big \rangle_{L^2(\R^d)} \\[0.2cm]
  & \hspace*{-7cm} =  \big \langle \Op_\hbar\big( ( q_{j,m} \circ F_t) * W_\hbar[\varphi_0^\hbar, \varphi_0^\hbar] \big) \varphi_{\alpha+\gamma}^\hbar,   \varphi_{\alpha}^\hbar \big \rangle_{L^2(\R^d)} \\[0.2cm]
  & \hspace*{-7cm} = \big \langle \Op_\hbar^{\text{AW}} \big( q_{j,m} \circ F_t \big) \varphi_{\alpha+\gamma}^\hbar, \varphi_{\alpha}^\hbar \big \rangle_{L^2(\R^d)}.
\end{align*}
Moreover, using the Bargmann transform,
$$
 \big \langle \Op_\hbar^{\text{AW}} \big( q_{j,m} \circ F_t \big) \varphi_{\alpha+\gamma}^\hbar, \varphi_{\alpha}^\hbar \big \rangle_{L^2(\R^d)}  = \frac{1}{C_{\hbar, \alpha, \alpha + \gamma}} \int_{\mathbb{C}^d} z^{\alpha+\gamma} \big( q_{j,m} \circ F_t \big)(z) \overline{z}^{\alpha} e^{- \frac{ \vert z \vert^2}{2\hbar}} dz d \overline{z}.
$$
%Substituting in the expression the definition \eqref{e:remainder_at_origin} of $R_{N,0}$ and applying Fubini, we have
%\begin{align*}
%\frac{1}{C_{\hbar, \alpha, \alpha + \gamma}} \int_{\mathbb{C}^d} z^\alpha \big( \widetilde{R}_0 \circ F_t \big)(z) \overline{z}^{\alpha + \gamma} e^{- \frac{ \vert z \vert^2}{2\hbar}} dz d \overline{z} & =  \frac{1}{C_{\hbar, \alpha, \alpha + \gamma}} \int_0^1 \int_{\mathbb{C}^d} z^\alpha \mathcal{P}_N(t,z,s) \overline{z}^{\alpha + \gamma} e^{- \frac{ \vert z \vert^2}{2\hbar}} dz d \overline{z} ds,
%\end{align*}
%where
%$$
%\mathcal{P}_N(t,z,s) = \sum_{\vert \beta \vert = 3} \frac{\vert \beta \vert}{\beta!} (F_tz)^\beta (1-s)^{\vert \beta \vert - 1} D^\beta P_N(z_t + s F_t z).
%$$
Applying Lemma \ref{l:hermite_computation} to $q(z) = q_{j,m}(t, F_tz)$, \eqref{e:matrix_estimate_general} holds.
\end{proof}

\subsection{Propagation result}
\label{s:propagation}

In this section we study the problem 
\begin{equation}
\label{e:simplified_non_selfadjoint_problem_2}
\big( i \hbar \partial_t  + \Op_\hbar(P_2)+ \Op_{\hbar,Z_t}^{\operatorname{AW}}(\widetilde{P}_N) \big) \varphi_\hbar(t,x) = 0, \quad \varphi_\hbar(0,x) = \varphi_0^\hbar[Z_0,z_0](x),
\end{equation}
with the ansatz \eqref{e:ansatz}. The vector of coefficients $\vec{c}(t) = (c_\alpha(t))$ obeys the equation:
\begin{equation}
\label{e:differential_equation_coefficients}
\dot{c}_\alpha(t) = \left( \dot{\varrho}_t + \frac{i \dot{\Lambda}_t}{\hbar}  \right) c_\alpha(t) + \sum_{\vert \beta - \alpha \vert \leq 2} \kappa_{\alpha \beta}(t) c_\beta(t) + \sum_{\vert \gamma - \alpha \vert \leq 2N} \mu_{\alpha \gamma}(t,\hbar) c_\gamma(t).
\end{equation}
The matrix elements $\kappa_{\alpha \beta}(t)$ correspond to the quadratic part $P_{2}$ and have been estimated in Propositon \ref{p:coefficients_quadratic_part}, while the matrix elements $\mu_{\alpha \gamma}(t)$ come from the remainder term $\Op_{\hbar,Z_t}^{\operatorname{AW}}(\widetilde{P}_N)$:
$$
\mu_{\alpha \gamma}(t) := i\hbar^{-1} \big \langle \Op_{\hbar,Z_t}^{\operatorname{AW}}(\widetilde{P}_N) \varphi_{\alpha+\gamma}^\hbar[Z_t,z_t], \varphi_{\alpha}^\hbar[Z_t,z_t] \big \rangle_{L^2(\R^d)},
$$
and have been estimated in Proposition \ref{p:matrix_elements}. In this section, we prove the following propagation result:

\begin{prop}
\label{p:final_estimate_coefficients}
Let $\rho > 0$. Let $0 < \sigma < \rho/3$. Then, given $\vec{c}_0 = (1, 0, \ldots) \in \ell_\rho(\mathbb{N}^d)$,  there exists $\delta = \delta(\rho,\sigma) > 0$ and a unique solution
\begin{equation}
\label{e:solution_time_dependent}
\varphi_\hbar(t,x) := \sum_{\alpha \in \mathbb{N}^d} c_\alpha(t,\hbar) \varphi_\alpha^\hbar[Z_t,z_t](x), \quad t \in [-\delta,\delta],
\end{equation}
to \eqref{e:simplified_non_selfadjoint_problem_2}. Moreover, $\vec{c}(t) = (c_\alpha(t,\hbar) ) \in \mathcal{C}([-\delta,\delta], \ell_{\rho-3\sigma}(\mathbb{N}^d))$, and:
\begin{align}
c_0(t) & = e^{  \frac{i\Lambda_t}{\hbar} +\varrho_t } \big( 1 + O(\sqrt{\hbar}) \big), \\[0.2cm]
\label{e:decayment_coefficients}
c_\alpha(t) & = e^{  \frac{i\Lambda_t}{\hbar} +\varrho_t } O\left( \sqrt{\hbar} \exp \big( - (\rho - 3\sigma) \vert \alpha \vert \big) \right), \quad \alpha \neq 0,
\end{align}
uniformly in $t \in [-\delta,\delta]$.

\end{prop}

\begin{proof}
We rewrite equation \eqref{e:differential_equation_coefficients} as:
$$
\frac{d}{dt} \vec{c}(t) = ( \mathcal{A}(t) + \mathcal{B}(t) ) \vec{c}(t), \quad \vec{c}(0) = \vec{c}_0,
$$
where
\begin{align}
\label{e:equation_on_A}
\big( \mathcal{A}(t) \vec{c}(t) \big)_\alpha & = \left( \dot{\varrho}_t + \frac{i \dot{\Lambda}_t}{\hbar} \right) c_\alpha(t) + \sum_{\vert \beta - \alpha \vert \leq 2} \kappa_{\alpha \beta}(t) c_\beta(t), \\[0.2cm]
\label{e:equation_on_B}
 \big( \mathcal{B}(t) \vec{c}(t) \big)_\alpha & = \sum_{\vert \gamma - \alpha \vert \leq 2N} \mu_{\alpha \gamma}(t) c_\gamma(t).
\end{align}
By Proposition \ref{p:coefficients_quadratic_part}, 
\begin{align}
\label{e:general_estimate_kappa}
\vert \kappa_{\alpha \beta}(t) \vert &  \leq C \vert \alpha \vert,  \quad \vert \alpha - \beta \vert \leq 2, \quad { \vert \alpha \vert \leq \vert \beta \vert}, \\[0.2cm]
\label{e:particular_estiamate_kappa}
\kappa_{\alpha 0}(t) & = 0, \quad \alpha \in \mathbb{N}^d_0,
\end{align}
for $t \in [-\delta, \delta]$, with $\delta > 0$ small and fixed. Moreover, by Proposition \ref{p:matrix_elements}, 
\begin{align}
\label{e:general_estimate_mu}
\vert \mu_{\alpha \beta}(t) \vert & \leq
C_N ( 1 + \vert \alpha \vert ), \quad \vert \alpha - \beta \vert \leq 2N, \\[0.2cm]
\label{e:particular_estimate_mu}
 \vert \mu_{\alpha 0}(t) \vert & \leq C_N \hbar^{1/2}, \quad \vert \alpha \vert \leq 2N,
\end{align}
for $t \in [-\delta,\delta]$. Moreover, the operators given by the matrix elements $\mathcal{K}(t) = (\kappa_{\alpha \beta}(t))$ and $\mathcal{B}(t)=(\mu_{\alpha \beta}(t))$ belong to $\mathcal{C}([-\delta,\delta],\mathscr{D}_\rho)$. Indeed, by Corollary \ref{c:kappa_well_defined}, $\mathcal{K}(t) \in \mathcal{C}([-\delta,\delta],\mathscr{D}_\rho)$. Similarly, by \eqref{e:general_estimate_mu}, we have:
\begin{align*}
\Vert (\mu_{\alpha \beta}(t)) \vec{c} \, \Vert_{\rho - \sigma} & \leq \sum_{\alpha \in \mathbb{N}^d} \left \vert \sum_{\vert \alpha - \beta \vert \leq 2} \mu_{\alpha \beta}(t) c_\beta \right \vert e^{(\rho - \sigma) \vert \alpha \vert} \\[0.2cm]
 & \leq \sum_{\alpha \in \mathbb{N}^d}  \sum_{\vert \alpha - \beta \vert \leq 2N} \vert \mu_{\alpha \beta}(t) c_\beta \vert e^{(\rho - \sigma) \vert \alpha \vert} \\[0.2cm]
 & \leq C_\rho(N) \sum_{\beta \in \mathbb{N}^d} (2N + \vert \beta \vert) \vert c_\beta \vert e^{(\rho - \sigma) \vert \beta \vert} \\[0.2cm]
 & \leq C_\rho(N) \sup_{r \geq 0} r e^{- \sigma r} \sum_{\beta \in \mathbb{N}^d} \vert c_\beta \vert e^{\rho \vert \beta \vert} \\[0.2cm]
 & \leq \frac{C_\rho(N)}{e \sigma} \Vert \vec{c} \, \Vert_\rho. 
\end{align*}
Then, by Proposition \ref{p:estimate_difference}, \eqref{e:duhamel_principle}, and \eqref{e:particular_estiamate_kappa}, we have:
$$
\vec{c}(t) = e^{\frac{i \Lambda_t}{\hbar} + \varrho_t} \vec{c}_0 + e^{\frac{i \Lambda_\hbar}{\hbar} + \varrho_t} \int_0^t V(t,r) \mathcal{B}(r) \vec{c}_0 dr,
$$
where $V(t,r)$ is the propagator corresponding to $\mathcal{K}(t) + \mathcal{B}(t)$. Finally, using  \eqref{e:particular_estimate_mu} and Lemma \ref{l:homogeneous_evolution}, we obtain the claim.

\end{proof}

\section{Construction of quasimodes}

This section is devoted to prove Theorems \ref{t:T1} and \ref{t:T2}. 

\subsection{Proof of Theorem \ref{t:T1}}
{ Let $N \geq 1$ be fixed. Let $\chi \in \mathcal{C}_c^\infty(\R)$ be a bump function with support contained in $(-1,3)$ and iqual to one on $(-\frac{1}{2},2)$. Let us define
\begin{equation}
\label{e:minimum_function}
L_\hbar :=  \left \lbrace \begin{array}{l} \displaystyle \sqrt{ \frac{2 \beta_\hbar}{\hbar^{2/3} \gamma_0}}, \quad \text{if } \beta_\hbar \geq \hbar^{2/3}, \\[0.5cm]
1, \quad \text{if } 0 \leq \beta_\hbar \leq \hbar^{2/3},
\end{array} \right.
\end{equation} 
and set $\chi_\hbar(t) := \chi(t/h^{1/3} L_\hbar)$. We define our candidate $\psi_\hbar$ to be a quasimode for $\widehat{P}_{\hbar}$  by:
$$
 \psi_\hbar(x) :=  \Theta_\hbar^{1/2} \int_\R \chi_\hbar(t) e^{-\frac{it}{\hbar}( \alpha_\hbar + i \beta_\hbar)}  \varphi_\hbar(t,x) dt, \quad \Theta_\hbar :=  \frac{C_\hbar({N}) \vert \nabla V (z_0) \vert}{\hbar^{5/6} \sqrt{\pi}},
$$
where $C_\hbar(N) > 0$ is a normalizing constant to be estimated later, $\alpha_\hbar = V(z_0)$, $\varphi_\hbar(t,x)$ is the solution to \eqref{e:simplified_non_selfadjoint_problem_2} given by \eqref{e:solution_time_dependent}, and we take $\hbar \leq \hbar_0$ so that $3\hbar^{1/3} L_\hbar \leq \delta$, where $\delta = \delta(\rho,\sigma)> 0$ is given by Proposition \ref{p:final_estimate_coefficients}. We will take $\rho,\sigma > 0$ along the proof so that $\rho - 3\sigma > 0$ is sufficiently large.}

%and $\chi_\hbar \in \mathcal{C}_c^\infty(\R)$ is a bump function satisfying, for $0 < \epsilon < 1/3$ and $0 < c < 1$ small:
%\begin{align}
%\label{e:support1}
%\supp \chi_\hbar & \subset \{ -\hbar^{1/3} \leq t \leq \hbar^{1/3 -\epsilon} \}, \\[0.2cm]
%\label{e:support2}
%\supp \chi'_\hbar & \subset \{ - \hbar^{1/3} \leq t \leq -\hbar^{1/3}(1 - c) \} \cup \{ \hbar^{1/3}( \hbar^{-\epsilon} - c) \leq t \leq \hbar^{1/3 - \epsilon} \}.
%\end{align}
\medskip

Our first goal is to prove the following proposition:
\begin{prop}
\label{p:wigner_function}
Let $a \in \mathcal{C}_c^\infty(\R^{2d})$. Then:
$$
\int_{\R^{2d}} a(z) W_\hbar[\psi_\hbar, \psi_\hbar](z) dz = I_\hbar \cdot \Big( a(z_0) + O(\hbar^{1/6}) \Big),
$$
where
\begin{equation}
\label{e:I_h}
I_\hbar := C_\hbar(N)  \hbar^{1/3} \int_\R \chi(s/L_\hbar)^2 \exp \left( \frac{2 \beta_\hbar s}{\hbar^{2/3}} - \frac{\gamma_0 s^3}{3} \right) ds,
\end{equation}
and the constant $\gamma_0$ is given by:
$$
\gamma_0 = \big \langle \Omega  \nabla V(z_0) , \partial^2 A(z_0) \Omega  \nabla V(z_0) \big \rangle,
$$
which is positive by hypothesis \eqref{e:non-degenerate}.
\end{prop}

The idea of the proof is to give a stationary-phase argument near the diagonal $t \sim t'$, together with a Taylor expansion in $t$ near $0$ inside the oscillatory integral.

\begin{proof}

By definition, the Wigner distribution $W_\hbar[\psi_\hbar,\psi_\hbar]$ is given by:
\begin{align*}
W_\hbar[\psi_\hbar, \psi_\hbar](z) =  \Theta_\hbar \sum_{\alpha,\beta \in \mathbb{N}^d} \int_{\R^2}  \varsigma_{\alpha \beta}^\hbar(t,t') e^{\phi_\hbar(t,t')} W_\hbar\big[ \varphi_\alpha^\hbar[Z_t,z_t], \varphi_\beta^\hbar[Z_{t'},z_{t'}] \big](z) dt dt',
\end{align*}
where we denote $\varsigma_{\alpha\beta}^\hbar(t,t') = \chi_\hbar(t) \chi_\hbar(t') c_\alpha(t) \overline{c_\beta(t')}$, and the oscillatory phase $\phi_\hbar(t,t')$ is given by
\begin{equation}
\label{e:phase}
\phi_\hbar(t,t') = \frac{i}{\hbar}  (t-t') \alpha_\hbar + \frac{1}{\hbar}(t + t') \beta_\hbar + \frac{i}{\hbar} \big( \Lambda_t - \overline{\Lambda}_{t'} \big) + \varrho_t + \overline{\varrho}_{t'}.
\end{equation}
%Notice that
%$$
%\frac{i}{\hbar}(\Lambda_t - \overline{\Lambda}_{t'}) = \frac{i}{\hbar} \int_{t'}^t \left( \frac{ \dot{q}_s \cdot p_s -  \dot{p}_s \cdot q_s}{2}- \Re \mathcal{P}(z_s) \right) ds - \frac{1}{\hbar} \int_0^t \Im \mathcal{P}(z_s) ds - \frac{1}{\hbar} \int_0^{t'} \Im \mathcal{P}(z_s) ds.
%$$
We aim at showing that this oscillatory integral is small away from the diagonal $t \sim t'$. First, by a simple computation using the definition of Wigner distribution (see Definition \ref{d:wigner_distribution}), and the definition of the Heisenberg-Weyl translation operator (see Definition \ref{d:heisenberg_weyl}) we see that
\begin{align*}
W_\hbar\big[ \varphi_\alpha^\hbar[Z_t,z_t], \varphi_\beta^\hbar[Z_{t'},z_{t'}] \big](z) & \\[0.2cm]
 & \hspace*{-3cm} \exp \left( - \frac{i}{2\hbar} \sigma(z_t,z_{t'}) - \frac{i}{\hbar} \sigma(z,z_t - z_{t'}) \right) W_\hbar \big[ \varphi_\alpha^\hbar[Z_t], \varphi_\beta^\hbar[Z_{t'}] \big](z - \mathbf{z}(t,t')),
\end{align*}
where $\mathbf{z}(t,t') = \frac{1}{2} ( z_t + z_{t'})$ and $\sigma(\cdot,\cdot)$ is the standard symplectic product:
$$
\sigma(z,z') = z \cdot \Omega z', \quad z,z' \in \R^{2d}.
$$
In particular, we recall, by \eqref{e:ground_state_lift}, that
$$
W_\hbar \big[ \varphi_0^\hbar[Z_t], \varphi_0^\hbar[Z_{t'}] \big](z) = \Phi_{(0,0)}^\hbar [\mathcal{Z}](z) = \frac{1}{(\pi \hbar)^{d}} \det \big( \operatorname{Re} G(t,t') \big)^{1/4} e^{- \frac{1}{\hbar} G(t,t') z \cdot z},
$$
where
$$
G = \frac{1}{2i} \mathcal{P} \mathcal{Q}^{-1}, \quad \mathcal{Z} =  \mathcal{Z}(t,t')= 
\left( \begin{array}{c} 
\mathcal{P} \\[0.1cm]
\mathcal{Q} 
\end{array} \right) = \left( \begin{array}{cc}
\frac{1}{2} \overline{Z}_t & \frac{1}{2} Z_{t'} \\[0.1cm]
- \Omega \overline{Z}_t & \Omega Z_{t'}
\end{array} \right).
$$
Using Proposition \ref{p:wave_packets_phase_space} and testing the Wigner distribution against $a \in \mathcal{C}_c^\infty(\R^{2d})$, we then have
\begin{align*}
\int_{\R^{2d}} W_\hbar\big[ \varphi_\alpha^\hbar[Z_t,z_t], \varphi_\beta^\hbar[Z_{t'},z_{t'}] \big](z) a(z) dz &   \\[0.2cm] 
 & \hspace*{-5cm} = \int_{\R^{2d}} \exp \left( \frac{i}{2\hbar} \sigma(z_t,z_{t'}) - \frac{i}{\hbar} \sigma(z,z_t - z_{t'}) \right)\Phi_{(\alpha,\beta)}^\hbar[\mathcal{Z}] (z) a(z + \mathbf{z}(t,t') ) dz \\[0.2cm]
 & \hspace*{-5cm} =  \int_{\R^{2d}} \exp \left( \frac{i}{2\hbar} \sigma(z_t,z_{t'}) - \frac{i}{\sqrt{\hbar}} \sigma(z,z_t - z_{t'}) \right)  \Phi_{(\alpha,\beta)}^1[\mathcal{Z}](z) a(\sqrt{\hbar} z + \mathbf{z}(t,t') ) dz.
\end{align*}
This implies, for the Wigner distribution $W_\hbar[\psi_\hbar,\psi_\hbar]$, that
\begin{align*}
\int_{\R^{2d}} a(z) W_\hbar[\psi_\hbar, \psi_\hbar](z) dz & \\[0.2cm]
 & \hspace*{-3.7cm}  =   \Theta_\hbar \sum_{\alpha,\beta \in \mathbb{N}^d} \int_{\R^{2}}  \varsigma_{\alpha\beta}^\hbar(t,t')  e^{\phi_\hbar(t,t')} \int_{\R^{2d}} e^{-\frac{i}{\sqrt{\hbar}} z \cdot \Omega(z_t - z_{t'}) } \Phi_{(\alpha,\beta)}^1(z) a(\sqrt{\hbar} z + \mathbf{z}(t,t') ) dz dt dt'.
\end{align*}
To study this oscillatory integral, we first look at the integral in $z$. We have, by Taylor's theorem,
\begin{align*}
\int_{\R^{2d}} e^{\frac{i}{\sqrt{\hbar}} z\cdot \Omega (z_t - z_{t'}) } \Phi_{(\alpha,\beta)}^1[\mathcal{Z}](z) a(\sqrt{\hbar} z + \mathbf{z}(t,t') ) dz & \\[0.2cm]
 & \hspace*{-5cm} = \int_{\R^{2d}} e^{\frac{i}{\sqrt{\hbar}} z\cdot \Omega (z_t - z_{t'}) } \Phi_{(\alpha,\beta)}^1[\mathcal{Z}](z) \Big( a(\mathbf{z}(t,t')) + \sqrt{\hbar}R_a(z) \Big) dz \\[0.2cm]
 & \hspace*{-5cm} = a(\mathbf{z}(t,t')) \widehat{\Phi}_{(\alpha,\beta)}^1[\mathcal{Z}]\left( \frac{ \Omega (z_t - z_{t'})}{\sqrt{\hbar}} \right) + \sqrt{\hbar} \mathcal{F} \big[ \Phi_{(\alpha,\beta)}^1[\mathcal{Z}] R_a \big] \left( \frac{ \Omega (z_t - z_{t'})}{\sqrt{\hbar}} \right),
\end{align*}
where $\mathcal{F}$ denotes the Fourier transform on $\R^{2d}$ and $R_a$ is the Taylor remainder,
\begin{equation}
\label{e:taylor_remainder_a}
R_a(z) =  z \cdot \int_0^1 \nabla a\big( \mathbf{z}(t,t') + s \sqrt{\hbar} z \big) ds.
\end{equation}
In order to study the integral in $(t,t') \in \R^2$, which is localized by the bump function $\chi_\hbar$, we expand by Taylor in $t'$ near $t$, so that:
\begin{align*}
z_{t'} & = z_t + (t'-t) \dot{z}_t + O(\vert t- t' \vert^2), \\[0.2cm]
\Lambda_{t'} & = \Lambda_t + (t'-t) \dot{\Lambda}_t + O(\vert t - t' \vert^2).
\end{align*}
We then obtain the following expressions for the terms appearing in the phase $\phi_\hbar(t,t')$:
\begin{align*}
\frac{i}{2\hbar} \sigma(z_t, z_t) & = 0 \\[0.2cm]
\frac{i}{2\hbar} \sigma(z_t, \dot{z}_t) & = i\cdot  \frac{ p_t \cdot \dot{q}_t - \dot{p}_t \cdot q_t}{2\hbar}, \\[0.2cm]
\frac{i}{\hbar} (\Lambda_t - \overline{\Lambda}_{t'}) & = -\frac{2}{\hbar} \int_0^t \Im P(z_s) ds + \frac{i}{\hbar} (t-t') \left(  \frac{p_t \cdot \dot{q}_t  - \dot{p}_t \cdot q_t}{2} + P(z_t)  \right) + \frac{1}{\hbar} O(\vert t - t' \vert^2), \\[0.2cm]
\varrho_t + \overline{\varrho}_{t'} & = O(\vert t - t' \vert).
\end{align*}
Plugging this in the definition of $\phi_\hbar(t,t')$ yields 
\begin{align*}
\hbar \, \phi_\hbar(t,t')  = - i(t-t') \alpha_\hbar + (t+t') \beta_\hbar - 2 \int_0^t \Im P(z_s) ds + i (t-t') \overline{P}(z_t) + O(\vert t - t' \vert^2).
\end{align*}
In addition, making the change $r = \frac{t'-t}{\sqrt{\hbar}}$ we  get
$$
\hbar \, \phi_\hbar(t,t + \sqrt{\hbar} r)  = i\sqrt{\hbar}r \alpha_\hbar + (2t + \sqrt{\hbar} r) \beta_\hbar - 2 \int_0^t \Im P(z_s) ds - i \sqrt{\hbar}r  \overline{P}(z_t) + O(\hbar^2 r^2).
$$
We next expand by Taylor in $t$ near $t =0$, and use that
\begin{align*}
\dot{z}_0 & = \Omega \Re \nabla P(z_0), \\[0.2cm]
\ddot{z}_0 & = \big[ \Omega \Re \partial^2 P(z_0) \big]  \Omega \Re \nabla P(z_0)  +  \big[ \Im \partial^2 P(z_0) \big]\Omega \Re \nabla P(z_0),
\end{align*}
to obtain, modulo terms of order $O(t^3)$,
\begin{align*}
\Re P(z_t)  & = \Re P(z_0) + t \dot{z}_0 \cdot \nabla \Re P(z_0) + \frac{t^2}{2} \big( \ddot{z}_0 \cdot \Re \nabla P(z_0) + \dot{z}_0 \cdot \Re \partial^2 P(z_0) \dot{z}_0 \big) \\[0.2cm]
 & =  \Re P(z_0) + \frac{t^2}{2} \Re \nabla P(z_0)  \cdot \big[ \Im \partial^2 P(z_0) \big] \Omega \Re \nabla P(z_0), \\[0.2cm]
 \Im P(z_t) & = \Im P(z_0) + t \dot{z}_0 \cdot \nabla \Im P(z_0) + \frac{t^2}{2} \big( \ddot{z}_0 \cdot \Im \nabla P(z_0) + \dot{z}_0 \cdot \Im \partial^2 P(z_0) \dot{z}_0 \big) \\[0.2cm]
  & = \frac{t^2}{2} \Omega \Re \nabla P(z_0) \cdot \Im \partial^2 P(z_0) \Omega \Re \nabla P(z_0).
\end{align*}
Therefore, making the change $t = \hbar^{1/3}s$, and taking $\alpha_\hbar = \Re P(z_0) = V(z_0)$, we obtain
\begin{align*}
 \phi_\hbar(\hbar^{1/3}s, \hbar^{1/3}s + \hbar^{1/2} r) & = \widetilde{\phi}_\hbar(s)  + O(\hbar^{1/6}rs^2),
\end{align*}
where 
$$
\widetilde{\phi}_\hbar(s) := \frac{2s \beta_\hbar}{\hbar^{2/3}} - \frac{ s^3\gamma_0}{3},
$$
and the constant $\gamma_0 = \gamma_0(V,A,z_0)$ is given by:
\begin{align*}
\gamma_0 & =\big \langle \Omega \Re \nabla P(z_0) , \Im \partial^2 P(z_0) \Omega \Re \nabla P(z_0) \big \rangle \\[0.2cm]
 & = \big \langle \Omega  \nabla V(z_0) , \partial^2 A(z_0) \Omega  \nabla V(z_0) \big \rangle,
\end{align*}
which is positive due to  condition \eqref{e:non-degenerate}. Thus, denoting $\mathcal{Z} = \mathcal{Z}(t,t')$ and $\mathcal{Z}_0 = \mathcal{Z}(0,0)$, we obtain
\begin{align*}
 \frac{C_\hbar(N) \vert \nabla V (z_0) \vert}{\hbar^{5/6} \sqrt{\pi}} \int_{\R^2}  \varsigma_{\alpha\beta}^\hbar(t,t')  e^{\phi_\hbar(t,t')} \int_{\R^{2d}} e^{-\frac{i}{\sqrt{\hbar}} z \cdot \Omega(z_t - z_{t'}) } \Phi_{(\alpha,\beta)}^1[\mathcal{Z}](z) a(\sqrt{\hbar} z + \mathbf{z}(t,t') ) dz dt' dt & \\[0.2cm]
  & \hspace*{-14cm} = \frac{C_\hbar(N)\vert \nabla V (z_0) \vert \hbar^{1/3}}{\sqrt{\pi}} \int_{\R}  \varsigma_{\alpha\beta}^\hbar(0,0) e^{\widetilde{\phi}_\hbar(s)} \left( \int_\R a(z_0)  \widehat{\Phi}_{(\alpha,\beta)}^1[\mathcal{Z}_0 ] \big( r \Omega \dot{z}_0 \big) dr  + \mathcal{R}_{\alpha\beta}(\hbar) \right) ds \\[0.2cm]
  &\hspace*{-14cm} = \frac{C_\hbar(N) \vert \nabla V (z_0) \vert \hbar^{1/3}}{\sqrt{\pi}}  \int_{\R}  \varsigma_{\alpha\beta}^\hbar(0,0) e^{\widetilde{\phi}_\hbar(s)} \left( \int_\R a(z_0) \widehat{\Phi}_{(\alpha,\beta)}^1[\mathcal{Z}_0 ]    \big( r \nabla V(z_0) \big) dr  + \mathcal{R}_{\alpha \beta}(\hbar) \right) ds,
\end{align*}
where the remainder term $\mathcal{R}_{\alpha \beta}(\hbar)$ satisfies, for $v = \nabla V(z_0) + O(\hbar^{1/3})$: 
\begin{align*}
\vert \mathcal{R}_{\alpha \beta}(\hbar) \vert & \leq C \hbar^{1/6} \sup_{\vert t \vert, \vert t' \vert \leq 3 \hbar^{1/3} L_\hbar} \vert a( \mathbf{z}(t,t')) \vert \int_\R \left \vert \widehat{ \Phi}_{\alpha,\beta}^1[\mathcal{Z}](rv) \right \vert dr \\[0.2cm]
 & \quad + C \hbar^{2/3} \sup_{\vert t \vert , \vert t' \vert \leq 3 \hbar^{1/3} L_\hbar}  \int_\R \big \vert \mathcal{F}\big[ \Phi_{(\alpha,\beta)}^1[\mathcal{Z}] R_a \big] (rv ) \big \vert dr,
\end{align*}
where $R_a$ is given by \eqref{e:taylor_remainder_a}. On the one hand, using Lemma \ref{e:estimate_L1_norm} for the lifted Hagedorn wave-packet $\Phi_{(\alpha,\beta)}^1[\mathcal{Z}]$, we have
$$
\sup_{\vert t \vert, \vert t' \vert \leq 3 \hbar^{1/3} L_\hbar} \vert a( \mathbf{z}(t,t')) \vert \int_\R \left \vert \widehat{ \Phi}_{\alpha,\beta}^1[\mathcal{Z}](rv) \right \vert dr \leq \Vert a \Vert_{L^\infty(K_0)} C^{\vert \alpha \vert + \vert \beta \vert},
$$
where $K_0 \subset \R^{2d}$ is a fixed compact set containing $z_0$.
Moreover, using that
\begin{equation}
\label{e:convolution}
\mathcal{F}\big[ \Phi_{(\alpha,\beta)}^1[\mathcal{Z}] R_a \big] =  \sum_{j=1}^{2d} -i \partial_{w_j} \mathcal{F} \big[ \Phi_{(\alpha,\beta)}^1[\mathcal{Z}] \big] * \mathcal{F}[  { r^j_a}],
\end{equation}
where the remainder term $r^j_a$ is given by
$$
r^j_a(z) = \int_0^1 \partial_{z_j} a( \mathbf{z}(t,t') + s \sqrt{\hbar}z ) ds,
$$
we have, using Lemma \ref{e:estimate_L1_norm} and Remark \ref{r:with derivatives} for the lifted Hagedorn wave-packet $\Phi_{(\alpha,\beta)}^1[\mathcal{Z}]$ with $v = \nabla V(z_0) + O(\hbar^{1/3})$, and Young's convolution inequality:
\begin{align*}
\sup_{\vert t \vert , \vert t' \vert \leq \hbar^{1/3}}  \int_\R \big \vert \mathcal{F}\big[ \Phi_{(\alpha,\beta)}^1[\mathcal{Z}] R_a \big] (rv ) \big \vert dr & \\[0.2cm]
 & \hspace*{-3cm} \leq \Vert \mathcal{F}( \nabla a) \Vert_{L^1(\R^{2d})} \sup_{\vert t \vert , \vert t' \vert \leq 3 \hbar^{1/3} L_\hbar }  \int_\R \big \vert \partial_{w_j} \mathcal{F}\big[ \Phi_{(\alpha,\beta)}^1[\mathcal{Z}] \big](rv) \big \vert dr \\[0.2cm]
 & \hspace*{-3cm} \leq \Vert \mathcal{F}( \nabla a) \Vert_{L^1(\R^{2d})} C^{\vert \alpha \vert + \vert \beta \vert}.
\end{align*}
This implies,  using Proposition \ref{p:final_estimate_coefficients}, that we can sum in $(\alpha,\beta) \in \mathbb{N}^{2d}$, that is:
$$
\sum_{(\alpha, \beta) \in \mathbb{N}^{2d}} \vert \mathcal{R}_{\alpha\beta}(\hbar) \vert \leq C_a \hbar^{1/6}  \sum_{(\alpha, \beta) \in \mathbb{N}^{2d}} \sup_{\vert t \vert , \vert t' \vert \leq 3 \hbar^{1/3} L_\hbar} \left(  \vert c_\alpha(t) \vert \vert c_\beta(t')  \vert C^{\vert \alpha\vert + \vert \beta \vert} \right) = O(\hbar^{1/6}),
$$
provided that $(c_\alpha(t))_\alpha \in \ell_{\rho - 3\sigma}(\mathbb{N}^d)$ for $\rho - 3\sigma > 0$ sufficiently large, and $\hbar \leq \hbar_0$ for $\hbar_0$ sufficiently small. Finally, since
$$
\frac{\vert \nabla V(z_0) \vert}{\sqrt{\pi}} \int_{\R} \widehat{\Phi}_{(0,0)}^1[\mathcal{Z}_0 ] \big( r \nabla V(z_0) \big) dr = 1,
$$
the claim of the proposition holds.
\end{proof}

{ \begin{prop}
\label{p:chose_constants} Assume that $C_\hbar(N) > 0$ is chosen so that $I_\hbar = 1$ for $\beta_\hbar$ satisfying \eqref{e:range_beta_h}. Then there exists a constant $c_0 = c_0(\gamma_0) > 0$ such that, for every $N \geq 0$, $C_\hbar(N)$ satisfies the following estimate, for $\hbar \leq \hbar_0(N,\gamma_0)$ sufficiently small:
\begin{equation}
\label{e:estimate_C_h_N}
C_\hbar(N) \leq c_0 \, \hbar^{-1/3} \left( 1 + \beta_\hbar \hbar^{-2/3} \right) \exp \left( -  \frac{c_0 \beta_\hbar^{3/2}}{\hbar} \right).
\end{equation}
\end{prop}

\begin{proof}
%Denoting $\widetilde{\chi}_\hbar(s) = \chi_\hbar (\hbar^{1/3} s)$, we take $C_\hbar(N)$ so that the integral
%\begin{equation}
%\label{e:non-selfadjoint_integral}
%I_\hbar = C_\hbar(N) \int_\R  \widetilde{\chi}_\hbar(s)^2 e^{\widetilde{\phi}_\hbar(s)} ds = C_\hbar(N)  \int_\R  \widetilde{\chi}_\hbar(s)^2 \exp \left( \frac{2s \beta_\hbar}{\hbar^{2/3}} - \frac{s^3 c_0}{3} \right) ds
%\end{equation}
%is equal to one, provided that $\beta_\hbar$ is given by \eqref{e:choose_beta}. By definition of the bump function $\widetilde{\chi}_\hbar$,
%\begin{align}
%\label{e:beta_estimate_proof}
%\supp \widetilde{\chi}_\hbar & \subset \{ - 1 \leq  s  \leq \hbar^{-\epsilon} \}, \\[0.2cm]
%\label{e:support:derivative}
%\supp \widetilde{\chi}'_\hbar & \subset \{ -1 \leq s \leq -1 + c \} \cup \{ \hbar^{-\epsilon}- c \leq  s  \leq \hbar^{-\epsilon} \}.
%\end{align}
Let us denote $b_\hbar : = \beta_\hbar \hbar^{-2/3}$. Assume first that $b_\hbar \geq 1$. Then the function $\exp \left( 2sb_\hbar - \frac{s^3 \gamma_0}{3} \right)$ reaches its maximum (for $s > 0$) at $L_\hbar$ given by  \eqref{e:minimum_function}. Moreover,
\begin{equation}
\label{e:estimate_exotic_integral}
 \exp \left( 2s b_\hbar - \frac{s^3 \gamma_0}{3} \right) \geq \exp \left( \frac{4 b_\hbar s}{3} \right), \quad \text{for} \quad 0 \leq s \leq L_\hbar.
\end{equation}
Then there exists $c_0 = c_0(\gamma_0) > 0$ such that
\begin{align*}
\int_\R  \chi(s/L_\hbar)^2 \exp \left( 2sb_\hbar - \frac{s^3 \gamma_0}{3} \right) ds & \geq \int_0^{L_\hbar} e^{\frac{4}{3} b_\hbar s } ds \\[0.2cm]
  & = \frac{1}{b_\hbar} \int_0^{b_\hbar L_\hbar} e^{\frac{4}{3}s} ds \\[0.2cm]
  & \geq  \frac{c_0}{b_\hbar} \exp \left(  \frac{c_0 \beta_\hbar^{3/2}}{\hbar}\right).
\end{align*}
Otherwise, if $0 \leq \beta_\hbar \leq \hbar^{2/3}$, there exists $c_0 > 0$ such that
$$
\int_\R  \chi(s/L_\hbar)^2 \exp \left( \frac{2s \beta_\hbar}{\hbar^{2/3}} - \frac{s^3 \gamma_0}{3} \right) ds \geq c_0 > 0.
$$
Then, using \eqref{e:I_h}, the claim holds true.
\end{proof}
}
\begin{proof}[Proof of Theorem \ref{t:T1}] Let $\beta_\hbar$ satisfy \eqref{e:range_beta_h} and $C_\hbar(N)$ such that $I_\hbar = 1$, then by Propositions \ref{p:wigner_function} and \ref{p:chose_constants},
$$
W_\hbar[\psi_\hbar,\psi_\hbar] \rightharpoonup^* \delta_{z_0}.
$$
This shows \eqref{e:weak_limit}. Moreover, taking $a \equiv 1$, we observe that the sequence $(\psi_\hbar)$ is asymptotically normalized in $L^2(\R^d)$.

It remains to show that the sequence $(\psi_\hbar)$ defines a quasimode of width $O\big( \hbar^{2/3}\exp(-\beta_\hbar^{3/2}/C_0\hbar) \big)$ for $\widehat{P}_\hbar$, that is,
$$
\widehat{P}_\hbar \psi_\hbar = (\alpha_\hbar + i \beta_\hbar) \psi_\hbar + O\big( \hbar^{2/3}\exp(-\beta_\hbar^{3/2}/C_0\hbar) \big).
$$
To this aim, observe that, by the decomposition
$$
P(z) = P_2(t,z) + \chi( \vert F_t^{-1} (z-z_t) \vert^2) P_N(t,z) + (1- \chi)P_N(t,z) + R_N(t,z),
$$
and by Lemma \ref{l:weyl_to_antiwick}:
\begin{align*}
\widehat{P}_\hbar \psi_\hbar & = \Theta_\hbar^{1/2} \int_\R \chi_\hbar(t) e^{-\frac{it}{\hbar}( \alpha_\hbar + i \beta_\hbar)}  \big( \Op_\hbar(P_2) + \Op_{\hbar,Z_t}^{\operatorname{AW}} (\widetilde{P}_N) \big)  \varphi_\hbar(t,x) dt \\[0.2cm]
 & \quad + \Theta_\hbar^{1/2} \int_\R \chi_\hbar(t) e^{-\frac{it}{\hbar}( \alpha_\hbar + i \beta_\hbar)} \Op_\hbar \big((1- \chi)P_N(t,z) + R_N(t,z) \big) \varphi_\hbar(t,x) dt + O(\hbar^{N+1}).
\end{align*}
Using that $\varphi_\hbar(t,x)$ solves equation \eqref{e:simplified_non_selfadjoint_problem_2} and integration by parts in $t$ yields:
\begin{align*}
(\widehat{P}_\hbar - \lambda_\hbar ) \psi_\hbar & = i \hbar \Theta_\hbar^{1/2} \int_\R \chi'_\hbar(t) e^{-\frac{it}{\hbar}( \alpha_\hbar + i \beta_\hbar)}  \varphi_\hbar(t,x) dt \\[0.2cm]
& \quad + \Theta_\hbar^{1/2} \int_\R \chi_\hbar(t) e^{-\frac{it}{\hbar}( \alpha_\hbar + i \beta_\hbar)}\Op_\hbar  \big((1- \chi)P_N(t,z) + R_N(t,z) \big) \varphi_\hbar(t,x) dt + O(\hbar^{N+1}).
\end{align*}
To estimate the second term of the right-hand side by $O(\hbar^{N+1})$, we repeat the argument to estimate the Wigner distribution with $(1- \chi)P_N(t,z) + R_N(t,z)$ replacing $a$. Notice that the hypothesis $V,A \in S^k(\R^{2d})$ is necessary to bound the term \eqref{e:convolution} with $R_N$ or $(1- \chi)P_N$ instead of $a$, considering a higher order Taylor expansion near $\mathbf{z}(t,t')$ and a higher order Taylor remainder replacing \eqref{e:taylor_remainder_a}, and using Lemma \ref{e:estimate_L1_norm} and Remark \ref{r:with derivatives}.

Finally, to estimate the rest of the remainder term, we repeat the argument above with $\chi_\hbar'(t)$ instead of $\chi_\hbar(t)$, to obtain:
$$
\big \langle (\widehat{P}_\hbar - \lambda_\hbar) \psi_\hbar, \psi_\hbar \big \rangle_{L^2(\R^d)} = \frac{ i \hbar C_\hbar(N)}{L_\hbar} \int_{- \infty}^\infty \chi'(s/L_\hbar) \chi(s/L_\hbar)  e^{\widetilde{\phi}_\hbar(s)} ds \left( 1 + O(\hbar^{1/6}) \right) + O(\hbar^N).
$$
{ We get
\begin{align*}
\left \vert \frac{ \hbar C_\hbar(N)}{L_\hbar} \int_{- \infty}^\infty \chi'(s)\chi(s) \exp \left( \frac{2s \beta_\hbar}{\hbar^{2/3}} - \frac{ s^3 \gamma_0}{3} \right) ds \right \vert & \\[0.2cm]
& \hspace*{-3cm}  \leq  \frac{ \hbar C_\hbar(N)}{L_\hbar} \int_{-L_\hbar \leq s \leq -L_\hbar/2} \exp \left( \frac{2s \beta_\hbar}{\hbar^{2/3}} - \frac{ s^3 \gamma_0}{3} \right) ds \\[0.2cm]
& \hspace*{-2.5cm} + \frac{\hbar C_\hbar(N)}{L_\hbar} \int_{3L_\hbar/2 \leq s \leq 2 L_\hbar} \exp \left( \frac{2s \beta_\hbar}{\hbar^{2/3}} - \frac{ s^3 c_0}{3} \right) ds \\[0.2cm]
& \hspace*{-3cm} \leq c_0 \hbar^{2/3} \exp \left( - \frac{\beta_\hbar^{3/2}}{C_0 \hbar} \right),
\end{align*}
for some $C_0 = C_0(\gamma_0) > 0$, where the last inequality holds due to \eqref{e:estimate_C_h_N} and the fact that the function $\exp\left( \frac{2s \beta_\hbar}{\hbar^{2/3}} - \frac{ s^3 c_0}{3} \right)$, assuming $\beta_\hbar \geq \hbar^{2/3}$, reaches its minimum for $s < 0$ at $-L_\hbar$, and satisfies
$$
\int_{-L_\hbar \leq s \leq -L_\hbar/2} \exp \left( \frac{2s \beta_\hbar}{\hbar^{2/3}} - \frac{ s^3 \gamma_0}{3} \right) ds \leq \exp \left( - \frac{ 11 \sqrt{2} \beta_\hbar^{3/2}}{12\hbar \sqrt{\gamma_0}} \right) \int_{-L_\hbar \leq s \leq -L_\hbar/2} ds,
$$ 
while it reaches its maximum for $s > 0$ at $L_\hbar$, and satisfies
$$
\int_{2L_\hbar \leq s \leq 3 L_\hbar} \exp \left( \frac{2s \beta_\hbar}{\hbar^{2/3}} - \frac{ s^3 c_0}{3} \right) ds \leq \exp\left( - \frac{7 \sqrt{2} \beta_\hbar^{3/2}}{3\hbar \sqrt{\gamma_0}} \right)  \int_{3L_\hbar/2 \leq s \leq 2 L_\hbar}  ds.
$$
Then the claim holds.}

\end{proof}

\begin{remark}
If we mimic our proof assuming the point $z_0 \in \R^{2d}$ satisfies the Hörmander bracket condition
$\gamma_0 = \{ V, A \}(z_0) < 0$, then the strategy works the same. It appears the phase function
$$
\exp \left( \frac{2s\beta_\hbar }{\hbar^{1/2}} - \frac{s^2 \gamma_0}{2} \right)
$$
replacing $\widetilde{\phi}_\hbar(s)$, which has exponential decay in both tails, so that $\beta_\hbar \equiv 0$ is enough to obtain normalization. This gives an alternative proof  for \cite[Thm 1.2]{Dencker04}.
\end{remark}

\subsection{Proof of Theorem \ref{t:T2}}

%Let $z_0 \in H^{-1}(1)$ and let $\mathcal{T}(z_0)$ be the minimal invariant torus issued from $z_0$ by the flow $\phi_t^H$. Let $E = M_H(z_0)$ and $\mathcal{T}_E = M_H^{-1}(E)$.  If $E \in X$, then $\mathcal{T}_E$ is Lagrangian, and
%$$
%\Phi_{z_0} : \mathbb{T}^d \to \mathcal{T}_E
%$$
%is a diffeomorphism. Moreover,
%$$
%\Phi_{z_0}^{-1} \circ \phi_t^H \circ \Phi_{z_0}(\tau) = \tau + t \omega, \quad \forall t \in \R.
%$$
All along this section we use the notations of Appendix \ref{a:averaging_method}. The idea of the proof of Theorem \ref{t:T2} is very similar to the one for Theorem \ref{t:T1}, but, roughly speaking, in this case we consider the propagation of a wave-packet $\varphi_0^\hbar[Z_0,z_0]$ by both the quantum flow of the harmonic oscillator $\widehat{H}_\hbar$ and the non-selfadjoint flow generated by $\Op_\hbar(V+ i A)$. 

We now sketch the lines of the proof of Theorem \ref{t:T2}. First of all, it is necessary to conjugate the operator $\mathcal{P}_\hbar$ into its normal form so that the perturbation commutes with $\widehat{H}_\hbar$ up to order $N$. To do this it is necessary to use the Diophantine property \eqref{e:diophantine_property} of $\omega$.  Let us consider the Fourier integral operator $\mathcal{F}_{N,\hbar}$ given by Proposition  \ref{l:first_normal_form} of  Appendix \ref{a:averaging_method}, which conjugates the operator $\widehat{\mathcal{P}}_\hbar$ into its normal form:
\begin{equation}
\label{e:normal_form_in_the_proof}
\widehat{\mathcal{P}}^\dagger_\hbar := \mathcal{F}_{N,\hbar} \big( \widehat{H}_\hbar + \hbar \widehat{V}_\hbar + i \hbar \widehat{A}_\hbar \big) \mathcal{F}^{-1}_{N,\hbar} = \widehat{H}_\hbar + \hbar \Op_\hbar(\mathcal{I}_{P_\hbar}) + \widehat{R}_{N,\hbar},
\end{equation}
where $P_\hbar = V + i A + O_{S^0(\R^{2d})}(\hbar)$, and $\Vert \widehat{R}_{N,\hbar} \Vert_{\mathcal{L}(L^2)} = O(\hbar^{N+1})$.
\medskip

Using the notations of the Appendix \ref{e:averages_and_cohomological} and \eqref{e:different_formulas_average}, we have the following expression for the average of $P$ by the flow $\phi_t^H$:
$$
\mathcal{I}_{P_\hbar} (z) = \int_{\mathbb{T}_\omega} P_\hbar \circ \Phi_z(\tau) \mu_\omega(d\tau).
$$
Considering next the flow $z_t$ given by Lemma \ref{l:center_evolution} with $\mathcal{I}_{P_\hbar}$ replacing $V + i A$, we expand $\mathcal{I}_{P_\hbar}$ by Taylor near $z_t$:
$$
 \mathcal{I}_{P_\hbar}(z) =  P_{2}(t,z) + P_N(t,z) +  R_{N} (t,z),
$$
where $P_{2}$ is the quadratic approximation of $\mathcal{I}_{P_\hbar}$ near the orbit $z_t$, $P_N$ is the rest of the Taylor polynomial up to order $N$, and $R_N$ is the Taylor remainder, similar to \eqref{e:taylor_2}, \eqref{e:taylor_N} and \eqref{e:taylor_R}. We define also
$$
\widetilde{P}_N (t,z) = \sigma_{N,Z_t}^{\operatorname{AW}}\big( \chi\big( \vert F_{t}^{-1}(z-z(t)) \vert^2 \big) P_N \big)(t,z),
$$
where $F_t$ is the symplectic matrix associated with $Z_t$, and the Lagrangian frame $Z_t$ satisfies
$$
Z_t = S_t Z_0 N_t, \quad Z_0 = (i \Id, \Id)^{\mathfrak{t}},
$$
where $N_t$ is given by \eqref{e:normalizing_matrix} and $S_t$ satisfies the linearized equation
$$
\dot{S}_t = - \Omega \partial^2  \mathcal{I}_{P_\hbar} ( z(t)) S_t, \quad S_0 = \Id_{2d}.
$$
We next consider the evolution problem 
\begin{equation}
\label{e:only_for_t}
\big( i \hbar \partial_t  + \Op_\hbar(P_2) + \Op_{\hbar,Z_t}^{\operatorname{AW}}(\widetilde{P}_N) \big) \varphi_\hbar(t,x) = 0, \quad \varphi_\hbar(0,x) = \varphi_0^\hbar[Z_0,z_0](x),
\end{equation}
and we write the solution as
$$
\varphi_\hbar(t,x) = U_\hbar(t) \varphi_0^\hbar[Z_0,z_0](x) = \sum_{\alpha \in \mathbb{N}^d} c_\alpha(t) \varphi_\alpha^\hbar[Z_t,z_t](x),
$$
where $U_\hbar(t)$ denotes the propagator of \eqref{e:only_for_t}.

The next point in the proof is to propagate $\varphi_\hbar(t,x)$ also by the flow of the harmonic oscillator. More precisely, we consider the propagation on the minimal invariant torus $\mathcal{T}_\omega(z_t)$ issued from the point $z_t$ by $\phi_t^H$. To this aim, let us define the moment map:
$$
\Op_\hbar(M_H) := \big( \Op_\hbar(H_1), \ldots, \Op_\hbar(H_d) \big). 
$$
For any $\tau \in \mathbb{T}_{d_0} := \pi_\omega(\mathbb{T}_\omega) \subset \mathbb{T}^d$, where $\mathbb{T}_\omega$ is defined by \eqref{e:torus} and $\pi_\omega$ by \eqref{e:projection_torus}, we consider the propagated states:
\begin{align}
\label{e:double_propagated_solution}
\varphi_\hbar(\tau,t,x) & := \exp \left( \frac{i \tau \cdot \Op_\hbar(M_H)}{\hbar} \right)  \sum_{\alpha \in \mathbb{N}^d} c_\alpha(t) \varphi^\hbar_\alpha[Z_t, z(t)](x) \\[0.2cm]
 & =  \sum_{\alpha \in \mathbb{N}^d} c_\alpha(\tau, t) \varphi^\hbar_\alpha[Z(\tau,t), z(\tau,t)](x),
\end{align}
where $z(\tau,t) := \Phi_{z(t)}(\tau)$, the normalized Lagrangian frame $Z(\tau,t)$ obeys the differential equation
\begin{equation}
\label{e:quadratic_part_harmonic_oscillator}
\partial_{t_j} Z(\tau,t) = -\Omega \partial^2 H_j(z_t) Z(\tau,t), \quad Z(0,t) = Z_t, \quad \tau = (t_1, \ldots, t_d) \in \R^d,
\end{equation}
and
$$
c_\alpha(\tau,t) = e^{-\frac{i \vert \tau \vert_1}{2}} c_\alpha(t).
$$
Let us fix the constant $\Theta_\hbar$ given by
$$
\Theta_\hbar := \frac{ C_\hbar(N) \sqrt{ \det \mathcal{G}_{z_0}}}{\hbar^{5/6} \sqrt{\pi^{d_0 + 1}}}, \quad \mathcal{G}_{z_0} = \big[  D_{\tau,t}  z(\tau,t) \vert_{(\tau,t) = (0,0)} \big] \big[ D_{\tau,t}  z(\tau,t) \vert_{(\tau,t) = (0,0)} \big]^T,
$$ 
where $D_{\tau,t} z(\tau,t)$ denotes the differential with respect to $(\tau,t)\in \mathbb{T}_{d_0} \times \R$, and the constant $C_\hbar(N)$ is chosen as in the proof of Theorem \ref{t:T1}. Precisely, denoting $\mu_\omega^{z_0} := (\pi_{z_0})_* \mu_\omega$, we set:
\begin{equation}
\label{e:intermediate_quasimode}
\psi_\hbar(x) :=\Theta_\hbar^{1/2}  \int_{\mathbb{T}_{d_0}} \int_\R \chi_\hbar(t) e^{\frac{i}{\hbar} \tau \cdot E_\hbar} e^{-it(\alpha_\hbar + i \beta_\hbar)} \varphi_\hbar(\tau,t,x) dt \mu_\omega^{z_0}(d\tau),
\end{equation}
where $\chi_\hbar(t) \in \mathcal{C}_c^\infty(\R)$ is defined as in the proof of Theorem \ref{t:T1}.  Moreover, we take $\alpha_\hbar = \mathcal{I}_V(z_0)$ and $\beta_\hbar$ as in the proof of Theorem \ref{t:T1}. In addition, we choose the vector $E_\hbar$ as
$$
E_\hbar = \hbar \left( N_1(\hbar)+\frac{1}{2}, \ldots, N_d(\hbar) + \frac{1}{2} \right),
$$
with vector of integers $(N_1(\hbar), \ldots, N_d(\hbar)) \in \mathbb{N}_0^d$ taken so that $E_\hbar = M_H(z_0) + O(\hbar)$. This can be done due to the explicit structure of the spectrum $\eqref{e:eigenvalue_oscillator}$ of $\widehat{H}_\hbar$, see  \cite[Lemma 1]{Ar_Mac18}. We will show that $\psi_\hbar$ is a quasimode of width $O\big( \hbar^{2/3}\exp(-\beta_\hbar^{3/2}/C_0\hbar) \big)$ for $\widehat{\mathcal{P}}_\hbar^\dagger$.  Finally, our  quasimode $\psi_\hbar^\dagger$ for $\widehat{\mathcal{P}}_\hbar$ will be defined by 
\begin{equation}
\label{e:final_quasimode}
\psi_\hbar^\dagger := \frac{ \mathcal{F}_{N,\hbar}^{-1} \psi_\hbar}{\Vert \mathcal{F}_{N,\hbar}^{-1} \psi_\hbar \Vert_{L^2(\R^d)}}.
\end{equation}

\begin{prop}
The Wigner measure $W_\hbar[\psi_\hbar]$ satisfies, for any $a \in \mathcal{C}_c^\infty(\R^{2d})$:
$$
\int_{\R^{2d}} a(z) W_\hbar[\psi_\hbar](z) dz = \int_{\mathbb{T}_\omega} a\circ \Phi_{z_0}(\tau) \mu_\omega(d\tau) + O(\hbar^{1/6}).
$$
\end{prop}
\begin{proof}

By Egorov's theorem, and since $[\widehat{H}_\hbar, \Op_\hbar( \mathcal{I}_{P_\hbar})] = 0$, we  have that $\varphi_\hbar(\tau,t,x)$ given by \eqref{e:double_propagated_solution} satisfies:
$$
\varphi_\hbar(\tau,t,x) = U_\hbar(\tau,t) \exp \left( \frac{i \tau \cdot \Op_\hbar(M_H)}{\hbar} \right)  \varphi^\hbar_0[Z_0, z_0](x),
$$
where $U_\hbar(\tau,t)$ denotes the propagator of the evolution equation
$$
\big( i \hbar \partial_t  + \Op_\hbar(P_2(\tau,t,z)) +  \Op_{\hbar,Z(\tau,t)}^{\operatorname{AW}}(\widetilde{P}_N(\tau,t,z)) \big) \varphi_\hbar(t,x) = 0,
$$
where $P_2(\tau,t,z)$ denotes the quadratic approximation of $\mathcal{I}_{P_\hbar}$ at $z(\tau,t)$ and, let $P_N(\tau,t,z)$ be the rest of the Taylor polynomial up to order $N$ at $z(\tau,t)$, the symbol $\widetilde{P}_N(\tau,t,z)$ is given by:
\begin{align*}
\widetilde{P}_N(\tau,t,z) & = \sigma_{N,Z(\tau,t)}^{\operatorname{AW}} \big( \chi\big( \vert F(\tau,t)^{-1}(z- z(\tau,t)) \vert^2 \big) P_N \big)(\tau, t,z).
\end{align*}
Notice in particular that, by \eqref{e:quadratic_part_harmonic_oscillator}, the symplectic matrix $F(\tau,t)$ corresponding to the Lagrangian frame $Z(\tau,t)$ satisfies that $F(\tau,t)^{-1} F(\tau,t)^{-T} = F_t^{-1} F_t^{-T}$. 

We define also the centered-at-zero function
\begin{equation}
\label{e:zero_states}
\varphi_\hbar^0(\tau,t,x) := \sum_{\alpha \in \mathbb{N}^d} c_\alpha(\tau, t) \varphi^\hbar_\alpha[Z(\tau,t)](x).
\end{equation}
With these assumptions, the computation of the Wigner distribution
$$
\big \langle \Op_\hbar(a) \psi_\hbar, \psi_\hbar \big \rangle_{L^2(\R^d)} = \int_{\R^{2d}} W_\hbar[\psi_\hbar,\psi_\hbar](z) a(z) dz
$$
is carried out by analogous arguments as those of the proof of Theorem \ref{t:T1} and \cite[Lemmas 6.1 and 7.1]{Ar_Mac18}, provided that $[\widehat{H}_\hbar, \Op_\hbar(\mathcal{I}_{P_\hbar})] = 0$.

Denoting $\mathbf{t} = (\tau, t)$ and $d\mathbf{t} = dt \otimes \mu_\omega^{z_0}(d\tau)$ for shortness, we have
\begin{align*}
\int_{\R^{2d}} W_\hbar[\psi_\hbar,\psi_\hbar](z) a(z) dz & \\[0.2cm]
 & \hspace*{-3cm} = \Theta_\hbar \int_{\mathbb{T}_{d_0}^2} \int_{\R^2} \int_{\R^{2d}} \chi_\hbar(t) \chi_\hbar(t')e^{-\frac{i}{\hbar}(\mathbf{t} - \mathbf{t}') \cdot (E_\hbar, \alpha_\hbar)}  e^{ \frac{\beta_\hbar}{\hbar} (t+t') } W_\hbar[ \varphi_\hbar(\mathbf{t}),\varphi_\hbar(\mathbf{t}) ]a(z) dzd\mathbf{t} d \mathbf{t}'.
\end{align*}
Using the definition of Wigner function and \eqref{e:zero_states}, we compute:
\begin{align*}
\int_{\R^{2d}} W_\hbar[\psi_\hbar,\psi_\hbar](z) a(z) dz & \\[0.2cm]
 & \hspace*{-3cm} = \Theta_\hbar \int_{ \mathbb{T}_{d_0}^2 \times \R^2} \chi_\hbar(t) \chi_\hbar(t') e^{\phi_\hbar^\dagger(\mathbf{t},\mathbf{t}')} \int_{\R^{2d}} e^{-\frac{i}{\sqrt{\hbar}} z \cdot \Omega(z(\mathbf{t}) - z(\mathbf{t}'))} \mathcal{W}[\mathbf{t},\mathbf{t}'](z) \mathbf{a}(\mathbf{t},\mathbf{t}',z) dz d\mathbf{t} d \mathbf{t}',
\end{align*}
where $\mathbf{a}(\mathbf{t},\mathbf{t}',z) := a (\sqrt{\hbar}z + \mathbf{z}(\mathbf{t}, \mathbf{t}'))$,
$$
\mathcal{W}[\mathbf{t},\mathbf{t}'](z) := W_1[\varphi_0^1(\mathbf{t}), \varphi_0^1(\mathbf{t}')](z),
$$
and the phase function $\phi_\hbar^\dagger(\mathbf{t},\mathbf{t}')$ is given by
$$
\phi_\hbar^\dagger(\mathbf{t},\mathbf{t}') := \frac{i}{\hbar}  (t-t')\alpha_\hbar + \frac{1}{\hbar}(t + t') \beta_\hbar + \frac{i}{\hbar} \big( \Lambda_{\mathbf{t}} - \overline{\Lambda}_{\mathbf{t}'} \big) + \frac{i}{2\hbar} \sigma(z(\mathbf{t}), z(\mathbf{t}'))+  \varrho_{\mathbf{t}} + \overline{\varrho}_{\mathbf{t}'},
$$
where, denoting $z(\tau,t) = (q(\tau,t), p(\tau,t))$,
\begin{align*}
\Lambda_{\mathbf{t}} & =-  \int_0^t \left( \frac{ \partial_s p(\tau,s) \cdot q(\tau,s) - \partial_s q(\tau,s) \cdot p(\tau,s)}{2}  - \mathcal{I}_{P_\hbar}(z(\tau,s)) \right)ds, \\[0.2cm]
\varrho_{\mathbf{t}} & = \frac{i \vert \tau \vert}{2} - \frac{1}{4} \int_0^t \operatorname{tr}\big( G^{-1}(\tau,s) \operatorname{Im} \partial^2 \mathcal{I}_{P_\hbar} (z(\tau,s)) \big) ds, \quad \mathbf{t} = (\tau,t).
\end{align*}

The Wigner distribution $W_\hbar[\psi_\hbar,\psi_\hbar]$ has stationary phase on the diagonal $\mathbf{t} = \mathbf{t}'$ and is highly oscillatory away from it. On the one hand, the integral in $(t,t') \in \R^2$ is computed following the proof of Theorem \ref{t:T1}. At the same time, the integral in $(\tau,\tau') \in \mathbb{T}_{d_0}^2$ also has stationary-phase on the diagonal $\tau = \tau'$ (see \cite[Lemma 1]{Ar_Mac18}). Notice, in particular, that near the diagonal $\vert \mathbf{t} - \mathbf{t}' \vert \leq \epsilon$,
\begin{align*}
\frac{1}{2} \sigma(z(\mathbf{t}), z(\mathbf{t}')) & = (\tau' - \tau) \cdot M_H(z_0) + (t'-t) \frac{ \partial_t q(\tau,t) \cdot p(\tau,t) - \partial_t p(\tau,s) \cdot q(\tau,t)}{2},
\end{align*}
plus lower order terms of size $O(\vert \mathbf{t} - \mathbf{t}' \vert^2)$. Notice also that $(\tau' - \tau) \cdot (E_\hbar -M_H(z_0) ) = O(\hbar \vert \tau - \tau' \vert)$ due to the choice of the eigenvector sequence $E_\hbar$. The rest of the computation  in the region $\vert \mathbf{t} - \mathbf{t}' \vert \leq \epsilon$ can be carried out following the proof of Theorem \ref{t:T1} with these small modifications coming from the quantum flow of the harmonic oscillator. 

On the other hand, observe that $\vert \tau - \tau' \vert \geq \epsilon \Rightarrow \vert z(t, \tau) - z(\tau',t') \vert \geq C \epsilon$. Therefore,
$$
\Theta_\hbar  \int_{\vert \tau - \tau' \vert \geq \epsilon} \int_{\R^2}  \chi_\hbar(t') e^{\phi_\hbar^\dagger(\mathbf{t},\mathbf{t}')} \mathcal{F} [\mathcal{W}[\mathbf{t},\mathbf{t}'] \mathbf{a}] \left( \frac{z(\mathbf{t}) - z(\mathbf{t}')}{\sqrt{\hbar}} \right) d \tau' dt dt'  = O(\hbar^N),
$$
for every $N \geq 1$, where $\mathcal{F}$ denotes the Fourier transform in the variable $z$. 

Using finally that
$$
 \sqrt{ \frac{ \det \mathcal{G}_{z_0}}{\pi^{d_0+1}}} \int_{\R^{d_0+1}} \widehat{\Phi}_{(0,0)}^1[\mathcal{Z}_0 ] \big( D_{\mathbf{t}}(z(\mathbf{t}) \vert_{\mathbf{t} = 0})^T \mathbf{t} \big) d\mathbf{t} = 1,
$$
we obtain that
$$
\int_{\R^{2d}} W_\hbar[\psi_\hbar](z) a(z) dz =  \int_{\mathbb{T}_\omega} a\circ \Phi_{z_0}(\tau) \mu_\omega(d\tau) + O(\hbar^{1/6}),
$$
where
$$
\gamma_0 = \big \langle X_{\mathcal{I}_{V}}(z_0), \partial^2 \mathcal{I}_{A}(z_0) X_{\mathcal{I}_V}(z_0) \big \rangle
$$
is positive due to condition \eqref{e:non-degenerate_2}.
 \end{proof}

\begin{proof}[Proof of Theorem \ref{t:T2}]

Let $\lambda_\hbar^\dagger = \omega \cdot E_\hbar + \hbar( \alpha_\hbar + i \beta_\hbar)$,  notice that 
\begin{align*}
\widehat{H}_\hbar \varphi_\hbar(\tau,t,x) & = \omega \cdot \Op_\hbar( M_H) \varphi_\hbar(\tau,t,x) \\[0.2cm]
 & = i \hbar \omega \cdot \partial_\tau \varphi_\hbar(\tau,t,x),
\end{align*}
and then, by integration by parts in $\tau$ and the definition \eqref{e:intermediate_quasimode} of $\psi_\hbar$, we observe that $\psi_\hbar$ is an eigenfunction for $\widehat{H}_\hbar$ with sequence of eigenvalues given by $\omega \cdot E_\hbar$ which, by definition, converges to one as $\hbar \to 0^+$.  Moreover, since $\Op_\hbar( \mathcal{I}_{P_\hbar})$ commutes with $\widehat{H}_\hbar$, we have that:
\begin{align*}
\Op_\hbar(\mathcal{I}_{P_\hbar}) \varphi_\hbar(\tau,t) \\[0.2cm]
 & \hspace*{-2cm} = \exp \left( \frac{i \tau \cdot \Op_\hbar(M_H)}{\hbar} \right) \Op_\hbar(\mathcal{I}_{P_\hbar}) \sum_{\alpha \in \mathbb{N}^d} c_\alpha(t)  \varphi^\hbar_\alpha[Z_t, z(t)] \\[0.2cm]
 & \hspace*{-2cm} = \exp \left( \frac{i \tau \cdot \Op_\hbar(M_H)}{\hbar} \right)\big(  \Op_\hbar(P_2) + \Op_{\hbar,Z_t}^{\operatorname{AW}}(\widetilde{P}_N) \big) \sum_{\alpha \in \mathbb{N}^d} c_\alpha(t)  \varphi^\hbar_\alpha[Z_t, z(t)] + O(\hbar^{N+1}) \\[0.2cm]
 & \hspace*{-2cm} =  i \hbar \exp \left( \frac{i \tau \cdot \Op_\hbar(M_H)}{\hbar} \right)  \partial_t \sum_{\alpha \in \mathbb{N}^d} c_\alpha(t)  \varphi^\hbar_\alpha[Z_t, z(t)] + O(\hbar^{N+1}).
\end{align*}
Thus, integrating by parts in $t$ as in the end of the proof of Theorem \ref{t:T1} and repeating the stationary phase argument we get:
\begin{align*}
\big \langle (\widehat{\mathcal{P}}_\hbar^\dagger - \lambda^\dagger_\hbar) \psi_\hbar, \psi_\hbar \big \rangle_{L^2(\R^d)} & = \frac{ i \hbar C_\hbar(N)}{L_\hbar} \int_{- \infty}^\infty \chi'(s/L_\hbar) \chi(s/L_\hbar)  e^{\widetilde{\phi}_\hbar(s)} ds \left( 1 + O(\hbar^{1/6}) \right) + O(\hbar^{N+1}) \\[0.2cm]
 & = O \left( \hbar^{2/3} \exp \left( - \frac{\beta_\hbar^{3/2}}{C_0 \hbar} \right)\right) +  O(\hbar^{N+1}).
\end{align*}
This shows that the sequence $(\psi_\hbar, \lambda_\hbar^\dagger)$ defines a quasimode for $\widehat{\mathcal{P}}^\dagger_\hbar$ of the desired width.  Finally, by  \eqref{e:final_quasimode} and \eqref{e:same_semiclassical_measure}, we have that
$$
 \int_{\R^{2d}} a(z) W_\hbar[\psi_\hbar^\dagger](z) dz = \int_{\mathbb{T}_\omega} a\circ \Phi_{z_0}(\tau) \mu_\omega(d\tau) + o(1).
$$
Moreover, by \eqref{e:normal_form_in_the_proof},
$$
(\widehat{\mathcal{P}}_\hbar - \lambda_\hbar^\dagger) \psi_\hbar^\dagger = O \left( \hbar^{2/3} \exp \left( - \frac{\beta_\hbar^{3/2}}{C_0 \hbar} \right)\right) +  O(\hbar^{N+1}).
$$
This concludes the proof.

\end{proof}
 
\appendix

\section{Evolution equations}
\label{a:evolution_equations}

In this Appendix, we give an abstract propagation result on weighted Banach spaces of sequences.  

\begin{definition}
For any $\rho > 0$, we define the weighted Banach space of sequences $\ell_\rho(\mathbb{N}^d)$ as:
$$
\ell_\rho(\mathbb{N}^d) := \left \{ \vec{c} = (c_\alpha)_{\alpha \in \mathbb{N}^d} \, : \, \Vert \vec{c} \, \Vert_\rho := \sum_{\alpha \in \mathbb{N}^d} \vert c_\alpha \vert \exp \left( \rho \vert \alpha \vert \right)  < + \infty \right \}.
$$
\end{definition}

Let us define the following class of bounded operators $\mathcal{A} : { \ell_\rho(\mathbb{N}^d)} \to{\ell_{\rho - \sigma}(\mathbb{N}^d)}$ for every $\rho > 0$ and every $0 < \sigma < \rho$:
\begin{definition}
{ Let $\rho > 0$}. We define the space { $\mathscr{D}_\rho$} of operators $\mathcal{A}$ satisfying:
\begin{enumerate}
\item For every  $0 < \sigma < \rho$, $\mathcal{A} : \ell_\rho(\mathbb{N}^d) \to \ell_{\rho-\sigma}(\mathbb{N}^d)$ is continuous.
\medskip

\item There exists { $C_\rho > 0$} such that, for every $0 < \sigma < \rho$, and every $\vec{c} \in \ell_\rho(\mathbb{N}^d)$,
\begin{equation}
\label{e:operator_norm}
\Vert \, \mathcal{A} \vec{c} \, \Vert_{\rho - \sigma} \leq \frac{C_\rho}{ e\sigma} \Vert \, \vec{c} \, \Vert_\rho.
\end{equation}
We denote by { $\Vert \mathcal{A} \Vert_{\mathscr{D}_\rho}$} the infimum of the constants $C_\rho$ satisfying \eqref{e:operator_norm}.
\end{enumerate}
\end{definition}

\begin{example} Let us consider an operator $\mathcal{A} : \ell_\rho(\mathbb{N}^d) \to \ell_{\rho - \sigma}(\mathbb{N}^d)$ such that
$$
\big \vert ( \mathcal{A} \vec{c} \, )_\alpha \big \vert \leq \vert \alpha \vert \vert c_\alpha \vert, \quad \forall \alpha \in \mathbb{N}^d.
$$
Then $\mathcal{A} \in \mathscr{D}_\rho$ for every $\rho > 0$ and $\Vert \mathcal{A} \Vert_{\mathscr{D}_\rho} = 1$.
\end{example}

We also define, for any $t_0 > 0$, the Banach space
$$
\mathscr{B}_{\rho,\sigma}(t_0) :=  \mathcal{C}([-t_0,t_0]^2, \mathcal{L}(\ell_\rho(\mathbb{N}^d); \ell_{\rho-\sigma}(\mathbb{N}^d)).
$$

\begin{lemma}
\label{l:homogeneous_evolution}
Let $\rho  > 0$, $t_0 > 0$ and $t \mapsto \mathcal{A}(t) \in \mathcal{C}([-t_0,t_0];\mathscr{D}_\rho)$. Then, for any $0 < \sigma< \rho$, there exists $0 <t_1 \leq t_0 $ and $U \in \mathscr{B}_{\rho,\sigma}(t_1)$  such that, for every $-t_1 \leq s , t \leq t_1$, 
$$
\frac{\partial}{\partial t} U(t,s)  = \mathcal{A}(t) U(t,s) , \quad \frac{\partial}{\partial s} U(t,s)  =  - U(t,s) \mathcal{A}(s), \quad U(0,0) = \operatorname{Id}.
$$
\end{lemma}

\begin{proof}
We use the Picard iteration method. Take $0 < t_1 \leq t_0$ to be chosen later, and define the map 
$$
S :  \mathscr{B}_{\rho-\frac{\sigma}{2},\frac{\sigma}{2}}(t_1) \to \mathscr{B}_{\rho,\sigma}(t_1)
$$
by:
$$
S U (t,s)  = \operatorname{Id} + \int_s^t  \mathcal{A}(\tau) U(\tau,s)  d\tau.
$$
Given $U,V \in \mathscr{B}_{\rho-\frac{\sigma}{2},\frac{\sigma}{2}}(t_1)$, we have:
$$
\Vert SU - SV  \Vert_{\mathscr{B}_{\rho,\sigma}(t_1)} \leq  \frac{2t_0}{e\sigma} \Vert \mathcal{A} \Vert \Vert U - V \Vert_{\mathscr{B}_{\rho-\frac{\sigma}{2}, \frac{\sigma}{2}}(t_1)},
$$
where we denote $\Vert \cdot \Vert = \Vert \cdot \Vert_{\mathcal{C}([-t_0,t_0], \mathscr{D}_\rho)}$ for simplicity. Iterating this procedure, the operator $S^n$ can be viewed as a map
$$
S^n : \mathscr{B}_{\rho-\frac{n\sigma}{n+1},\frac{\sigma}{n+1}}(t_1) \to \mathscr{B}_{\rho,\sigma}(t_1),
$$
and, for any  $U,V \in \mathscr{B}_{\rho-\frac{n\sigma}{n+1},\frac{\sigma}{n+1}}(t_1)$,
\begin{align*}
\Vert S^n U - S^n V \Vert_{\mathscr{B}_{\rho,\sigma}(t_1)} \leq  \frac{(n+1)^{n}}{n!} \left( \frac{t_1 \Vert \mathcal{A} \Vert}{ e\sigma} \right)^n  \Vert U - V \Vert_{\mathscr{B}_{\rho-\frac{n\sigma}{n+1},\frac{\sigma}{n+1}}(t_1)}.
\end{align*}
Using Stirling's formula $n^n/(e^{n-1}n!) \leq 1$, we see that there exists $t_1 > 0$ small enough and a constant $\epsilon < 1$ such that
$$
\frac{(n+1)^{n}}{n!} \left( \frac{t_1 \Vert \mathcal{A} \Vert }{e \sigma} \right)^n  \leq \left( \frac{t_1 \Vert \mathcal{A} \Vert }{ \sigma} \right)^n =  \epsilon^n.
$$
Therefore, the sequence given by $U_n = S^n \operatorname{Id}$ satisfies
\begin{align*}
\Vert U_{n+1} - U_n \Vert_{\mathscr{B}_{\rho,\sigma}(t_1)} & = \Vert S^n S \operatorname{Id} - S^n \operatorname{Id} \Vert_{\mathscr{B}_{\rho,\sigma}(t_1)}  \\[0.2cm]
 & \leq  \delta^n \Vert S \operatorname{Id} - \operatorname{Id} \Vert_{\mathscr{B}_{\rho-\frac{n\sigma}{n+1},\frac{\sigma}{n+1}}(t_1)}  \\[0.2cm]
 & \leq \epsilon^n   \frac{ (n+1) \Vert \mathcal{A} \Vert}{e \sigma},
\end{align*}
and, similarly,
\begin{align*}
\Vert U_{n+m} - U_n \Vert_{\mathscr{B}_{\rho,\sigma}(t_1)} & \leq \sum_{j=1}^m \Vert U_{n+j} - U_{n+j-1} \Vert_{\mathscr{B}_{\rho,\sigma}(t_1)} \\[0.2cm]
 & \leq   \frac{ \epsilon^n \Vert \mathcal{A} \Vert}{e\sigma}  \sum_{j=1}^m \epsilon^{j-1}(n+j) \\[0.2cm]
 & \leq  \frac{ \epsilon^n \Vert \mathcal{A} \Vert}{e\sigma} \left( \frac{1}{(1- \epsilon)^2} + \frac{n}{1-\epsilon} \right).
\end{align*}
Thus, $(U_n)$ is a Cauchy sequence in $\mathscr{B}_{\rho,\sigma}(t_1)$, and then there exists a limit operator $U = \lim_n U_n \in \mathscr{B}_{\rho,\sigma}(t_1)$. Moreover, one can show by similar arguments, that $U$ is the unique solution to the integral equation
\begin{equation}
\label{e:fixed_point_equation}
U(t,s) = \operatorname{Id} + \int_s^t \mathcal{A}(\tau) U(\tau,s) d\tau.
\end{equation} 
In particular, $U(0,0) = \operatorname{Id}$. Deriving \eqref{e:fixed_point_equation} with respect to $t$ we obtain
$$
\frac{\partial}{\partial t} U(t,s)  = \mathcal{A}(t) U(t,s).
$$
Moreover, deriving \eqref{e:fixed_point_equation} with respect to $s$, we have
\begin{equation}
\label{e:derivative_integral_equation}
\frac{\partial}{\partial s} U(t,s) = - \mathcal{A}(s) + \int_s^t \mathcal{A}(\tau) \frac{\partial}{\partial s} U(\tau,s) d\tau.
\end{equation}
But notice, composing both sides of \eqref{e:fixed_point_equation} with $-\mathcal{A}(s)$ by the right, that \eqref{e:derivative_integral_equation} is also satisfied by $- U(t,s) \mathcal{A}(s)$. Since the solution to the integral equation \eqref{e:fixed_point_equation} is unique, we obtain that $\frac{\partial}{\partial s} U(t,s) = - U(t,s) \mathcal{A}(s)$, as we wanted.

\end{proof}

We next use Duhamel's principle to obtain the solution for the inhomogeneous problem. Let us consider the evolution problem:
$$
\frac{d}{dt} \vec{c}(t) = \mathcal{A}(t) \vec{c}(t) + f(t), \quad \vec{c}(0) = \vec{c}_0,
$$
where we assume that $\vec{c}_0 \in \ell_\rho(\mathbb{N}^d)$ and $f \in \mathcal{C}( [-t_0, t_0], \ell_{\rho - 2\sigma}(\mathbb{N}^d))$ for some $t_0 > 0$ and some $0 < \sigma < \rho/3$. Then, applying Lemma \ref{l:homogeneous_evolution}, we see that there exist $0 < t_1 \leq t_0$ and a solution $\vec{c}(t) \in \mathcal{C}([-t_1,t_1], \ell_{\rho - 3\sigma}(\mathbb{N}^d))$ such that
\begin{equation}
\label{e:inhomogeneous_problem}
\vec{c}(t) = U(t,0) \vec{c}_0 + \int_0^t U(t,r) f(r) dr.
\end{equation}
This can be used to compare the solutions between two evolution problems. 
\begin{prop}
\label{p:estimate_difference}
Let $\rho > 0$ and $\vec{u} = (u_\alpha) \in \ell_\rho(\mathbb{N}^d)$. Let $\mathcal{A}, \mathcal{B} \in \mathcal{C}([-t_0,t_0], \mathscr{D}_\rho)$ for some $t_0 > 0$. Consider the evolution problems:
\begin{align}
\label{e:evolution_1}
\frac{d}{dt} \vec{u}(t) & = \mathcal{A}(t) \vec{u}(t), \hspace*{1.8cm}    \vec{u}(0) = \vec{u}, \\[0.2cm]
\label{e:evolution_2}
\frac{d}{dt} \vec{v}(t) & = (\mathcal{A}(t) + \mathcal{B}(t)) \vec{v}(t),   \quad \vec{v}(0) = \vec{u}.
\end{align}
Then there exist $0 < \sigma < \rho/3$ and $0 < t_1 \leq t_0$ such that $\vec{w}(t) = \vec{v}(t)- \vec{v}(t) \in \mathcal{C}([-t_1,t_1],\ell_{\rho-3\sigma}(\mathbb{N}^d))$ satisfies:
\begin{equation}
\label{e:estimate_difference}
\sup_{ t \in [-t_1, t_1]} \Vert \vec{w}(t) \Vert_{\rho - 3 \sigma} \leq \frac{t_1}{e\sigma} \Vert V \Vert_{\mathscr{B}_{\rho-2\sigma,\sigma}(t_1)} \Vert \mathcal{B} \Vert \Vert U \Vert_{\mathscr{B}_{\rho,\sigma}(t_1)} \Vert \vec{u} \Vert_{\rho}. 
\end{equation}
\end{prop}

\begin{proof}
By Lemma \ref{l:homogeneous_evolution}, there exist $0 < \sigma < \rho/3$, a small time $t_1 > 0$, and propagators $U(t,s)$ and $V(t,s)$ to the evolution problems \eqref{e:evolution_1} and \eqref{e:evolution_2} respectively such that:
\begin{align*}
V \in \mathscr{B}_{\rho-2\sigma,\sigma}(t_1), \quad U \in \mathscr{B}_{\rho,\sigma}(t_1).
\end{align*}
Then, using \eqref{e:inhomogeneous_problem} for the evolution problem corresponding to the difference $\vec{w}(t) = \vec{v}(t) - \vec{u}(t)$:
$$
\frac{d}{dt} \vec{w}(t) = (\mathcal{A}(t) + \mathcal{B}(t)) \vec{w}(t) + \mathcal{B}(t) \vec{u}(t), \quad \vec{w}(0) = 0,
$$
we obtain, taking $f(t) =  \mathcal{B}(t) \vec{u}(t)$ as inhomogeneous term, that  $\vec{w}(t) \in \mathcal{C}([-t_1,t_1], \ell_{\rho - 3\sigma}(\mathbb{N}^d))$ satisfies
\begin{equation}
\label{e:duhamel_principle}
\vec{w}(t) = \int_0^t V(t,r) \mathcal{B}(r) U(r,0) \vec{u} \, dr,
\end{equation}
and then \eqref{e:estimate_difference} holds.

\end{proof}

\section{Averaging method for non-selfadjoint perturbations of the harmonic oscillator}
\label{a:averaging_method}

In this appendix, we recall some well established results describing some important features of the quantum and classic harmonic oscillator. Moreover, we give a brief proof of the construction of a quantum Birkhoff normal form for the perturbed harmonic oscillator 
$$
\widehat{\mathcal{P}}_\hbar := \widehat{H}_\hbar + \hbar \widehat{V}_\hbar + i\hbar \widehat{A}_\hbar.
$$
The presention is based on the previous works \cite{Ar_Mac18} and \cite{Ar_Riv18}.

\subsection{Classical averages and cohomological equations}
\label{e:averages_and_cohomological}
Given any function $a \in \mathcal{C}^\infty(\R^{2d})$,  we define its average $\mathcal{I}_a$ along the flow $\phi_t^H$ as
\begin{equation}
\label{average-definition1}
\mathcal{I}_a(z) := \lim_{T \to \infty} \frac{1}{T} \int_0^T a \circ \phi_t^H(z) dt = \lim_{T\to\infty}\frac{1}{T}\int_0^T a\circ\Phi_{z}(t\omega) dt,\quad z\in\R^{2d}.
\end{equation}
This limit is well defined; in fact it holds in the $\mathcal{C}^\infty(\R^{2d})$ topology. To see this, write $a \in \mathcal{C}^\infty(\R^{2d})$ as a Fourier series as follows. First define
\begin{equation}
\label{e:fourier_coefficients_oscillator}
a_k(z) := \int_{\mathbb{T}^d} a \circ \Phi_{z}(\tau) e^{-ik\cdot \tau}d\tau.
\end{equation}
Since, given any $z\in\R^{2d}$, the function $a \circ \Phi_{z}$ is smooth on $\T^d$  it follows that is Fourier coefficients $a_k$ decay faster than $\vert k \vert^{-N}$ in any compact set. In particular: 
\[
a =\sum_{k\in\Z^d}a_k,
\]
and notice that $a_k \circ \Phi_z(\tau) = a_k(z) \, e^{ik\cdot \tau}$. Hence the average $\mathcal{I}_a$ is given by
\begin{equation}
\label{e:different_formulas_average}
\mathcal{I}_a(z) = \frac{1}{(2\pi)^d} \sum_{k \in \Lambda_\omega} a_k(z) = \int_{\mathbb{T}_\omega} a \circ \Phi_z(\tau) \mu_\omega(d\tau),
\end{equation}
where $\mu_\omega$ denotes the Haar measure on the torus $\mathbb{T}_\omega$ (i.e. the uniform probability measure on $\mathbb{T}_\omega$ extended by zero to $\mathbb{T}^d$).

The energy hypersurface $ H^{-1}(E_0) \subset \R^{2d}$ is compact for every $E_0 \geq 0$ and, due to the complete integrability of the system, each of these hypersurfaces is foliated by Kronecker tori that are invariant by the flow $\phi_t^H$. 
Moreover, defining the submodule
\begin{equation}
\label{e:submodule}
\Lambda_\omega := \{k \in \mathbb{Z}^d \, : \, k \cdot \omega = 0 \},
\end{equation}
and the subtorus
\begin{equation}
\label{e:torus}
\mathbb{T}_\omega := \Lambda_\omega^\perp/(2\pi \mathbb{Z}^d \cap \Lambda_\omega^\perp) \subset \mathbb{T}^d,
\end{equation}
we have $\mathcal{T}_\omega(z_0) = \Phi_{z_0}(\mathbb{T}_\omega)$, and then $d_\omega = \dim \mathbb{T}_\omega = d - \operatorname{rk} \Lambda_\omega$. Kronecker's theorem states that the family of probability measures on $\mathbb{T}^d$ defined by
$$\frac{1}{T}\int_0^T \delta_{t\omega} \,dt$$ 
converges (in the weak-$\star$ topology)  to  the normalized Haar measure $\mu_\omega$ on the subtorus $\mathbb{T}_\omega\subset \mathbb{T}^d$. Moreover, the family of functions $\frac{1}{T}\int_0^T a \circ \phi_t^H dt$ converges to $\mathcal{I}_a$ in the $\mathcal{C}^\infty(\R^{2d})$ topology, and
\begin{equation}
\label{average-formula}
\mathcal{I}_a(z) =\int_{\mathbb{T}_\omega}a\circ\Phi_z(\tau)\mu_\omega(d\tau),
\end{equation}
and in particular, if $a \in \mathcal{C}^\infty(\R^{2d})$ then $\mathcal{I}_a\in\mathcal{C}^\infty(\mathbb{R}^{2d})$. In the case $d_{\omega}=1$ and $\omega=\omega_1(1,\ldots,1)$, the flow $\phi_t^H$ is $2\pi/\omega_1$-periodic. On the other hand, if $d_\omega = d$, then, for 
every $a\in\mathcal{C}^{\infty}(\mathbb{R}^{2d})$, there exists $\mathcal{G}_{\mathcal{I}_a}\in\mathcal{C}^\infty(\R^{d})$ such that
$$
\mathcal{I}_a(z) = \mathcal{G}_{\mathcal{I}_a}(H_1(z), \ldots, H_d(z)).
$$
In particular, for every $a$ and $b$ in $\mathcal{C}^{\infty}(\mathbb{R}^{2d})$, one has $\{\mathcal{I}_a,\mathcal{I}_b\}=0$ whenever $d_{\omega}=d$.

One of the technical difficulties that we will find in the process of averaging the perturbation $V + i A$ by the flow of the harmonic oscillator, will be to deal with cohomological equations \cite[Sec. 2.5]{Llav03} as the following:
\begin{equation}
\label{cohomological}
\{ H , f \} = g,
\end{equation}
where $g \in \mathcal{C}^\infty(\R^{2d})$ is a smooth function such that $\mathcal{I}_g = 0$. The goal is to solve this equation preserving the smooth properties of $g$. 

For any $f \in \mathcal{C}^\infty(\R^{2d})$, we can write $f\circ\Phi_z(\tau)$ as a Fourier series:
\begin{equation}
\label{e:harmonic_fourier-decomposition}
f\circ\Phi_z(\tau)=\sum_{k \in \mathbb{Z}^d}f_k(z) \frac{e^{ik\cdot\tau}}{(2\pi)^d}  ,\quad  f_k(z):=\int_{\mathbb{T}^d} f \circ\Phi_z(\tau) e^{-ik \cdot \tau} d\tau.
\end{equation}
Combining the fact that $f_k \circ \Phi_z(\tau) = f_k(z) e^{ik \cdot \tau}$ with (\ref{average-definition1}) gives:
\begin{align}
\label{e:different_forms_average}
\mathcal{I}_f(z) & = \frac{1}{(2\pi)^d}\sum_{k \in \Lambda_\omega}f_k(z) = \int_{\mathbb{T}_\omega}f\circ\Phi_z(\tau) \mu_\omega(d\tau).
\end{align}
Observe that if $f$ is a solution to (\ref{cohomological}), then so is $f + \lambda \mathcal{I}_f$ for any $\lambda \in \mathbb{R}$, since $\{ H, \mathcal{I}_f \} = 0$. Thus we can try to solve the equation for $\mathcal{I}_f = 0$ fixed, imposing
$$
f(z) = \frac{1}{(2\pi)^{2d}}\sum_{k \in \mathbb{Z}^d \setminus \Lambda_\omega} f_k(z).
$$
Writing down
$$
\{ H, f \}(z) = \frac{d}{dt} \left(f \circ \Phi_z(t \omega) \right) \vert_{t=0} = \frac{1}{(2\pi)^{d}} \sum_{k \in \mathbb{Z}^d \setminus \Lambda_\omega} ik \cdot \omega \, f_k(z) =  \frac{1}{(2\pi)^{d}} \sum_{k \in \mathbb{Z}^d \setminus \Lambda_\omega} g_k(z),
$$
we obtain that the solution of (\ref{cohomological}) is given (at least formally) by
\begin{equation}
\label{solution-cohomological}
f(z) = \frac{1}{(2\pi)^d} \sum_{k \in \mathbb{Z}^d \setminus \Lambda_\omega} \frac{1}{ik\cdot \omega} \, g_k(z).
\end{equation}
It is not difficult to see that, unless we impose some quantitive restriction on how fast $\vert k \cdot \omega \vert^{-1}$ can grow, the solutions given formally by (\ref{solution-cohomological}) may fail to be even distributions (see for instance \cite[Ex. 2.16.]{Llav03}). But if $\omega$ is partially Diophantine, 
and $g \in \mathcal{C}^\infty(\R^{2d})$ is such that $\langle g \rangle = 0$, then \textnormal{(\ref{solution-cohomological})} defines a smooth solution $f \in \mathcal{C}^\infty(\R^{2d})$ of \textnormal{(\ref{cohomological})}.

Finally, in the periodic case (assuming $\omega = (1, \ldots , 1)$ for simplicity), the solution to the cohomological equation \textnormal{(\ref{cohomological})} is given by the explicit formula
\begin{equation}
\label{e:solution_periodic_case}
f = \frac{-1}{2\pi} \int_0^{2\pi} \int_0^t  g \circ \phi_s^H  \, ds \, dt,
\end{equation}
provided that $\mathcal{I}_f = \mathcal{I}_g = 0$.

\subsection{Quantum Birkhoff normal form}
\label{s:normal_form}

This section is devoted to recall the semiclassical averaging method in the context of nonselfadjoint operators. Our aim is to average both the operators $\widehat{V}_\hbar$ and $\widehat{A}_\hbar$ by the quantum flow generated by $\widehat{H}_\hbar$ via conjugation through a suitable Fourier integral operator. 

Given $a \in \mathcal{C}^\infty(\R^{2d})$, we define the quantum average $\widehat{\mathcal{I}}_{ \Op_\hbar(a) }$ of the operator $\Op_\hbar(a)$ is given by:
 \begin{equation}
\label{quatum_average}
\widehat{\mathcal{I}}_{ \Op_\hbar(a) } := \lim_{T \to \infty} \frac{1}{T} \int_0^T e^{i\frac{t}{\hbar} \widehat{H}_\hbar} \Op_\hbar(a) e^{-i\frac{t}{\hbar} \widehat{H}_\hbar} \, dt.
\end{equation}
This limit is well defined due to Egorov's theorem and since the limit \eqref{average-definition1} takes place in the $\mathcal{C}^\infty(\R^{2d})$ topology.  Moreover, Egorov's theorem also implies that:
$$
\widehat{\mathcal{I}}_{ \Op_\hbar(a) } = \Op_\hbar(\mathcal{I}_a).
$$

The goal of this section is to prove the following:

\begin{prop}
\label{l:first_normal_form} For every $N \geq 1$, There exists a Fourier integral operator $\mathcal{F}_{N,\hbar}$ such that
\begin{equation}
\label{e:normal_form}
\widehat{\mathcal{P}}_{\hbar}^{\dagger} :=  \mathcal{F}_{N,\hbar} \big( \widehat{H}_\hbar + \hbar \widehat{V}_\hbar + i \hbar \widehat{A}_\hbar \big) \mathcal{F}^{-1}_{N,\hbar} = \widehat{H}_\hbar + \hbar \Op_\hbar(\mathcal{I}_{P_\hbar}) + \widehat{R}_{N,\hbar},
\end{equation}
where $P_\hbar = V + i A + O_{S^0(\R^{2d})}(\hbar)$ and $\Vert \widehat{R}_\hbar\Vert_{\mathcal{L}(L^2)} = O(\hbar^{N})$.
\medskip

Moreover, for every $a \in \mathcal{C}_c^\infty(\R^{2d})$,
\begin{equation}
\label{e:same_semiclassical_measure}
\Big \Vert \big( \mathcal{F}_{N,\hbar}^{-1} \big)^* \Op_\hbar(a) \mathcal{F}_{N,\hbar}^{-1} - \Op_\hbar(a) \Big \Vert_{\mathcal{L}(L^2)} = O(\hbar). 
\end{equation}
\end{prop}

We will require the following nonselfadjoint version of Egorov's theorem:

\begin{lemma}[Non-selfadjoint Egorov's theorem]
\label{Egorov}
Let $\mathcal{G}_\hbar(t)$ be a family of Fourier integral operators of the form
$$
\mathcal{G}_\hbar(t) := e^{\frac{it}{\hbar}(\widehat{G}_{1,\hbar} -i \hbar\widehat{G}_{2,\hbar}) }, \quad t\in \R,
$$
where $\widehat{G}_{j,\hbar} = \Op_\hbar(G_j)$ for $G_j \in S^0(\R^{2d};\R)$ and $j = 1,2$. Then, for every $t \in \R$ and every $a \in S^0(\R^{2d})$, the following holds:
$$
\mathcal{G}_\hbar(t) \Op_\hbar(a) \mathcal{G}_\hbar(-t) = \Op_\hbar(a\circ \phi_t^{G_1}  ) + O_t(\hbar),
$$
where $\phi_t^{G_1}$ is the Hamiltonian flow generated by $G_1$.
\end{lemma}

\begin{proof}
By \cite[Thm. III.1.3]{Eng00}, the family $\mathcal{G}_\hbar(t)$ defines a strongly continuous semigroup on $L^2(\R^d)$ such that
\begin{equation}
\label{e:bound_semigroup}
\Vert \mathcal{G}_\hbar(t) \Vert_{\mathcal{L}(L^2)} \leq e^{\vert t \vert \Vert \widehat{G}_{2,\hbar} \Vert_{\mathcal{L}(L^2)}}.
\end{equation}
Let $t \geq 0$. For every $r \in [0,t]$, we define
$$
a_r := a \circ \phi_{t-r}^{G_1} .
$$
By the product rule:
\begin{align*}
\frac{d}{dr} \big( \mathcal{G}_\hbar(r) \Op_\hbar(a_r) \mathcal{G}_\hbar(-r) \big) & \\[0.2cm]
 & \hspace*{-4cm} = \mathcal{G}_\hbar(r) \left( \frac{i}{\hbar}[ \widehat{G}_{1,\hbar} , \Op_\hbar(a_r)]  +  [ \widehat{G}_{2,\hbar} , \Op_\hbar(a_r)]  + \Op_\hbar ( \partial_r a_r) \right) \mathcal{G}_\hbar(-r).
\end{align*}
Using the symbolic calculus for Weyl pseudodifferential operators, we have
\begin{align*}
\frac{i}{\hbar}[ \widehat{G}_{j,\hbar} , \Op_\hbar(a_r)] & = \Op_\hbar( \{ G_j , a_r \}) + O(\hbar^2), \quad j=1,2.
\end{align*}
Moreover:
$$
\partial_r a_r = - \{ G_1 , a_r \}.
$$
These facts and (\ref{e:bound_semigroup}) give:
\begin{align*}
\mathcal{G}_\hbar(t) \Op_\hbar(a) \mathcal{G}_\hbar(-t) - \Op_\hbar(a\circ \phi_t^{G_1}) = \int_0^t \frac{d}{dr} \big( \mathcal{G}_\hbar(r) \Op_\hbar(a_r) \mathcal{G}_\hbar(-r) \big) dr = O_t(\hbar).
\end{align*}
Moreover, it can be shown that the remainder term $O_t(\hbar)$ is a semiclassical pseudodifferential operator with symbol in $S^0(\R^{2d})$.
\end{proof}

\begin{proof}[Proof of Proposition \ref{l:first_normal_form}]
We define
$$\widehat{F}_{\hbar} := \Op_\hbar(\hbar F_1+i \hbar F_2),$$
where $F_1$ and $F_2$ are two real valued symbols to be chosen below. We make the assumption 
that  $F_1,F_2 \in S^0(\R^{2d})$. For every $t$ in $[0,1]$, 
we set 
$$
\mathcal{F}_{1,\hbar}(t)=e^{\frac{i}{\hbar}t\widehat{F}_{\hbar}}.
$$
Denoting $\mathcal{F}_{1,\hbar} =\mathcal{F}_{1,\hbar}(1)$, we consider the  operator
$$\widehat{\mathcal{P}}_{1,\hbar}^\dagger:=\mathcal{F}_{1,\hbar}  \widehat{P}^\dagger_{\hbar}  \mathcal{F}_{1,\hbar}^{-1}$$
We define the symbols $F_1$ and $F_2$  to be the solutions to the cohomological equations (see Section \ref{e:averages_and_cohomological}):
\begin{align}
\label{cohomological_1}
\{ H, F_1 \} & = V - \mathcal{I}_V, \\[0.2cm]
\label{cohomological_2}
\{ H, F_2 \} & =  A - \mathcal{I}_A.
\end{align}
Observe that $F_j$ are real valued for $j=1,2$.  Using Taylor's theorem we write the  operator $\widehat{\mathcal{P}}_{1,\hbar}^\dagger$ as
\begin{align*}
\widehat{\mathcal{P}}_{1,\hbar}^\dagger = \mathcal{F}_{1,\hbar} \widehat{\mathcal{P}}_\hbar \mathcal{F}_{1,\hbar}^{-1} & = \widehat{H}_\hbar + \hbar \widehat{V}_\hbar + i \hbar  \widehat{A}_\hbar + \frac{i}{\hbar} [ \widehat{F}_{\hbar} , \widehat{H}_\hbar ] \\[0.2cm]
 & \quad  + \frac{i}{\hbar} \int_0^1 \mathcal{F}_{1,\hbar}(t)  [ \widehat{F}_{\hbar}, \hbar \widehat{V}_\hbar + i\hbar \widehat{A}_\hbar ]\mathcal{F}_{\hbar}(t)^{-1} dt \\[0.2cm]
 & \quad + \left( \frac{i}{\hbar} \right)^2 \int_0^1(1-t) \mathcal{F}_{1,\hbar}(t)  [\widehat{F}_\hbar,  [\widehat{F}_\hbar, \widehat{H}_\hbar ] ] \mathcal{F}_{1,\hbar}(t)^{-1}dt.
\end{align*}
By the symbolic calculus for Weyl pseudodifferential operators,
\begin{align*}
 \frac{i}{\hbar} [\widehat{F}_{j,\hbar} , \widehat{H}_\hbar ] &  =  \Op_h( \{ F_j, H \}), \quad j = 1,2.
  \end{align*}
Since $F_1$ and $F_2$ solve cohomological equations (\ref{cohomological_1}) and (\ref{cohomological_2}),  we obtain
$$
\widehat{\mathcal{P}}_{1,\hbar}^\dagger = \widehat{H}_\hbar + \hbar \Op_\hbar(\mathcal{I}_V + i \mathcal{I}_A) + \widehat{R}_{1,\hbar},
$$
where
\begin{equation}
\label{reminder}
\widehat{R}_{1,\hbar} =  \frac{i}{\hbar} \int_0^1 \mathcal{F}_{1,\hbar}(t) [\widehat{F}_\hbar, \widehat{K}_\hbar(t)] \mathcal{F}_{1,\hbar}(t)^{-1} dt,
\end{equation}
and
$$
\widehat{K}_\hbar(t) = t(\hbar \widehat{V}_\hbar + i\hbar \widehat{A}_\hbar) + (1-t) \hbar \Op_\hbar(\mathcal{I}_V + i \mathcal{I}_A), \quad t \in [0,1].
$$
Using the pseudodifferential calculus one more time, we see that $\Vert \widehat{R}_{1,\hbar} \Vert_{\mathcal{L}(L^2)} = O(\hbar^2)$.  Iterating this method up to order $N$, we obtain the normal form \eqref{e:normal_form}.
\medskip

Finally, \eqref{e:same_semiclassical_measure} follows by Lemma \ref{Egorov}.

\end{proof}

\bibliography{Referencias}

\begin{thebibliography}{10}

\bibitem{Anantharaman10}
N.~Anantharaman.
\newblock Spectral deviations for the damped wave equation.
\newblock {\em Geom. Funct. Anal.}, 20(3):593--626, 2010.

\bibitem{AnantharamanLeautaud}
N.~Anantharaman and M.~L\'eautaud.
\newblock Sharp polynomial decay rates for the damped wave equation on the
  torus.
\newblock {\em Anal. PDE}, 7(1):159--214, 2014.
\newblock With an appendix by St\'ephane Nonnenmacher.

\bibitem{Ar_Mac18}
V.~Arnaiz and F.~Macià.
\newblock Localization and delocalization of eigenmodes of harmonic
  oscillators.
\newblock {\em Preprint, arXiv:2010.13436}, 2020.

\bibitem{Ar_Riv18}
V.~Arnaiz and G.~Rivi\`ere.
\newblock Semiclassical asymptotics for nonselfadjoint harmonic oscillators.
\newblock {\em Pure Appl. Anal.}, 2(2):427--445, 2020.

\bibitem{AschLebeau}
M.~Asch and G.~Lebeau.
\newblock The spectrum of the damped wave operator for a bounded domain in
  {${\bf R}^2$}.
\newblock {\em Experiment. Math.}, 12(2):227--241, 2003.

\bibitem{Bargmann61}
V.~Bargmann.
\newblock On a {H}ilbert space of analytic functions and an associated integral
  transform.
\newblock {\em Comm. Pure Appl. Math.}, 14:187--214, 1961.

\bibitem{Bargmann67}
V.~Bargmann.
\newblock On a {H}ilbert space of analytic functions and an associated integral
  transform. {P}art {II}. {A} family of related function spaces. {A}pplication
  to distribution theory.
\newblock {\em Comm. Pure Appl. Math.}, 20:1--101, 1967.

\bibitem{Berezin_Shubin91}
F.~A. Berezin and M.~A. Shubin.
\newblock {\em The {S}chr\"{o}dinger equation}, volume~66 of {\em Mathematics
  and its Applications (Soviet Series)}.
\newblock Kluwer Academic Publishers Group, Dordrecht, 1991.
\newblock Translated from the 1983 Russian edition by Yu. Rajabov, D. A.
  Le\u{\i}tes and N. A. Sakharova and revised by Shubin, With contributions by
  G. L. Litvinov and Le\u{\i}tes.

\bibitem{Borthwick_Graffi05}
D.~Borthwick and S.~Graffi.
\newblock A local quantum version of the {K}olmogorov theorem.
\newblock {\em Comm. Math. Phys.}, 257(2):499--514, 2005.

\bibitem{BurqGerard18}
N.~Burq and P.~G\'erard.
\newblock Stabilisation of wave equations on the torus with rough dampings.
\newblock 2018.
\newblock Preprint arXiv:1801.00983.

\bibitem{BurqHitrik}
N.~Burq and M.~Hitrik.
\newblock Energy decay for damped wave equations on partially rectangular
  domains.
\newblock {\em Math. Res. Lett.}, 14(1):35--47, 2007.

\bibitem{Charles08}
L.~Charles and S.~V\~{u}~Ng\d{o}c.
\newblock Spectral asymptotics via the semiclassical {B}irkhoff normal form.
\newblock {\em Duke Math. J.}, 143(3):463--511, 2008.

\bibitem{Christianson07}
H.~Christianson.
\newblock Semiclassical non-concentration near hyperbolic orbits.
\newblock {\em J. Funct. Anal.}, 246(2):145--195, 2007.

\bibitem{ChristiansonSchenckVasyWunsch}
H.~Christianson, E.~Schenck, A.~Vasy, and J.~Wunsch.
\newblock From resolvent estimates to damped waves.
\newblock {\em J. Anal. Math.}, 122:143--162, 2014.

\bibitem{Hitrik19}
L.~A. Coburn, M.~Hitrik, and J.~Sj\"{o}strand.
\newblock Positivity, complex {FIO}s, and {T}oeplitz operators.
\newblock {\em Pure Appl. Anal.}, 1(3):327--357, 2019.

\bibitem{CombescureRobertWP}
M.~Combescure and D.~Robert.
\newblock Semiclassical spreading of quantum wave packets and applications near
  unstable fixed points of the classical flow.
\newblock {\em Asymptot. Anal.}, 14(4):377--404, 1997.

\bibitem{Robert12}
M.~Combescure and D.~Robert.
\newblock {\em Coherent states and applications in mathematical physics}.
\newblock Theoretical and Mathematical Physics. Springer, Dordrecht, 2012.

\bibitem{Davies02}
E.~B. Davies.
\newblock Non-self-adjoint differential operators.
\newblock {\em Bull. London Math. Soc.}, 34(5):513--532, 2002.

\bibitem{DBievre92}
S.~De~Bi\`evre.
\newblock Oscillator eigenstates concentrated on classical trajectories.
\newblock {\em J. Phys. A}, 25(11):3399--3418, 1992.

\bibitem{DBievre93}
S.~De~Bi\`evre, J.-C. Houard, and M.~Irac-Astaud.
\newblock Wave packets localized on closed classical trajectories.
\newblock In {\em Differential equations with applications to mathematical
  physics}, volume 192 of {\em Math. Sci. Engrg.}, pages 25--32. Academic
  Press, Boston, MA, 1993.

\bibitem{Llav03}
R.~de~la Llave.
\newblock A tutorial on {KAM} theory.
\newblock In {\em Smooth ergodic theory and its applications ({S}eattle, {WA},
  1999)}, volume~69 of {\em Proc. Sympos. Pure Math.}, pages 175--292. Amer.
  Math. Soc., Providence, RI, 2001.

\bibitem{Dencker04}
N.~Dencker, J.~Sj\"{o}strand, and M.~Zworski.
\newblock Pseudospectra of semiclassical (pseudo-) differential operators.
\newblock {\em Comm. Pure Appl. Math.}, 57(3):384--415, 2004.

\bibitem{Dietert17}
H.~Dietert, J.~Keller, and S.~Troppmann.
\newblock An invariant class of wave packets for the {W}igner transform.
\newblock {\em J. Math. Anal. Appl.}, 450(2):1317--1332, 2017.

\bibitem{EswaNon17}
S.~Eswarathasan and S.~Nonnenmacher.
\newblock Strong scarring of logarithmic quasimodes.
\newblock {\em Annales de l'institut Fourier, 67(6):2307--2347}, 2017.

\bibitem{Eswarathasan18}
S.~Eswarathasan and L.~Silberman.
\newblock Scarring of quasimodes on hyperbolic manifolds.
\newblock {\em Nonlinearity}, 31(1):1--29, 2018.

\bibitem{Schubert15}
E-M. Graefe, H.~J. Korsch, A.~Rush, and R.~Schubert.
\newblock Classical and quantum dynamics in the (non-{H}ermitian) {S}wanson
  oscillator.
\newblock {\em J. Phys. A}, 48(5):055301, 16, 2015.

\bibitem{Schubert12}
E-M. Graefe and R.~Schubert.
\newblock Complexified coherent states and quantum evolution with
  non-{H}ermitian {H}amiltonians.
\newblock {\em J. Phys. A}, 45(24):244033, 15, 2012.

\bibitem{Schubert11}
E-M. Graefe and R.~Schübert.
\newblock Wave packet evolution in non-{H}ermitian quantum systems.
\newblock {\em Physical Review A}, 83, 2011.

\bibitem{Hagedorn85}
G.~A. Hagedorn.
\newblock Semiclassical quantum mechanics. {IV}. {L}arge order asymptotics and
  more general states in more than one dimension.
\newblock {\em Ann. Inst. H. Poincar\'{e} Phys. Th\'{e}or.}, 42(4):363--374,
  1985.

\bibitem{Hagedorn98}
G.~A. Hagedorn.
\newblock Raising and lowering operators for semiclassical wave packets.
\newblock {\em Ann. Physics}, 269(1):77--104, 1998.

\bibitem{Hitrik02}
M.~Hitrik.
\newblock Eigenfrequencies for damped wave equations on {Z}oll manifolds.
\newblock {\em Asymptot. Anal.}, 31(3-4):265--277, 2002.

\bibitem{HitrikSjostrand04}
M.~Hitrik and J.~Sj\"ostrand.
\newblock Non-selfadjoint perturbations of selfadjoint operators in 2
  dimensions. {I}.
\newblock {\em Ann. Henri Poincar\'e}, 5(1):1--73, 2004.

\bibitem{HitrikSjostrand05}
M.~Hitrik and J.~Sj\"ostrand.
\newblock Nonselfadjoint perturbations of selfadjoint operators in two
  dimensions. {II}. {V}anishing averages.
\newblock {\em Comm. Partial Differential Equations}, 30(7-9):1065--1106, 2005.

\bibitem{HitrikSjostrand08}
M.~Hitrik and J.~Sj\"ostrand.
\newblock Non-selfadjoint perturbations of selfadjoint operators in two
  dimensions. {III}a. {O}ne branching point.
\newblock {\em Canad. J. Math.}, 60(3):572--657, 2008.

\bibitem{HitrikSjostrand08b}
M.~Hitrik and J.~Sj\"ostrand.
\newblock Rational invariant tori, phase space tunneling, and spectra for
  non-selfadjoint operators in dimension 2.
\newblock {\em Ann. Sci. \'Ec. Norm. Sup\'er. (4)}, 41(4):511--571, 2008.

\bibitem{HitrikSjostrand12}
M.~Hitrik and J.~Sj\"ostrand.
\newblock Diophantine tori and {W}eyl laws for non-selfadjoint operators in
  dimension two.
\newblock {\em Comm. Math. Phys.}, 314(2):373--417, 2012.

\bibitem{HitrikSjostrand18}
M.~Hitrik and J.~Sj\"ostrand.
\newblock Rational invariant tori and band edge spectra for non-selfadjoint
  operators.
\newblock {\em J. Eur. Math. Soc. (JEMS)}, 20(2):391--457, 2018.

\bibitem{HitrikSjostrandVuNgoc07}
M.~Hitrik, J.~Sj\"ostrand, and S.~V\~u Ngoc.
\newblock Diophantine tori and spectral asymptotics for nonselfadjoint
  operators.
\newblock {\em Amer. J. Math.}, 129(1):105--182, 2007.

\bibitem{Viola13}
M.~Hitrik, J.~Sj\"{o}strand, and J.~Viola.
\newblock Resolvent estimates for elliptic quadratic differential operators.
\newblock {\em Anal. PDE}, 6(1):181--196, 2013.

\bibitem{Hor94IV}
L.~H\"{o}rmander.
\newblock {\em The analysis of linear partial differential operators. {IV}}.
\newblock Classics in Mathematics. Springer-Verlag, Berlin, 2009.
\newblock Fourier integral operators, Reprint of the 1994 edition.

\bibitem{Jin17}
L.~Jin.
\newblock Damped wave equations on compact hyperbolic surfaces.
\newblock 2017.
\newblock Preprint arXiv:1712.02692.

\bibitem{Eng00}
R.~Nagel~(auth.) K.~J.~Engel.
\newblock {\em One-Parameter Semigroups for Linear Evolution Equations}.

\bibitem{Keeler20}
B.~Keeler and P~Kleinhenz.
\newblock Exponential energy decay of solutions to the wave equation with
  anisotropic damping.
\newblock {\em Preprint, arXiv:2009.10832}, 2020.

\bibitem{Siegl18}
D.~Krej\v{c}i\v{r}\'{\i}k and P.~Siegl.
\newblock Pseudomodes for {S}chr\"{o}dinger operators with complex potentials.
\newblock {\em J. Funct. Anal.}, 276(9):2856--2900, 2019.

\bibitem{Viola15}
D.~Krej\v{c}i\v{r}\'{\i}k, P.~Siegl, M.~Tater, and J.~Viola.
\newblock Pseudospectra in non-{H}ermitian quantum mechanics.
\newblock {\em J. Math. Phys.}, 56(10):103513, 32, 2015.

\bibitem{Schubert18}
C.~Lasser, R.~Schubert, and S.~Troppmann.
\newblock Non-{H}ermitian propagation of {H}agedorn wavepackets.
\newblock {\em J. Math. Phys.}, 59(8):082102, 35, 2018.

\bibitem{Lebeau96}
G.~Lebeau.
\newblock {\'E}quation des ondes amorties.
\newblock In {\em Algebraic and geometric methods in mathematical physics
  ({K}aciveli, 1993)}, volume~19 of {\em Math. Phys. Stud.}, pages 73--109.
  Kluwer Acad. Publ., Dordrecht, 1996.

\bibitem{Mac_Rive16}
F.~Maci\`a and G.~Rivi\`ere.
\newblock Concentration and non-concentration for the {S}chr\"{o}dinger
  evolution on {Z}oll manifolds.
\newblock {\em Comm. Math. Phys.}, 345(3):1019--1054, 2016.

\bibitem{Mac_Riv18}
F.~Maci\`a and G.~Rivi\`ere.
\newblock Two-microlocal regularity of quasimodes on the torus.
\newblock {\em Anal. PDE}, 11(8):2111--2136, 2018.

\bibitem{Mac_Riv19}
F.~Maci\`a and G.~Rivi\`ere.
\newblock Observability and quantum limits for the {S}chr\"{o}dinger equation
  on {$\Bbb{S}^d$}.
\newblock In {\em Probabilistic methods in geometry, topology and spectral
  theory}, volume 739 of {\em Contemp. Math.}, pages 139--153. Amer. Math.
  Soc., Providence, RI, 2019.

\bibitem{Markus88}
A.~S. Markus.
\newblock {\em Introduction to the spectral theory of polynomial operator
  pencils}, volume~71 of {\em Translations of Mathematical Monographs}.
\newblock American Mathematical Society, Providence, RI, 1988.
\newblock Translated from the Russian by H. H. McFaden, Translation edited by
  Ben Silver, With an appendix by M. V. Keldysh.

\bibitem{Markus_Matsaev80}
A.~S. Markus and V.~I. Matsaev.
\newblock Asymptotic behavior of the spectrum of close-to-normal operators.
\newblock {\em Funktsional. Anal. i Prilozhen.}, 13(3):93--94, 1979.

\bibitem{Nonnenmacher11}
S.~Nonnenmacher.
\newblock Spectral theory of damped quantum chaotic systems.
\newblock 2011.
\newblock Preprint arXiv:1109.0930.

\bibitem{Ojeda_Villegas10}
D.~Ojeda-Valencia and C.~Villegas-Blas.
\newblock On limiting eigenvalue distributions theorems in semiclassical
  analysis.
\newblock {\em Spectral Analysis of Quantum Hamiltonians: Spectral Days.
  Springer}, 2010.

\bibitem{Paul93}
T.~Paul and A.~Uribe.
\newblock A construction of quasi-modes using coherent states.
\newblock {\em Ann. Inst. H. Poincar\'{e} Phys. Th\'{e}or.}, 59(4):357--381,
  1993.

\bibitem{Pravda_Starov2004}
K.~Pravda-Starov.
\newblock A general result about the pseudo-spectrum of {S}chr\"{o}dinger
  operators.
\newblock {\em Proc. R. Soc. Lond. Ser. A Math. Phys. Eng. Sci.},
  460(2042):471--477, 2004.

\bibitem{Starov07}
K.~Pravda-Starov.
\newblock Sur le pseudo-spectre de certaines classes d'op\'{e}rateurs
  pseudo-diff\'{e}rentiels non auto-adjoints.
\newblock In {\em S\'{e}minaire: \'{E}quations aux {D}\'{e}riv\'{e}es
  {P}artielles. 2006--2007}, S\'{e}min. \'{E}qu. D\'{e}riv. Partielles, pages
  Exp. No. XV, 35. \'{E}cole Polytech., Palaiseau, 2007.

\bibitem{Starov08}
K.~Pravda-Starov.
\newblock On the pseudospectrum of elliptic quadratic differential operators.
\newblock {\em Duke Math. J.}, 145(2):249--279, 2008.

\bibitem{Starov08a}
K.~Pravda-Starov.
\newblock Pseudo-spectrum for a class of semi-classical operators.
\newblock {\em Bull. Soc. Math. France}, 136(3):329--372, 2008.

\bibitem{Starov18}
K.~Pravda-Starov.
\newblock Generalized {M}ehler formula for time-dependent non-selfadjoint
  quadratic operators and propagation of singularities.
\newblock {\em Math. Ann.}, 372(3-4):1335--1382, 2018.

\bibitem{RauchTaylor}
J.~Rauch and M.~Taylor.
\newblock Decay of solutions to nondissipative hyperbolic systems on compact
  manifolds.
\newblock {\em Comm. Pure Appl. Math.}, 28(4):501--523, 1975.

\bibitem{Riviere14}
G.~Rivi\`ere.
\newblock Eigenmodes of the damped wave equation and small hyperbolic subsets.
\newblock {\em Ann. Inst. Fourier (Grenoble)}, 64(3):1229--1267, 2014.
\newblock With an appendix by St\'ephane Nonnenmacher and Rivi\`ere.

\bibitem{RobertBook}
D.~Robert.
\newblock {\em Autour de l'approximation semi-classique}, volume~68 of {\em
  Progress in Mathematics}.
\newblock Birkh\"auser Boston Inc., Boston, MA, 1987.

\bibitem{Schenck10}
E.~Schenck.
\newblock Energy decay for the damped wave equation under a pressure condition.
\newblock {\em Comm. Math. Phys.}, 300(2):375--410, 2010.

\bibitem{Schenck11}
E.~Schenck.
\newblock Exponential stabilization without geometric control.
\newblock {\em Math. Res. Lett.}, 18(2):379--388, 2011.

\bibitem{Villegas18}
D.~Sher, A.~Uribe, and C.~Villegas-Blas.
\newblock On the pseudospectra of {S}chrödinger operators on {Z}oll manifolds.
\newblock {\em Preprint, arXiv:1812.01769v1}, 2018.

\bibitem{Sjostrand00}
J.~Sj\"ostrand.
\newblock Asymptotic distribution of eigenfrequencies for damped wave
  equations.
\newblock {\em Publ. Res. Inst. Math. Sci.}, 36(5):573--611, 2000.

\bibitem{Sjostrand_survey}
J.~Sj\"ostrand.
\newblock Lecture notes : Spectral properties of non-self-adjoint operators.
\newblock {\em Journ\'ees \'equations aux d\'eriv\'ees partielles}, 2009.

\bibitem{Sjostrand09}
J.~Sj\"{o}strand.
\newblock Resolvent estimates for non-selfadjoint operators via semigroups.
\newblock In {\em Around the research of {V}ladimir {M}az'ya. {III}}, volume~13
  of {\em Int. Math. Ser. (N. Y.)}, pages 359--384. Springer, New York, 2010.

\bibitem{Troppmann2017}
S.~Troppmann.
\newblock Non-{H}ermitian {S}chrödinger dynamics with {H}agedorn’s wave
  packets.
\newblock {\em Ph.D. Dissertation, Fakultät für Mathematik Lehrstuhl für
  wissenschaftliches Rechnen, Fakultät für Mathematik der Technischen
  Universität München}, 2017.

\bibitem{Zworski01}
M.~Zworski.
\newblock A remark on a paper of {E}. {B} {D}avies: ``{S}emi-classical states
  for non-self-adjoint {S}chr\"{o}dinger operators'' [{C}omm. {M}ath. {P}hys.
  {\bf 200} (1999), no. 1, 35--41; {MR}1671904 (99m:34197)].
\newblock {\em Proc. Amer. Math. Soc.}, 129(10):2955--2957, 2001.

\bibitem{Zwobook}
M.~Zworski.
\newblock {\em Semiclassical analysis}, volume 138 of {\em Graduate Studies in
  Mathematics}.
\newblock American Mathematical Society, Providence, RI, 2012.

\end{thebibliography}
\bibliographystyle{plain}

\end{document}